\documentclass[11pt, letterpaper]{amsart}
\usepackage{amssymb,latexsym, xcolor %,a4wide
}
\usepackage[dvips, hcentering,
includehead,width=14.2cm,
top=0.5cm,
height=21.6cm,
paperheight=23.8cm,paperwidth=17cm
]{geometry}
\usepackage{hyperref}
\usepackage{enumitem}
\usepackage{graphicx}
\usepackage{fancybox}

\usepackage{tikz-cd}

\usepackage{ccaption}
\changecaptionwidth
\captionwidth{5.5in}
\newcommand{\hide}[1]{}

\newtheorem{theorem}{Theorem}[section]
\newtheorem{proposition}[theorem]{Proposition}

\newtheorem{corollary}[theorem]{Corollary}
\newtheorem{lemma}[theorem]{Lemma}

\theoremstyle{definition}

\newtheorem{remark}[theorem]{Remark}

\newtheorem{example}[theorem]{Example}
\newtheorem{definition}[theorem]{Definition}

\numberwithin{equation}{section}

\newcommand{\red}[1]{{\color{red} #1}}
\newcommand{\blue}[1]{{\color{blue} #1}}

\def\ZZ{{\mathbb Z}}

\def\xx{{\mathbf{x}}}

\def\Quiv {{\mathbf{Quiv}}}

\newcommand{\T}[2]{\mu[#1](#2)} %command for sequences of mutations.

\newcommand{\mutation}[1]{\stackrel{#1}{\rule[.5ex]{2em}{0.5pt}}}
\newcommand{\shortmutation}[1]{\stackrel{#1}{\rule[.5ex]{1.2em}{0.5pt}}}
\newcommand{\vmutation}[1]{  {\scriptstyle #1}{\Bigl|\Bigr.}  }

\newcommand{\points}{\!\dasharrow\!}

\def\nsquare{\scalebox{.85}{$\scriptscriptstyle\Box$}}
\def\ssquare{\scalebox{.75}{$\scriptscriptstyle\Box$}}

\begin{document}

\title{Long mutation cycles}
\author{Sergey Fomin}
\address{Department of Mathematics, University of Michigan,
%530 Church Street,
Ann Arbor, MI 48109, USA}
\email{fomin@umich.edu}

\author{Scott Neville}
\address{Department of Mathematics, University of Michigan, Ann Arbor, MI 48109, USA}
\email{nevilles@umich.edu}

\newcommand{\h}{u} %for chebyshev polys. Probably regex these out.
\newcommand{\oboxed}[1]{\ovalbox{$#1$}}
\cornersize*{12pt}
\setlength{\fboxsep}{2pt}

\date{April 22, 2023. Revised July 24, 2024 and July 14, 2025.}

\thanks{Partially supported by NSF grants DMS-2054231 and 2348501 (S.~F.\ and~S.~N.),
DMS-1840234 (S.~N.), 
and by BSF grant 2022157 and a Simons Fellowship (S.~F.).
}



\begin{abstract}
A mutation cycle is a cycle in a graph whose vertices are labeled by~the quivers
in a given mutation  class and whose edges correspond to single mutations. 
For any fixed $n\ge 4$, we describe arbitrarily long mutation cycles involving $n$-vertex quivers.
Each of these mutation cycles allows for an arbitrary choice of $\binom{n}{2}$ positive integer parameters. 
None of the mutation cycles we construct can be paved by short mutation cycles. 
\end{abstract}

%\ \vspace{-.2in}

\maketitle

%\vspace{-.2in}

\section{Introduction}

Quiver mutations play fundamental role in the theory of cluster algebras and its numerous applications, 
see, e.g., \cite{FWZ} and references therein. 
In this paper, we focus on a particular aspect of the combinatorics of quiver mutations, namely 
the study of mutation cycles.
These are cycles in the \emph{mutation graph} of a mutation equivalence class,
the graph whose vertices correspond to the quivers in the class
and whose edges correspond to single mutations that relate quivers to each other. 
A~bit more precisely, a \emph{mutation cycle} of length~$N$ rooted at a quiver~$Q$
is a sequence of mutation steps that, when successively applied to~$Q$, ends up returning to~$Q$: 
\begin{equation*}
Q=Q^{(0)} \mutation{i_1}
Q^{(1)} \mutation{i_2}
\cdots
\mutation{i_{N-1}} Q^{(N-1)} \mutation{i_N}
Q^{(N)}=Q; 
\end{equation*}
here we require  $i_1\neq i_2\neq \cdots\neq i_{N-1}\neq i_N\neq i_1$ to ensure that we never 
apply the same mutation twice in a row. 

Not much is known about mutation cycles. Existing constructions involve 
\begin{itemize}[leftmargin=.15in]
\item 
quivers of finite mutation type, e.g., those associated with triangulated surfaces~\cite{MR2448067},
where elements of the mapping class group give rise to mutation cycles; 
\item
quivers related to square products of Dynkin diagrams~\cite{keller-annals, MR2767952} and more generally, 
flip moves in plabic graphs \cite{balitskiy-wellman}; 
\item
mutation-periodic quivers of  A.~Fordy and B.~R.~Marsh~\cite{Fordy-Marsh}. 
\end{itemize}
All of these mutation cycles rely on small arrow multiplicities and/or particular symmetries of the quivers involved. 
Perhaps more importantly, the only way to get a very long mutation cycle using either of these constructions
(provided we treat quivers up to isomorphism) 
is to use quivers with a lot of vertices.

In this paper, we considerably expand the zoo of mutation cycles by constructing,
for any fixed~$n\ge 4$, 
arbitrarily long mutation cycles involving $n$-vertex quivers. \linebreak[3]
Our main result (cf.\ Theorem~\ref{thm:Summary} below)
is a construction of a family of $n$-vertex quivers 
each of which lies on a mutation cycle of length~${n+4k}$. 
All quivers along any such mutation cycle are distinct. 
Moreover, none of these cycles can be paved by mutation cycles of length~${\le 4k}$. 
Thus, if $n$ is fixed and $k$ is large, say, $n=5$ and $k=100$,
then our construction produces quivers on 5~vertices that lie on mutation cycles of length~405 that cannot be paved by
cycles shorter than~400. 

Furthermore, there are lots of such mutation cycles! 
For fixed $n\ge 4$ and $k\ge 1$, our construction depends on an arbitrary choice of 
$\binom{n}{2}$ parameters $q_{ij}\in \ZZ_{\ge2}$; here $1\le i<j\le n$.
Different choices of parameters produce different mutation cycles. 
The parametrization $(q_{ij})\mapsto Q$
(as well as its inverse) is given by polynomials over~$\ZZ$. 

\medskip

We next state our main result, which combines 
Theorems~\ref{thm:GeneralSemiCycles}, \ref{thm:GeneralSemiCycleDistinct}, and~\ref{thm:primitive}. 

\begin{theorem}%[Main Theorem]
\label{thm:Summary}
Let $n\ge 4$ and $k\ge 1$. 
Choose integers $q_{ij} \geq 2$ for all pairs $1 \leq i < j \leq n$. 
Define the sequence $p_0,p_1,p_2,\dots$ by the recurrence 
\begin{equation}
%\label{eq:chebyshev-rec}
p_{j+1}=q_{12} p_j -p_{j-1}, 
\end{equation}
with initial values $p_0=1$, $p_1=q_{12}$. 
Let $Q$ be the quiver on the vertex set $\{1,\dots,n\}$ whose exchange matrix $B(Q)=(b_{ij})$ is defined as follows. 
For $1\le i<j\le n$, set 
\[
b_{ij} = 
\begin{cases}
-p_{2k-2} q_{1j} - p_{2k-1} q_{2j} & \text{if $i=1$ and $3\le j \le n-1$}; \\ 
p_{2k-1} q_{1j} + p_{2k} q_{2j} & \text{if $i=2$ and $3\le j \le n-1$}; \\
q_{ij} & \text{otherwise}; 
\end{cases}
\]
also set $b_{ii}=0$ and $b_{ji}=-b_{ij}$. 
Then 
\begin{itemize}[leftmargin=.2in]
\item 
applying the following sequence of mutations to~$Q$ recovers the same quiver~$Q$: 
\begin{equation}
\label{eq:main-cycle}
n, \underbrace{1,2,\dots ,1,2}_{\textup{$k$ times}}, n-1, n-2, \dots, 2, 1, \underbrace{2,1,\dots ,2,1}_{\textup{$k$ times}}; 
%n-1, n-2, \dots, 2, 1, \underbrace{2,1,\cdots ,2,1}_{\textup{$k$ times}}, n, \underbrace{1,2,\cdots ,1,2}_{\textup{$k$ times}}
\end{equation}
\item
all $n+4k$ quivers lying on this mutation cycle are distinct; 
\item
this mutation cycle cannot be paved by mutation cycles of length $\le 4k$. 
\end{itemize}
\end{theorem}

Figure~\ref{fig:generic-8-cycle} shows the simplest instance of this construction 
(for $n=4$ and $k=1$), a mutation cycle of length~8 involving 4-vertex quivers.
This mutation cycle cannot be paved by shorter cycles. 

\begin{figure}[ht]
%\vspace{-.1in}
\begin{equation*}
\hspace{7pt} 
\begin{array}{ccccccc}
\hspace{-17pt} Q=\!\!\begin{tikzcd}[arrows={-stealth}, sep=6em]
  1  \arrow[r,  "a"]
  & 2  \arrow[d, swap, "ba^2-b+ad"]
  \\
   4 \arrow[ur, "e", swap, bend left=12, near start, outer sep=-1.8, stealth-] \arrow[u, "f", swap, stealth-] \arrow[r, "c", swap, stealth-]  
   & 3  \arrow[ul, "d+ab", swap, bend left=12, very near end, outer sep=-1.8]
\end{tikzcd}
 & \mutation{4}
 &
\begin{tikzcd}[arrows={-stealth}, sep=6em]
  1  \arrow[r,  "a"]
  & 2  \arrow[d, swap, "ba^2-b+ad"]
  \\
   4 \arrow[ur, "e", swap, near start, bend left=12, outer sep=-1.8] \arrow[u, swap, "f"] \arrow[r, "c", swap]  
   & 3  \arrow[ul, "d+ab", swap, very near end, bend left=12, outer sep=-1.8]
\end{tikzcd}
 & \mutation{1}
 & \begin{tikzcd}[arrows={-stealth}, sep=6em]
  1  \arrow[r,  "a", stealth-]
  & 2  \arrow[d, "b", swap, stealth-]
  \\
   4 \arrow[ur, "e+af", swap, near start, outer sep=-1.8] \arrow[u, swap, "f", stealth-] \arrow[r, "c", swap]  
   & 3  \arrow[ul, "d+ab", swap, very near end, outer sep=-1.8, stealth-]
\end{tikzcd}
\\
\\[-8pt]
\vmutation{1}   & & & & \vmutation{2} \\[-6pt]
\\
\begin{tikzcd}[arrows={stealth-}, sep=6em]
  1  \arrow[r,  "a"]
  & 2  \arrow[d, swap, "b"]
  \\
   4 \arrow[ur, "e", swap, near start, outer sep=-1.8, stealth-] \arrow[u, swap, "f", -stealth] \arrow[r, "c+df+abf", swap, stealth-]  
   & 3  \arrow[ul, "d+ab", swap, very near end, outer sep=-1.8]
\end{tikzcd}
 &\multicolumn{3}{c}{}
 
&  \begin{tikzcd}[arrows={-stealth}, sep=6em]
  1     \arrow[dr, "d", bend left=12, near start, outer sep=-.8]  \arrow[d, "fa^2-f+ae", stealth-]   
  &2  \arrow[d, swap, "b"] \arrow[l, "a", swap, stealth-]
  \\
   4   \arrow[ur, "e+af", bend right=12, swap, near start, outer sep=-1.8, stealth-]
   &3  \arrow[l, "c", stealth-] 
\end{tikzcd}
 & \\
\\[-8pt]
\vmutation{2}   & & & & \vmutation{3} \\[-6pt]
\\
\begin{tikzcd}[arrows={-stealth}, sep=6em]
  1  \arrow[r,  "a"]
  & 2  \arrow[d, swap, "b"]
  \\
   4 \arrow[ur, "e", swap, near start, outer sep=-1] \arrow[u, swap, "f", -stealth] \arrow[r, "c+df+abf+eb", swap, stealth-]  
   & 3  \arrow[ul, "d", near end, swap, outer sep=-1, stealth-]
\end{tikzcd}
& \mutation 1
& \begin{tikzcd}[arrows={-stealth}, sep=6em]
  1     \arrow[dr, "d", near start, outer sep=-.8, stealth-]  \arrow[d,  "f"]   
  &2  \arrow[d, swap, "b"] \arrow[l, "a", swap ]
  \\
   4   \arrow[ur, "e+af", swap, near start, outer sep=-1.8]
   &3  \arrow[l, "c+eb+abf"] 
\end{tikzcd}
& \mutation 2 
&  \begin{tikzcd}[arrows={-stealth}, sep=6em]
  1     \arrow[dr, "d", near start, outer sep=-.8, bend left=12, stealth-]  \arrow[d, "fa^2-f+ae", stealth-]   
  &2  \arrow[d, swap, "b", stealth-] \arrow[l, "a", swap, stealth-]
  \\
   4   \arrow[ur, "e+af", swap, bend right=12, near start, outer sep=-1.8, stealth-]
   &3  \arrow[l, "c"] 
\end{tikzcd}
\end{array}
\hspace{-17pt} 
\end{equation*}
%\vspace{-.1in}
\caption{Mutation cycle of length~8 involving 4-vertex quivers.  
The $\binom{4}{2}=6$ integer parameters $a,b,c,d,e,f \geq 2$ can be chosen arbitrarily. 
(In fact, any $a,b,c,d,e,f\ge0$ would work if negative weights are interpreted using arrow reversal.) 
%Theorem~\ref{thm:primitive} shows that 
In the notation of Theorem~\ref{thm:Summary}, $a=q_{12}$, $b=q_{23}$, $c=q_{34}$, $d=q_{13}$, $e=q_{24}$, $f=q_{14}$. 
Each of these parameters appears somewhere along the cycle as an individual arrow multiplicity.
This phenomenon holds for any values of~$k$ and $n$. 
}
\label{fig:generic-8-cycle}
\vspace{-.2in}
\end{figure}
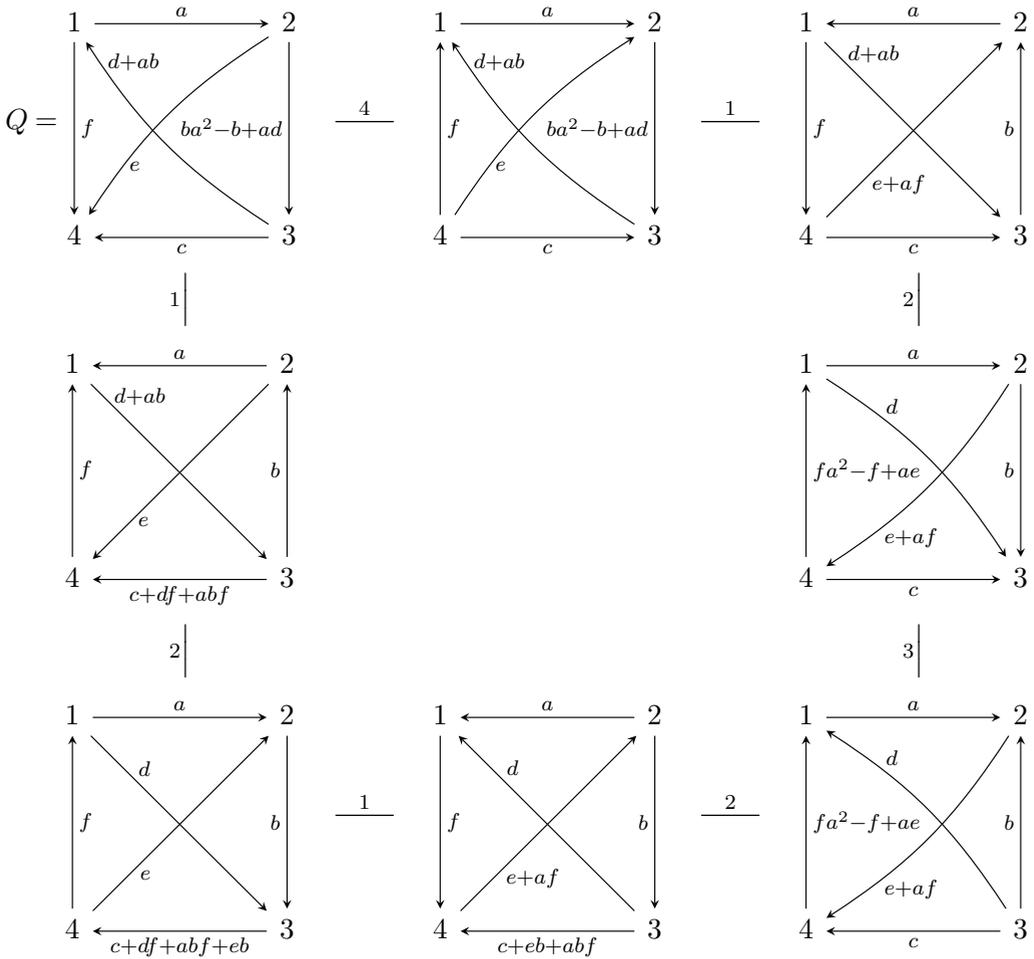

%\pagebreak[3]

For $n\ge 5$, the  statements in Theorem~\ref{thm:Summary} 
remain true even if we view our quivers up to isomorphism and/or global reversal of arrows. 

We conjecture (and prove in the $n=4$ case) that each of these mutation cycles is the unique cycle in the corresponding mutation graph, 
so that the associated cluster modular group~\cite{FG} 
(or cluster automorphism group~\cite{ASS}) is isomorphic to~$\ZZ$. 

All quivers appearing in the mutation cycles that we construct 
have arrow multiplicities~$\ge2$, i.e., any two vertices are connected by at least two arrows. 

\medskip

The construction from Theorem 1.1 can be extended to yield mutation cycles  (with different initial quivers)
whose mutation sequences are obtained by replacing the $1, 2, \dots, 1, 2$ and $2, 1, \dots, 2, 1$ fragments in~\eqref{eq:main-cycle} 
by much more general subwords. 
This has been independently observed by T.~Ervin~\cite{ervin}.

%In addition to the construction of Theorem~\ref{thm:Summary}, 
We describe several other ways to obtain nontrivial mutation cycles, see, e.g., 
Examples~\ref{eg:6-cycle-vortices} and \ref{eg:rosette}--\ref{eg:big-horseshoe}. 

\pagebreak[3]

Some of our original motivation came from the major unsolved problem of detecting mutation equivalence of quivers. 
For $3$-vertex quivers, the problem can be solved (cf.~\cite{ABBS}) using a descent algorithm that
repeatedly decreases the arrow multiplicities until it reaches a canonical minimal representative within a given mutation class.
In light of our results, any descent algorithm for detecting mutation equivalence of quivers 
would have to be confluent over arbitrarily large diamonds. 

It also follows that an algorithm for finding a shortest sequence of mutations
that connects two mutation-equivalent quivers cannot be based on a steepest descent heuristic.

\pagebreak[3]

None of the quivers appearing in this paper have frozen vertices.
In fact, we did not observe any large mutation cycles in which a particular vertex is frozen. 
%(so that the cycle does not involve a mutation at~$v$). 
To rephrase, each of the large mutation cycles that we found involves mutations at all vertices. 

\medskip

Throughout this paper, we use the following conventions. 
All our quivers are labeled, so that an $n$-vertex quiver~$Q$ typically has the vertex set $[1,n]=\{1,\ldots, n\}$. 
(Subquivers of such a quiver~$Q$ have vertex sets $V\subset [1,n]$.) 
%All vertices are mutable, there will be no frozen vertices. 
It is natural to consider three notions of quiver equivalence: 
(i)~equality of quivers as labeled graphs, 
\hbox{(ii)~isomorphism} of quivers  (i.e., quivers are the same up to a relabeling of vertices), and 
(iii)~isomorphism that may be combined with the global reversal of arrows. 
%(equal after some relabeling of vertices, and possibly reversing all arrows in the quiver). 
While we default to notion~(i), many of our results hold with respect to (ii) and/or~(iii).

\medskip

The paper is organized as follows. 

General background on quiver mutations is intro\-duced in Sections~\ref{sec:quiver-mut}--\ref{sec:3Vertex}. 

Various aspects of our main result are established in 
Sections~\ref{Sec:GeneralSemiCycles}--\ref{sec:genericity}.
In Section~\ref{Sec:GeneralSemiCycles}, we verify that applying the mutation sequence~\eqref{eq:main-cycle}
recovers the original quiver~$Q$ (cf.\ Theorem~\ref{thm:GeneralSemiCycles}). 
In Section~\ref{sec:distinctness}, we prove that all  quivers lying on this mutation cycle are distinct
(cf.\ Theorem~\ref{thm:GeneralSemiCycleDistinct}),
even if we treat them up to isomorphism and/or reversal of all arrows 
(under mild assumptions, cf.\ Theorem~\ref{thm:GeneralSemiCycleDistinct-noniso}). 
In Section~\ref{sec:exits}, we introduce and study certain properties of quivers that propagate under mutation in some directions. 
This combinatorial machinery is then used in Sections~\ref{Sec:primitiveCycles}--\ref{sec:genericity}.  
In~Section~\ref{Sec:primitiveCycles}, we show that the long mutation cycles constructed in Theorem~\ref{thm:GeneralSemiCycles}
cannot be paved by mutation cycles of bounded length. 
In Section~\ref{sec:genericity}, we demonstrate that the set of quivers that appear in our main construction 
is in some sense ``full-dimensional:''
it can be parametrized by $\binom{n}{2}$-tuples of nonnegative integers. 

In Section~\ref{sec:no-seed-cycles}, we show that 
none of the mutation cycles we construct gives rise to a cycle in the exchange graph of the associated cluster algebra. 
More generally, we show (see Theorem~\ref{th:no-seed-cycle}) that
no mutation cycle consisting entirely of quivers with at least two arrows between each pair of vertices 
yields a cycle in the exchange graph.

Each birational map obtained by composing seed mutations along one of our mutation cycles 
gives rise to a discrete dynamical system that deserves further study. 

In Section~\ref{sec:moreCycles} we present several additional mutation cycles 
that do not come from the construction of Theorem~\ref{thm:GeneralSemiCycles}. 
While each of these cycles belongs to a family of arbitrarily long (for a fixed~$n$) mutation cycles,
we do not attempt to describe these families explicitly. 

\medskip

\section*{Acknowledgments}
Some of our results were reported at the OPAC conference in Minneapolis (May 2022) and at the AMS special session on cluster algebras, positivity, and related topics (April 2023).
We thank the organizers of these events. 

After having developed the auxiliary machinery presented in Section~\ref{sec:exits},
we found out that a substantial part of it has appeared, in different form, in the earlier work by M.~Warkentin~\cite{Warkentin}. 
We are grateful to Tucker Ervin for bringing our attention to~\cite{Warkentin} and for sharing his own observations.

We thank Danielle Ensign for stimulating discussions and assistance with computer simulations. We thank the anonymous referee for their comments and corrections.

\clearpage

\newpage

\section{Quiver mutations}
\label{sec:quiver-mut}

In this section, we establish terminology and remind the reader of the relevant definitions and results. 
For a systematic introduction to the combinatorics of quiver mutations, see~\cite{FWZ}. 
For a sampling of additional results, see, e.g., \cite{BBH, Fordy-Marsh, Lawson-Mills, Warkentin}. 

\begin{definition}
\label{def:quiver}
A \emph{quiver} $Q$ is a finite directed graph with no directed cycles of length 1 or~2.
Its (directed) edges are called \emph{arrows}. 
Multiple arrows between a given pair of vertices are allowed. 
When drawing a quiver, we will typically write edge multiplicities next to single arrows,
rather then drawing multiple arrows.
Thus, we would draw $\bullet\!\stackrel{\scriptstyle2}{\longrightarrow}\!\bullet$ instead of 
$\bullet \!\rightrightarrows\! \bullet\,$. 

All our quivers will be \emph{labeled}, i.e., we will distinguish between different isomorphic quivers on the same vertex set. 
Accordingly, if $Q$ and $Q'$ are two quivers, then 
notation $Q = Q'$ will mean that $Q$ and $Q'$ are \emph{equal} as labeled graphs.
We note that in the literature, quivers are often considered up to graph isomorphism (a relabeling of vertices) and/or a global reversal of arrows; we do not follow this convention.
Cf.\ Example~\ref{eg:acyclicClasses}. 

For an $n$-vertex quiver~$Q$, we will typically use the~set 
\[
[1,n]=\{1,2,\dots,n\}
\]
as the set of (labels of the) vertices of~$Q$.  

\end{definition}

\begin{definition}
\label{def:B(Q)}
Each quiver gives rise to a skew-symmetric matrix $B(Q) = (b_{ij})$ 
whose entries $b_{ij}=b_{ij}(Q)$ indicate how many arrows run between each pair of vertices and what their orientation is.
The \emph{weight} $|b_{ij}(Q)|$ is the number of arrows between the vertices $i$ and~$j$. 
Notation $i \points j$ will indicate that all arrows between the vertices $i$ and~$j$ 
are directed from $i$ to~$j$; in other words, $b_{ij}\geq 0$. 
\end{definition}

\begin{definition}
%\label{def:}
A vertex $i$ in a quiver~$Q$ is called a \emph{sink} (resp., \emph{source}) if $j \points i$ (resp.,~$i \points j$) for all other vertices $j$ in $Q$. 
If $i$ is either a sink or a source, but we do not care which, we may say $i$ is a sink/source.
\end{definition}

\begin{definition}
%\label{def:}
We always use the term \emph{subquiver} to mean ``full subquiver'' (i.e., an induced subgraph). 
We denote the (full) subquiver of~$Q$ with vertex set $S$ by~$Q|_S$.
For example, $Q|_{ijk}$ denotes the subquiver of~$Q$ supported by the vertices~$\{i,j,k\}$. 
\end{definition}

\begin{definition}
\label{def:quiver-mutation}
To \emph{mutate} a quiver $Q$ at a vertex~$i$, perform the following steps:
\begin{enumerate}[leftmargin=.3in]
    \item for each path $j \rightarrow i \rightarrow k$ in~$Q$, add a new arrow from $j$ to $k$. 
    (Thus, if we have $j \stackrel{\scriptstyle a}{\longrightarrow} i \stackrel{\scriptstyle b}{\longrightarrow} k$ in~$Q$,
    then we should add $ab$ new arrows from $j$ to~$k$.)
    \item reverse all arrows incident to $i$.
    \item repeatedly remove oriented $2$-cycles until there are none left.
\end{enumerate}
The transformed (mutated) quiver is denoted by~$\T{i}{Q}$. We will write $Q \mutation{i} Q'$ to mean that $Q' = \T{i}{Q}$.
\end{definition}

The notation $\mu[i]$ is non-standard; in most of the literature, mutation at a vertex~$i$ is denoted by~$\mu_i$.  
We break convention in this paper to avoid nested subscripts and improve legibility.

\pagebreak[3]

Important properties of quiver mutation include:
\begin{itemize}[leftmargin=.2in]
\item 
$\mu[i]$ is an involution;
\item
$\mu[i]$ commutes with restriction, provided $i$ is in the restricted subset;
\item
$\mu[i]$ commutes with the reversal of all arrows in the quiver.
\end{itemize}

\begin{definition}
%\label{def:}
A \emph{sink/source mutation} $\mu[i]$ is a mutation at a sink or source vertex~$i$. 
Such a mutation simply reverses all the arrows incident to~$i$.
\end{definition}

\begin{remark}
%\label{rem:}
If $i$ is a sink or source in $Q$ and $S$ is a vertex set not containing~$i$, then
\[
\T{i}{Q}|_S = Q|_S.
\]
\end{remark}

We introduce the following notational shorthand to improve legibility. 

\begin{definition}
\label{def:stringOfMutations}
We denote $\mu[i_k i_{k-1}\cdots i_2 i_1] = \mu[i_k] \circ \cdots \circ \mu[i_2] \circ \mu[i_1]$, 
so that 
\begin{align*}
\mu[i_k i_{k-1}\cdots\, i_2 i_1](Q) &= \mu[i_k] \circ \mu[i_{k-1}] \circ \cdots \circ \mu[i_2] \circ \mu[i_1](Q).
\end{align*}
In other words, we apply the mutations indexed by the bracketed symbols 
in the right-to-left order (as is usual when composing maps).
\end{definition}

\begin{example}
\label{ex:notation-mu}
The following identities use the notation introduced in Definition~\ref{def:stringOfMutations}:
\begin{align*}
\T{i i}{Q} &= Q, \\
\T{j}{\T{i}{Q}} &= \T{j i }{Q}, \\
\T{i j } {\T{k l m}{Q}}&= \T{i j k l m}{Q}, \\
\T{1 2 3 2 3}{Q|_{123}} &= \T{12323}{Q}|_{123}.  
\end{align*}
\end{example}

\begin{definition}
\label{def:121212}
Notation $(i j)^k$ will denote the sequence $i j i j\cdots i j$ of length $2k$, alternating between $i$ and~$j$. 
Thus 
\[
\T{(i j)^k}{Q} \stackrel{\rm def}{=} \mu[i] \circ \mu[j] \circ \cdots \circ \mu[i] \circ \mu[j](Q)
\]
denotes the result of applying $k$ iterations of $\mu[i] \circ \mu[j]$ to~$Q$. %Usually $\{i,j\} = \{1, 2\}$.
\end{definition}

\begin{definition}
\label{def:mutationClass}
The \emph{mutation class} $[Q]$ of a (labeled) quiver $Q$ is the set of (labeled) quivers 
that can be obtained from~$Q$ by applying a sequence of mutations. 
\end{definition}

\begin{example}[\emph{Type~$\mathbf{A}_3$}]
Let $Q$ be an orientation of  the 3-vertex tree 
\begin{equation*}
\bullet\!\!-\!\!\!-\!\!\bullet\!\!-\!\!\!-\!\!\bullet 
\end{equation*}
(the Dynkin diagram of type~$A_3$). 
The mutation class $\mathbf{A}_3\!=\![Q]$ consists of $14$ quivers:
\begin{itemize}[leftmargin=.2in]
\item 
two oriented $3$-cycles on the vertices $1,2,3$; 
\item
$12$ quivers obtained by choosing all possible vertex labelings and edge orientations of the 3-vertex tree. 
\end{itemize}
By comparison, there are only $4$ quivers up to isomorphism in the mutation class~$\mathbf{A}_3$, 
namely the $3$-cycle and three different orientations of the 3-vertex tree:
\begin{equation*}
\bullet\!\rightarrow\!\bullet\!\rightarrow\!\bullet
\qquad\qquad
\bullet\!\rightarrow\!\bullet\!\leftarrow\!\bullet
\qquad\qquad
\bullet\!\leftarrow\!\bullet\!\rightarrow\!\bullet
\end{equation*}
The last two of these are equivalent up to global reversal of arrows, so there are $3$ quivers up to isomorphism and global reversal of arrows in the mutation class~$\mathbf{A}_3$.
\end{example}

\begin{remark}
%\label{rem:}
Two isomorphic (labeled) quivers may belong to different mutation classes; see Example~\ref{eg:acyclicClasses} below.  
\end{remark}

\begin{definition}
\label{def:mutationGraph}
The \emph{mutation graph} of a mutation class is the graph whose vertices are the quivers in the mutation class 
and whose edges correspond to mutations: for each pair of quivers $Q$ and~$Q'=\T{i}{Q}$, there is an edge
labeled $i$ connecting $Q$ and~$Q'$.  
\end{definition}

\begin{definition}
\label{def:mutation cycle}
A \emph{mutation cycle} in a mutation graph is a closed walk in which no two consecutive edges or vertices coincide. 
More precisely, a mutation cycle of length $N>0$ is defined by a quiver~$Q$ and a sequence $\mathbf{i}=i_1\cdots i_N$ such that
\begin{itemize}[leftmargin=.2in]
\item 
$\mu[i_N\cdots i_1](Q)=Q$; 
\item
%no two consecutive mutations in the cycle are the same: $i_j \neq i_{j+1}$;
$i_1\neq i_2\neq \cdots\neq i_{N-1}\neq i_N\neq i_1$; 
\item
no mutation is applied at an isolated vertex. 
\end{itemize}
We then say this mutation cycle is \emph{based} at~$Q$. 

We note that $\mu[i_N\cdots i_1](Q)=Q$ if and only if $\mu[i_1\cdots i_N](Q)=Q$.
\end{definition}

\begin{definition}
\label{def:acyclic}
A quiver is \emph{acyclic} if it contains no oriented cycles. 
\end{definition}

Note that the above definitions involve two kinds of cycles:
(a) oriented cycles within a particular quiver and 
(b) mutation cycles within a mutation graph (whose vertices correspond to quivers in a given mutation class).

\begin{proposition}
\label{pr:acyclic-mutation-cycle}
Let $Q$ be an acyclic quiver on the vertex set $[1,n]$,
with $b_{ij}(Q)\ge 0$ for $i<j$. 
Then $\T{12\cdots n}{Q}=Q$ and, equivalently,  $\T{n \cdots 21}{Q}=Q$.
\end{proposition}

\begin{proof}
When we mutate $Q$ at the sink~$n$, the latter vertex becomes a source
and the quiver becomes acyclic with respect to the linear ordering \hbox{$n\to 1\to 2\to\cdots \to n-1$},
a cyclic rearrangement of the original linear ordering. (The weights do not change.) 
After $n$ rotations, we return to the original quiver. See Figure~\ref{fig:central-triangle}. 
\end{proof}

\begin{figure}[ht]
\begin{equation*}
Q=\begin{tikzcd}[arrows={-stealth}]
  1 \arrow[r, "a"]  \arrow[rd, swap,"c"] 
  & 2 \arrow[d, "b" ] 
  \\
  & 3
\end{tikzcd}
\quad \mutation{1}
\quad
\begin{tikzcd}[arrows={-stealth}]
  1    
  & 2 \arrow[d, "b" ]  \arrow[l, swap, "a"]
  \\
  & 3 \arrow[lu, "c"]
\end{tikzcd}
\quad \mutation{2}
\quad
\begin{tikzcd}[arrows={-stealth}]
  1    \arrow[r, "a"]
  & 2  
  \\
  & 3 \arrow[lu, "c"] \arrow[u, swap, "b" ] 
\end{tikzcd}
\quad \mutation{3}
\quad
Q
\end{equation*}
\vspace{-.1in}
\caption{Acyclic quivers forming a mutation cycle. } 
\label{fig:central-triangle}
\end{figure}
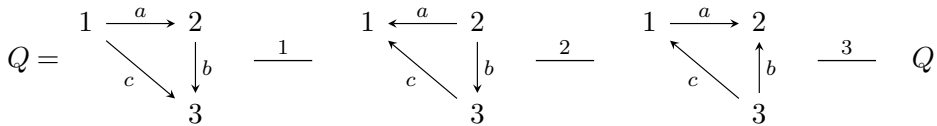

Proposition~\ref{pr:acyclic-mutation-cycle} shows that any $n$-vertex acyclic quiver~$Q$ lies on a mutation cycle of length~$n$
consisting entirely of acyclic quivers. 
All of them are re-orientations~of~$Q$. 

\medskip

In this paper, we will mostly study quivers satisfying the following condition:  

\begin{definition}
\label{def:largeWeights}
We say that a quiver~$Q$ has \emph{large weights} if $|b_{ij}(Q)| \geq 2$ for all pairs of distinct vertices~$i$ and~$j$. 
(Such quivers were called ``abundant'' in~\cite{Warkentin} and ``2-complete'' in~\cite{Felikson-Tumarkin-2018, LeeLee}.)
\end{definition}

Various results concerning quiver mutation simplify considerably when restricted to quivers with large weights. 
In particular, this is the case for the classification of mutation classes of 3-vertex quivers, to be discussed in Section~\ref{sec:3Vertex}.

\newpage

\section{Quivers on three vertices}
\label{sec:3Vertex}

Mutations of 3-vertex quivers are well understood, see especially~\cite{ABBS}.
In this section, we state and prove, without claiming any originality, all basic results about 3-vertex quivers that will be needed in the sequel.
To be specific, we will only need to treat the case of \emph{mutation-acyclic} 3-vertex quivers,
i.e., those quivers that are mutation equivalent to an acyclic quiver. 

\begin{definition}
\label{def:cyclic-quiver}
A 3-vertex quiver~$Q$ is called \emph{cyclic} %, or simply a \emph{$3$-cycle}, 
if it is not acyclic, i.e., if $Q$ contains an oriented 3-cycle. 
We will only use the term ``cyclic'' for 3-vertex quivers.
\end{definition}

\begin{definition}
\label{def:elbow}
Let $Q$ be an acyclic 3-vertex quiver with nonzero weights. 
Then $Q$ contains one sink and one source.
The remaining vertex is called the \emph{elbow} of~$Q$. 
\end{definition}

A mutation $\mu[i]$ in a 3-vertex quiver leaves two of the three weights unchanged. 

\begin{definition}
\label{def:ascent}
In a 3-vertex quiver~$Q$, 
a vertex~$i$ (or the mutation~$\mu[i]$)
is called an \emph{ascent} (resp., \emph{descent})
if $\mu[i]$ strictly increases (resp., decreases) one of the weights: 
\[
|b_{jk}(\T{i}{Q})| > |b_{jk}(Q)|. 
\]
Thus, $i$ is an ascent in $Q$ if and only if $i$ is a descent in~$\T{i}{Q}$. 
\end{definition}

\begin{example}
\label{eg:ascents}
Let $Q$ be a 3-vertex cyclic quiver shown below, with $b, c\ge 2$: 
\begin{equation}
\label{eq:2bc-3vertex}
Q=\begin{tikzcd}[arrows={-stealth}]
  1 \arrow[r, "2"] 
  & 2 \arrow[d, "b" ] 
  \\
  & 3   \arrow[lu, "c"] 
\end{tikzcd}
\end{equation}
(Thus $Q$ has large weights.) Then the ascents and descents of $Q$ are as follows:
\begin{itemize}[leftmargin=.2in]
\item 
if $2=b=c$ (the Markov quiver), 1, 2, 3 are neither ascents nor descents; 
\item
if $2<b=c$, then 3 is an ascent, 1 and 2 are neither ascents nor descents; 
\item 
if $2\le b<c$, then 1 and 3 are ascents and 2 is a descent; 
\item
if $2\le c<b$, then 2 and 3 are ascents and 1 is a descent. 
\end{itemize}

\end{example}

The following lemma is immediate from Definition~\ref{def:ascent}. 

\begin{lemma}
\label{lem:acyclic-descent}
Let $Q$ be an acyclic 3-vertex quiver with nonzero weights. 
Then $Q$ has one ascent (namely the elbow) and no descents. 
\end{lemma}

\begin{lemma}
\label{lem:3VertexUniqueDescents}
A cyclic $3$-vertex quiver $Q$ with large weights has at most one descent. 
If $Q$ has a descent, then it is unique, and is the vertex opposite the maximum weight.
Moreover, if $Q$ has a descent, then the other two vertices are ascents. 
\end{lemma}

\begin{proof}
By reversing arrows if necessary, we may assume that $1 \points 2 \points 3 \points 1$ in $Q.$ By relabeling the vertices, we may further assume that vertex $2$ is a descent of $Q.$ So we can compute explicitly
\[
|b_{13}(\mu[2](Q))| = |b_{31}(Q) - b_{12}(Q) b_{23}(Q)| < b_{31}(Q) 
\]
(where the inequality is due to $2$ being a descent of~$Q$). Since all of $b_{12}(Q), b_{23}(Q)$ and $b_{31}(Q)$ are positive and at least $2$, 
this implies that 
\[
2 b_{31}(Q) > b_{12}(Q) b_{23}(Q) \geq 2 \max (b_{12}(Q), b_{23}(Q)).
\]
So $b_{31}(Q),$ the weight opposite our descent, is the largest weight in~$Q.$ 
But $2$ was an arbitrary descent, so if we repeat the argument for another descent we would get two different maximal weights, a contradiction.

Next we check that if $Q$ has descent $2$, then the other vertices are ascents in $Q$. 
We compute
\[
b_{32}(\T{1}{Q}) = b_{31}(Q) b_{12}(Q) - b_{23}(Q) \ge 2 b_{31}(Q) - b_{23}(Q) > b_{23}(Q) > 0,
\]
so $1$ is an ascent in~$Q$. (Here we have used that the maximum weight is unique.) 
An analogous argument shows that $3$ is an ascent too.
\end{proof}

\begin{lemma}
\label{lem:3VertexAcyclicAscents}
Let $Q$ be an acyclic quiver with large weights on the vertex set $\{1,2,3\}$,
with the elbow at~2.  
Consider a sequence $i_1, i_2, i_3,\ldots\in\{1,2,3\}$ such that $i_1=2$ and $i_j\ne i_{j+1}$ for $j\ge 1$. 
Set $Q^{(0)}=Q$, $Q^{(1)}=\mu[i_1](Q)$, and more generally, 
\begin{equation}
\label{eq:Q^{(k)}}
Q^{(k)}=\mu[i_k \cdots i_2 i_1](Q) \quad (k\ge1).
\end{equation}
Then for any $k\ge 1$: 
\begin{itemize}[leftmargin=.2in]
\item
the mutation $\mu[i_{k}]$ is a descent in $Q^{(k)}$; 
\item
the quiver $Q^{(k)}$  is cyclic;
\item
the quivers $Q^{(k)}$ and $Q^{(k+1)}$ have opposite edge orientations. 
\end{itemize}
\end{lemma}

\begin{proof}
We argue by induction on~$k$. 

\emph{Base case:} $k=1$.
We have $i_1=2$ and $Q^{(1)} = \mu[2](Q).$ 
Suppose, without loss of generality, that in~$Q$ we have $1 \points 2 \points 3$ and hence $1 \points 3.$ 
Therefore
\begin{align*}
b_{13}(Q^{(1)}) &= b_{13}(Q) + b_{12}(Q) b_{23}(Q) > b_{13}(Q), \\
b_{12}(Q^{(1)}) &= - b_{12}(Q), \\
b_{23}(Q^{(1)}) &= - b_{23}(Q).
\end{align*}
In particular, $2$ is a descent in $Q^{(1)}$ and $Q^{(1)}$ is cyclic. 

Without loss of generality, we may assume that $i_2=1$, so that $Q^{(2)}= \mu[1](Q^{(1)})$. 
By Lemma~\ref{lem:3VertexUniqueDescents}, $1$ is an ascent in $Q^{(1)}$ (hence~$|b_{32}(Q^{(2)})| \geq |b_{32}(Q^{(1)})|$). 
Since the weights are nonzero and $2 \points 1 \points 3$ in $Q^{(1)},$ we also have $b_{32}(Q^{(2)}) < b_{32}(Q^{(1)})$. 
It follows that $b_{32}(Q^{(2)}) < 0,$ so $Q^{(1)}$ and $Q^{(2)}$ have opposite edge orientations. 

\emph{Induction step:} 
suppose that the claims are true for $Q^{(k-1)}.$ 
We will treat the case where $i_k = 2, i_{k+1}=1$ and $1 \points 2$ in $Q^{(k)},$ all other cases being similar. 
By induction,~$Q^{(k)}$ has opposite edge orientations from~$Q^{(k-1)},$ so in particular is cyclic. 
Also by induction,~$Q^{(k-1)}$ has a descent at $i_{k-1} \neq 2,$ and so Lemma~\ref{lem:3VertexUniqueDescents} implies that~$Q^{(k-1)}$
 has ascent $2$, or equivalently $Q^{(k)}$ has descent~$2$. 

Finally, the mutation $\mu[1]$ reverses all arrows incident to $1$ in $Q^{(k+1)} = \mu[1](Q^{(k)}),$ 
so it only remains to show that $3 \points 2$ in $Q^{(k+1)}.$ 
If not, then $Q^{(k+1)}$ is acyclic and~${b_{23}(Q^{(k)}) > b_{23}(Q^{(k+1)})}$ (since $Q^{(k)}$ has a path~$3 \points 1 \points 2$). 
But then $Q^{(k)}$ has two descents, contradicting Lemma~\ref{lem:3VertexUniqueDescents}.
\end{proof}

\pagebreak[3]

The mutation class of an arbitrary 3-vertex acyclic quiver with large weights can be parametrized as follows. 

\begin{lemma}
\label{lem:3VertexAcyclicAscents-1}
Let $Q$ be an acyclic quiver with large weights on the vertex set $\{1,2,3\}$,
with the elbow at the vertex~2.  
Then the map
\begin{equation*}
(i_1,\dots,i_k) \mapsto \mu[i_k \cdots i_2 i_1](Q)
\end{equation*}
(cf.~\eqref{eq:Q^{(k)}}) is a bijection between 
\begin{itemize}[leftmargin=.2in]
\item 
the set of  finite sequences of vertices $(i_1,\dots,i_k)$, $k\ge 0$, satisfying $i_j \neq i_{j+1}$ and $i_2\ne 2$, and  
\item
the mutation class~$[Q]$ (see Definition~\ref{def:mutationClass}). 
\end{itemize}
Thus the mutation graph of $Q$ consists of three complete infinite binary trees of cyclic quivers, with each root connected to a different acyclic quiver.
The three acyclic quivers form a mutation cycle of sink/source mutations (see Figure~\ref{fig:central-triangle}).
See Figure~\ref{fig:largeWeightAcyclic}. 
\end{lemma}

\begin{figure}[ht]
%\vspace{-.1in}
%\includegraphics[width=7.5cm]{LargeWeightAcyclic3Vertex.png}
%\vspace{-.1in}
\begin{tikzcd}[arrows={-stealth}, sep=5, cramped]
&&&& \bullet   \ar[rrrrd, no head, "3"] &&&&&&&& \bullet \ar[lllld, no head, swap, "1"] &&&&&& \bullet \ar[rrrrd, no head, "2"] &&&&&&&& \bullet \ar[lllld, no head, swap, "3"] \\ 
&&&&&&&& \bullet \ar[d, no head, "2"]&&&&&&&&&&&&&& \bullet \ar[d, no head, "1"] \\[20pt]
\bullet \ar[rrrrd, no head, "2"] &&&&&&&& \bullet \ar[rrrrd, no head, "1"] \ar[lllld, no head, swap, "3"] &&&&&&&&&&&&&& \bullet \ar[rrrrd, no head, "3"] \ar[lllld, no head, swap, "2"] &&&&&&&& \bullet \ar[lllld, no head, swap, "1"] \\
&&&& \bullet \ar[d, no head, "1"]&&&&&&&& \circ \ar[rrrrrr, no head, "\boxed{\scriptstyle 3}"] \ar[rrrdd, no head, near start, swap, "\boxed{\scriptstyle 2}"] &&&&&& \circ \ar[llldd, no head, near start, "\boxed{\scriptstyle 1}"] &&&&&&&& \bullet \ar[d, no head, "2"] \\[20pt]
&&&& \bullet  &&&&&&&&&&& &&&&&&&&&&& \bullet \\[-6pt]
&&&& &&&&&&&&&&& \circ \ar[d, no head, "3"]  \\[20pt]
&&&&&&&&&&&&&&& \bullet \ar[lllld, no head, swap, "1"] \ar[rrrrd, no head, "2"] \\[10pt]
&&&&&&& \bullet \ar[rrrr, no head, "3"] &&&& \bullet \ar[d, no head, "2"] &&&&&&&& \bullet \ar[d, no head, "1"] \ar[rrrr, no head, "3"] &&&& \bullet \\[20pt]
&&&&&&&&&&& \bullet &&&&&&&& \bullet
\end{tikzcd}

\caption{
Part of the mutation graph of a $3$-vertex acyclic quiver with large weights,
showing the vertices at distance $\le 3$ from the acyclic triangle. 
White %green 
(resp.,~black) %(resp.,~blue) 
\linebreak[3]
vertices correspond to acyclic (resp., cyclic) quivers. 
Each edge is labeled with the mutation performed.
The label is $\boxed{\text{boxed}}$ when the mutation is at a sink/source. 
%The cycle of green vertices is the only cycle in the mutation graph.
}
\label{fig:largeWeightAcyclic}
\end{figure}
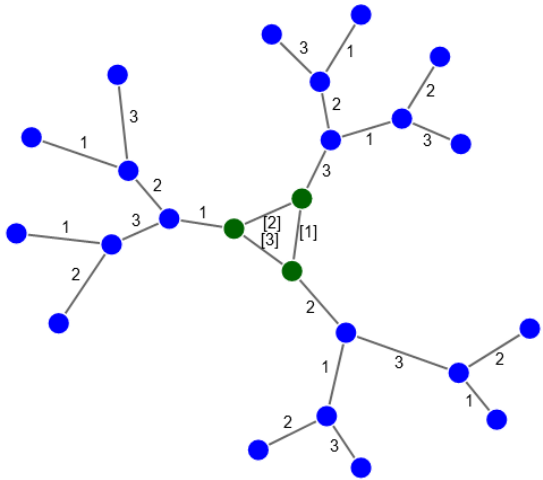

\vspace{-10pt}

\begin{proof}
%\comment{to be checked/edited}
Every quiver in $[Q]$ is given by $\mu[i_k \cdots i_1](Q)$, for some sequence $(i_1,\dots,i_k)$. 
If~$i_j=i_{j+1}$ for some $j$, then we can shorten the sequence by removing these two entries.
Similarly, if $i_2=2$, then we can shorten the sequence via 
\begin{align*}
\T{\cdots i_4 i_3 2 1}{Q}&=\T{\cdots i_4 i_3 3}{Q}, \\
\T{\cdots i_4 i_3 2 3}{Q}&=\T{\cdots i_4 i_3 1}{Q}
\end{align*}
(here we use that $\T{321}{Q}=\T{123}{Q}=Q$, see Proposition~\ref{pr:acyclic-mutation-cycle}). 
We may therefore assume that $i_j \neq i_{j+1}$ for all~$j$ and moreover
$i_2\neq 2$. Thus we have a surjection.

\pagebreak[3]

Let us show that for any $Q'\in[Q]$, the sequence $(i_1,\dots,i_k)$ with the required properties is unique. 
Indeed, it follows from Lemma~\ref{lem:3VertexAcyclicAscents} that $i_k$ is the unique descent in~$Q'$,~$i_{k-1}$
 is the unique descent in~$\mu[i_k](Q')$, 
$i_{k-2}$ is the unique descent in~$\mu[i_{k-1} i_k](Q')$, etc., until we reach an acyclic quiver
(either $Q$ or $\mu[1](Q)$ or~$\mu[3](Q)$). 
\end{proof}

\begin{remark}
\label{rem:3-vertex-arbitrary}
One can use the results in~\cite{ABBS}
to extend the above description to arbitrary mutation classes of 3-vertex quivers. 
This can be further generalized to $3\times 3$ skew-symmetrizable matrices, cf.~\cite{Seven3x3}. 
\end{remark}

\begin{definition}
\label{def:height}
Let $Q$ be a $3$-vertex quiver which is mutation equivalent to an acyclic quiver with large weights. 
By Lemma~\ref{lem:3VertexAcyclicAscents}, the quiver~$Q$ also has large weights.  
The \emph{descent sequence} of~$Q$ is the (unique) longest sequence $\mathbf{i}=i_1\cdots i_k$ of successive descents
originating at~$Q$; in other words, $i_1$ is the unique descent of~$Q$,
$i_2$ is the unique descent of $\mu[i_1](Q)$, etc. 
Alternatively, $\mathbf{i}$ is the (unique) shortest sequence of mutations such that the quiver $\T{\mathbf{i}}{Q}$ is acyclic. 
Cf.\ the last paragraph of the proof of Lemma~\ref{lem:3VertexAcyclicAscents-1}. 
\end{definition}

\begin{example}
\label{eg:acyclicClasses}
Let $R$ and $R'$  be acyclic $3$-vertex quivers shown below, with distinct large weights $a,b,c$: 
\begin{equation}
\label{eq:RR'-3vertex}
R=\begin{tikzcd}[arrows={-stealth}]
  1 \arrow[r, "a"] 
  \arrow[rd, swap,"c"] 
  & 2 \arrow[d, "b" ] 
  \\
  & 3
\end{tikzcd}
\qquad\qquad
R'=\begin{tikzcd}[arrows={-stealth}]
  1 
  & 2 \arrow[d, "a" ] \arrow[l, swap, "c"] 
  \\
  & 3   \arrow[lu, "b"] 
\end{tikzcd}
\end{equation}
The quivers $R$ and $R'$ are isomorphic but \emph{not} mutation equivalent. 
Indeed, by Lemma~\ref{lem:3VertexAcyclicAscents-1}, 
all acyclic quivers in the mutation class $[R]$ are obtained from~$R$ by reorienting its arrows
(while keeping the labeling intact), see Figure~\ref{fig:central-triangle}. 
By contrast, getting $R'$ from~$R$ requires relabeling of the vertices. 
\end{example}

To state our next technical result, we will need the following notation.

\begin{definition}
\label{def:chebyshev-poly}
We denote by $\h_j(a)\in\ZZ[a]$, $j=0,1,2,\dots$, the \emph{monic Chebyshev polynomials of the second kind}
defined by 
\begin{equation*}
%\h_j(x) = U_j(\tfrac x 2), \qquad 
\h_j(2\cos\theta)=\frac{\sin((j+1)\theta)}{\sin\theta}
\end{equation*}
or by the recurrence 
\begin{equation}
\label{eq:chebyshev-rec}
\h_{j+1}(a)=a\h_j(a)-\h_{j-1}(a), 
\end{equation}
with initial values $\h_0(a)=1$, $\h_1(a)=a$. 
Thus %$\h_0=1$, $\h_1=x$, 
$\h_2(a)=a^2-1$, $\h_3(a)=a^3-2a$, \dots
\end{definition}

\begin{proposition}
\label{pr:3VertexAlternatingMutations}
Let $Q^{(1)}$ be an acyclic quiver on the vertex set $\{1,2,3\}$, with large weights and~${3\points1\points2}$ (thus~${3\points2}$). 
Let $a\!=\! b_{12}(Q^{(1)})$, $b\!=\!b_{31}(Q^{(1)})$, $c\!=\!b_{32}(Q^{(1)})$. 
For $j=1,2,\dots$,~set
\begin{align*}
Q^{(2j+1)}&=\T{(2 1)^j}{Q^{(1)}}, \\
Q^{(2j+2)}&=\T{1(2 1)^j}{Q^{(1)}},
\end{align*}
so that
\begin{equation}
\label{eq:mu12121...}
Q^{(1)} \mutation{1} 
Q^{(2)} \mutation{2}
Q^{(3)} \mutation{1}
Q^{(4)} \mutation{2}
\cdots,
\end{equation}
cf.\ Figure~\ref{fig:mu1mu2}.
Then the weights of the cyclic quivers $Q^{(2)}$, $Q^{(3)}$, $Q^{(4)}$, \dots are given by the following formulas. 
For $j\ge0$, we have
\begin{align*}
b_{21}(Q^{(2j+2)}) &=  a,\\
b_{13}(Q^{(2j+2)}) &=  \h_{2j}(a)b + \h_{2j-1}(a)c,\\%convention, \h_{-1}(a)=0
b_{32}(Q^{(2j+2)}) &=  \h_{2j+1}(a)b + \h_{2j}(a)c. 
\end{align*}
and for $j>0$, we have
\begin{align*} 
b_{12}(Q^{(2j+1)}) &=  a,\\
b_{23}(Q^{(2j+1)}) &=  \h_{2j-1}(a)b + \h_{2j-2}(a)c,\\
b_{31}(Q^{(2j+1)}) &=  \h_{2j}(a)b + \h_{2j-1}(a)c.\\
\end{align*}
\end{proposition}

\begin{figure}[ht]
\begin{equation*}
\!\!\!\!\begin{array}{ccccccc}
\begin{tikzcd}[arrows={-stealth}]
  1 \arrow[r, "a"]  
  & 2 
  \\
  3 \arrow[u,"b"] \arrow[ur, swap, "c" ] &
\end{tikzcd}
& &
\begin{tikzcd}[arrows={-stealth}]
  1    \arrow[d, swap, "b"] 
  & 2 \arrow[l, swap, "a" ]  
  \\
   3 \arrow[ur, swap, "ab+c"] &
\end{tikzcd}
& &
\!\begin{tikzcd}[arrows={-stealth}]
  1    \arrow[r, "a"]
  & 2  \arrow[dl, "ab+c" ] 
  \\
   3 \arrow[u, "\!\!\!(a^2\!-1)b+ac"] &
\end{tikzcd}
& &
\!\begin{tikzcd}[arrows={-stealth}]
  1    \arrow[d, swap, "\!\!\!(a^2\!-1)b+ac"] 
  & 2 \arrow[l, swap, "a" ]  
  \\
   3 \arrow[ur, swap, "\scriptscriptstyle{(a^3\!-2a)b+(a^2\!-1)c}", near start]& %added scriptscriptstyle
\end{tikzcd}
\\
Q^{(1)}
& \!\!\!\!\!\!\mutation{1}\!\!\!\!\!\!
& Q^{(2)}
& \!\!\!\!\!\!\mutation{2}\!\!\!\!\!\!
& \!Q^{(3)}
& \!\!\!\!\!\!\!\!\mutation{1}\!\!\!\!\!\!\!\!
& \!\!\!\!\!\!\!\!\!\!\!\!\!\!\!\!\!\!\!\!Q^{(4)}
\end{array}
\!
\end{equation*}
%\vspace{-.5in}
\caption{The first three mutations in Proposition~\ref{pr:3VertexAlternatingMutations}.}
\label{fig:mu1mu2}
\end{figure}
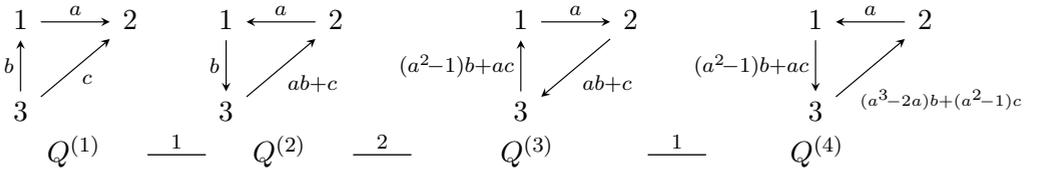

\begin{proof} 
Note that $Q^{(2)}, Q^{(3)}, \ldots$ are cyclic by Lemma~\ref{lem:3VertexAcyclicAscents}, since our first mutation is at the elbow~$1.$ 
Each of the mutations \eqref{eq:mu12121...} keeps two of the three weights intact and creates one new weight.
These new weights, together with $b$, form the sequence
\begin{equation*}
b, ab+c, (a^2-1)b+ac, (a^3-2a)b+(a^2-1)c, \dots 
\end{equation*}
that satisfies the same recurrence as the monic Chebyshev polynomials, see~\eqref{eq:chebyshev-rec}.
The desired formulas follow.
\end{proof}

We will later need the following simple observation about 3-vertex subquivers of a larger quiver. 

\begin{lemma}
\label{lem:4VertexWeightChanges}
Let $i,j,u,v$ be four vertices in a quiver~$Q$. 
Suppose that $i$ is an ascent or descent in $Q|_{iju}$ and an ascent or descent in $Q|_{ijv}$.
Then $b_{uv}(Q) = b_{uv}(\T{i}{Q})$.
\end{lemma}

\begin{proof}
Assume, without loss of generality, that we have orientation $j \points i$ in $Q$. 
In order for $i$ to be an ascent or descent in $Q|_{iju}$, 
we must have $i \points u$, and similarly~$i \points v$. 
But then $i$ is a source in $Q|_{iuv}$, so $\mu[i]$ does not affect $b_{uv}$.
\end{proof}

\newpage

\section{Main construction}
\label{Sec:GeneralSemiCycles}

In this section (see Theorem~\ref{thm:GeneralSemiCycles}), 
we establish a key claim made in Theorem~\ref{thm:Summary}: 
the quiver~$Q$ from Theorem~\ref{thm:Summary} 
is fixed by the mutation sequence~\eqref{eq:main-cycle} (cf.~\eqref{eq:GeneralSemiCycle}).
Note that while the constructions of~$Q$ presented in Theorems~\ref{thm:Summary} and~\ref{thm:GeneralSemiCycles} 
may seem different, they actually yield the same quiver, 
by Proposition~\ref{pr:3VertexAlternatingMutations} and Lemma~\ref{lem:4VertexWeightChanges}.

For $a,b\in\ZZ$, we will use the notation $[a,b]=\{a,a+1,\dots,b-1,b\}$. 

\begin{theorem}
\label{thm:GeneralSemiCycles}
Let $n\ge 4$ and $k> 0$.
Let $\tilde R$ be a quiver on the vertex set $[1,n-1]$ 
such that $b_{ij}(\tilde R)\ge 2$ %(so $i \rightarrow j$)
whenever $i<j$. 
(In particular, $\tilde R$ is acyclic and has large weights.) 
Define the quiver $Q$ on the vertex set $[1,n]$ 
by setting 
\begin{equation}
\label{eq:Q{[1,n]}}
Q|_{[1,n-1]}= \T{(12)^k}{\tilde R}
\end{equation}
and choosing the values $b_{in}(Q) \geq 2$ arbitrarily; 
in particular, $n$~is a sink in~$Q$.  
Then the quiver $Q$ lies on the following mutation cycle of length $n+4k$:
\begin{equation}
\label{eq:GeneralSemiCycle}
    Q = \T{(1 2)^k 1 2 3 \cdots (n-2) (n-1) (2 1)^k n}{Q}.
\end{equation}
\end{theorem}

\begin{example}
\label{eg:n=4,k=2}
The $n=4, k=2$ case of Theorem~\ref{thm:GeneralSemiCycles} is shown in Figure~\ref{fig:generic-8-cycle} , 
with the quiver~$Q$ appearing in the top left corner.
We assume that $a,b,c,d,e,f \geq 2$. 
\end{example}

\begin{example}
\label{eg:longcycle5}
Here is a specific example with $n=5$ and $k=2$. 
Choose 
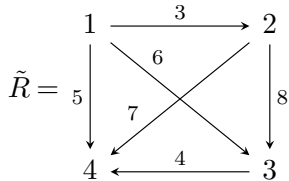
\begin{equation*}
\tilde R=\begin{tikzcd}[arrows={-stealth}, row sep=40]
  1 \arrow[rr, "3"]  \arrow[rrd, near start, "6"] \arrow[d, swap, "5" ]
  && 2 \arrow[d, "8" ] \arrow[dll,swap, near end, "7"] 
  \\
4   && 3 \arrow[ll, swap, "4" ] 
\end{tikzcd}
\end{equation*}
so that
\begin{equation*}
Q|_{[1,4]}\!=\!\T{1212}{\tilde R} = 
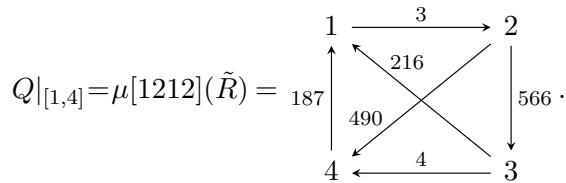
\begin{tikzcd}[arrows={-stealth}, row sep=40]
  1 \arrow[rr, "3"]  
  && 2 \arrow[d, "566" ] \arrow[dll,swap, near end, outer sep=-1, "490"] 
  \\
4 \arrow[u, "187" ]  && 3 \arrow[ll, swap, "4" ] \arrow[llu, swap, near end, outer sep=-1, "216"] 
\end{tikzcd}
.
\end{equation*}
Now extend $Q|_{[1,4]}$ to $Q$ by setting arbitrary values $b_{15},b_{25}, b_{35},b_{45}\ge2$.
E.g., 
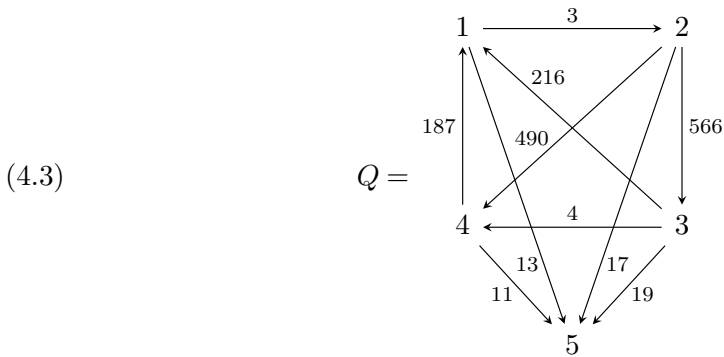
\begin{equation}
Q= \begin{tikzcd}[arrows={-stealth}, row sep=30]
  1 \arrow[rr, "3"]  \arrow[rddd, near end, swap, outer sep=-1.5, "13"]  
  && 2 \arrow[dd, "566" ] \arrow[ddll, swap, "490", pos=0.62, outer sep=-1.5] \arrow[lddd, near end, outer sep=-1.5, "17"]  
  \\
  \\
4 \arrow[uu, "187" ] \arrow[rd, swap, outer sep=-1.5, "11" ]  && 
       3 \arrow[ll, swap, "4" ] \arrow[lluu,  "216", swap, near end, outer sep=-1]  \arrow[ld, outer sep=-1.5, "19" ] 
\\
& 5
\end{tikzcd}
\end{equation}
Then \eqref{eq:GeneralSemiCycle} asserts that
$\T{1212123421215}{Q}=Q$, as in Figure~\ref{fig:egCycle}. 
\end{example}
 
\begin{figure}[ht]
\vspace{-.2in}
\begin{equation*}
\begin{tikzcd}[arrows={-stealth, cramped}
]
\red{\boxed{Q}}  \arrow[rrrr, no head, "{5}"] \arrow[d, swap, no head, "1"]
  &&&& \red{\bullet} \arrow[d, no head, "1"] \\
\blue{\bullet} \arrow[d, swap, no head,"2"] &&&&  \blue{\bullet} \arrow[d, no head,"2"]  \\
\blue{\bullet} \arrow[d, swap, no head,"1"] &&&&  \blue{\bullet} \arrow[d, no head,"1"]  \\
\blue{\bullet} \arrow[d, swap, no head,"2"] &&&&  \blue{\bullet} \arrow[d, no head,"2"]  \\
\blue{\boxed{L}} \arrow[r, no head, "1"] & \blue{\bullet} \arrow[r, no head, "2" ] 
    & \red{\bullet} \arrow[r, no head,"{3}"] & \red{\bullet} \arrow[r, no head,"{4}"] & \red{\boxed{R}}
\end{tikzcd}
\end{equation*}
\vspace{-.1in}
\caption{
The mutation cycle from Theorem~\ref{thm:GeneralSemiCycles}, for $n=5, k=2$. 
Red (resp.,~blue) vertices correspond to the quivers that have (resp., do not have) a source or sink.} 
\vspace{-.2in}
\label{fig:egCycle}
\end{figure}
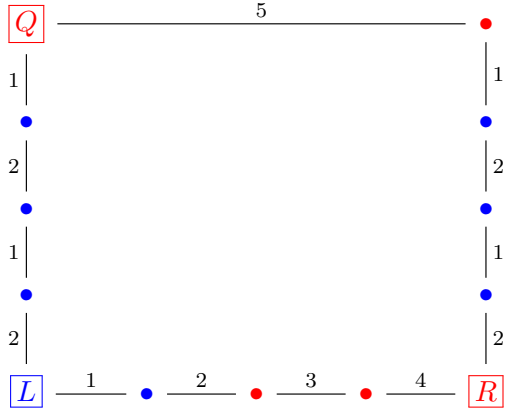

The remainder of this section is devoted to the proof of Theorem~\ref{thm:GeneralSemiCycles}. 

To make it easier to discuss particular mutations and quivers along the mutation sequence
appearing in~\eqref{eq:GeneralSemiCycle}, 
we denote by $Q^{(j)}$ the result of applying the first $j$ mutations in this sequence (starting with $\mu[n]$) to~$Q$,
so that 
\begin{equation}
\label{eq:Q123}
Q^{(0)}=Q,\ \ Q^{(1)}=\T{n}{Q},\ \  Q^{(2)}=\T{1n}{Q}, \ \  Q^{(3)}=\T{21n}{Q},  \dots, 
\end{equation}
see Figure~\ref{fig:GeneralSemiCycles-notation}. 
We then denote
\begin{align}
\label{eq:RLQ-R}
R&=Q^{(2k+1)}=\T{(2 1)^k n}{Q}, \\
\label{eq:RLQ-L}
L&= Q^{(2k+n)}=\T{1 \cdots (n-1) (21)^k n}{Q}=\T{1 \cdots (n-1)}{R} , \\
\label{eq:Qpdef}
Q' &= Q^{(4k+n)}= \T{(1 2)^k 1 2 3 \cdots (n-2) (n-1) (2 1)^k n}{Q}. 
\end{align}

\begin{figure}[ht]
\vspace{-.2in}
\begin{equation*}
\quad \begin{tikzcd}[arrows={-stealth}, sep=15]
\hspace{-.34in}%\hspace{-.69in}
Q^{(4k+n)}=Q' \ar[d, swap, no head, "1"] %\ar[r, equal, "?"] &
\stackrel{?}{=}  Q  %\arrow[r, double, no head, "?"]  & 
 \arrow[rrrrrr, no head, "n"] &
                                                                    &&&&& Q^{(1)} \arrow[d, no head, "1"] \\
\bullet \arrow[d, swap, no head,"2"] &&&&&&  Q^{(2)} \arrow[d, no head,"2"]  \\
%\bullet \arrow[d, swap, no head,"1"]  &&&&&&  Q^{(3)} \arrow[d, no head,"1"]  \\
%\bullet \arrow[d, swap, no head,"2"]  &&&&&&  Q^{(4)} \arrow[d, no head,"2"]  \\
\bullet \arrow[d, swap, no head,dotted]  &&&&&&  Q^{(3)} \arrow[d, no head,dotted]  \\
\bullet \arrow[d, swap, no head,"1"] &&&&&&  Q^{(2k-1)} \arrow[d, no head,"1"]  \\
\bullet \arrow[d, swap, no head,"2"] &&&&&&  Q^{(2k)} \arrow[d, no head,"2"]  \\
\hspace{-.69in}Q^{(2k+n)}=L \arrow[r, swap, no head, "1"] 
    & \bullet \arrow[r, swap, no head, "2" ] 
    & \bullet \arrow[r, swap, no head,"3"] 
    & \bullet \arrow[r, dotted, no head] & \bullet \arrow[r, swap, no head,"n-2"] 
         & \bullet \arrow[r, swap, no head,"n-1"] & R=Q^{(2k+1)} \hspace{-.65in}
\end{tikzcd}
\end{equation*}
\vspace{-.1in}
\caption{The putative mutation cycle, see \eqref{eq:Q123}--\eqref{eq:Qpdef}. 
}
\label{fig:GeneralSemiCycles-notation}
\end{figure}
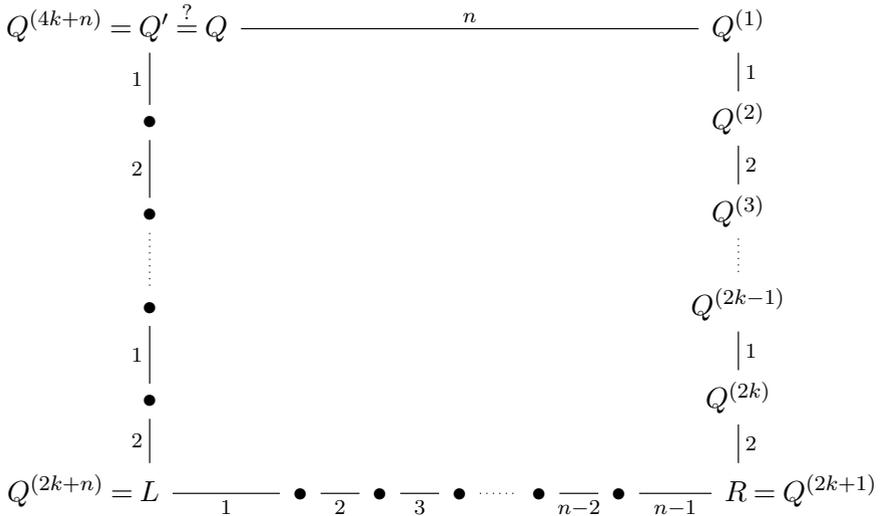

Our goal is to show that $Q = Q'$. 
We will gradually demonstrate that certain subquivers of $Q$ and $Q'$ are equal,
see Lemmas~\ref{lem:SubquiverTildeQ}, \ref{lem:Subquiver12n}, and~\ref{lem:bin}.

%%\textbf{Step 1:}
\begin{lemma}
\label{lem:SubquiverTildeQ} 
We have $L|_{[1,n-1]}\!=\!R|_{[1, n-1]}\!=\!\tilde R$ and $Q|_{[1,n-1]}\!=\!Q^{(1)}|_{[1,n-1]}\!=\!Q'|_{[1, n-1]}$.
More generally, 
\begin{equation}
\label{eq:SubquiverTildeQ-2}
Q^{(4k+n-\ell)}|_{[1,n-1]} = Q^{(\ell+1)}|_{[1,n-1]} \quad (0 \leq \ell \leq 2k). 
\end{equation}

\end{lemma}

\begin{proof}%[Proof of Lemma~\ref{lem:SubquiverTildeQ}]
Since $n$ is a sink in~$Q$, we have 
\begin{equation}
\label{eq:Q[n-1]=Q1}
Q|_{[1,n-1]}=Q^{(1)}|_{[1,n-1]}.
\end{equation}
Therefore
\begin{equation}
\label{eq:R[n-1]=tildeR}
R|_{[1,n-1]}\stackrel{\eqref{eq:RLQ-R}}{=}\T{(21)^k}{Q^{(1)}}|_{[1,n-1]}
\stackrel{\eqref{eq:Q[n-1]=Q1}}{=}\T{(21)^k}{Q}|_{[1,n-1]}\stackrel{\eqref{eq:Q{[1,n]}}}{=}\tilde R. 
\end{equation}
Recall that $\tilde R$ is acyclic with $b_{ij}(\tilde R)>0$ for $i<j$. 
It follows by Proposition~\ref{pr:acyclic-mutation-cycle} that
\begin{equation}
\label{eq:L[n-1]=tildeL}
L|_{[1, n-1]}\stackrel{\eqref{eq:RLQ-L}}{=}\T{12\cdots(n-1)}{R}|_{[1,n-1]}
\stackrel{\eqref{eq:R[n-1]=tildeR}}{=}\T{12\cdots(n-1)}{\tilde R}=\tilde R .
%\stackrel{\eqref{eq:R[n-1]=tildeR}}{=} R|_{[1,n-1]}. 
\end{equation}
We conclude that
\begin{equation*}
Q'|_{[1, n-1]}
\!\stackrel{\eqref{eq:Qpdef}}{=}\!   \T{(12)^k}{L}|_{[1, n-1]}
\!\stackrel{\eqref{eq:L[n-1]=tildeL}}{=\!\!=}\!   \T{(12)^k}{R}|_{[1, n-1]}
\!\stackrel{\eqref{eq:RLQ-R}}{=}\!   Q^{(1)}|_{[1,n-1]}=Q|_{[1,n-1]}. 
\end{equation*}
Identity \eqref{eq:SubquiverTildeQ-2} is deduced in the same way, by applying the mutations
$\mu[\cdots 212]$ to the identity $L|_{[1,n-1]}\!=\!R|_{[1, n-1]}$. 
\end{proof}

Now that we have shown that $Q|_{[1,n-1]}=Q'|_{[1, n-1]}$,
it remains to demonstrate that the multiplicities and directions of arrows incident to the vertex~$n$
are the same in $Q$ and~$Q'$; that is, we need to show that $b_{in}(Q)=b_{in}(Q')$ for $i=1,\dots,n-1$. 

\begin{lemma}
\label{lem:Subquiver12nStart} 
The quiver $Q^{(1)}|_{12n}$ is acyclic with large weights, with elbow at~$1$. \\
The quiver $Q|_{12n}$ is acyclic with large weights, with elbow at~$2$.
\end{lemma}

\begin{proof}
We have $b_{12}(R) =b_{12}(\tilde R)\ge 2$ by construction. 
Since $Q^{(1)}=\T{(12)^k}{R}$, it follows that $b_{12}(Q^{(1)})=b_{12}(R)\ge 2$. 

We have $b_{1n}(Q),b_{2n}(Q)\ge 2$ by construction.
Since $Q^{(1)}=\T{n}{Q}$, it follows that~$b_{n1}(Q^{(1)}), b_{n2}(Q^{(1)})\ge 2$, 
and we are done with $Q^{(1)}|_{12n}$. 

The claim regarding $Q|_{12n}=\T{n}{Q^{(1)}|_{12n}}$ immediately follows. 
\end{proof}

For $1\le \ell\le 2k$, we denote
\begin{equation}
\label{eq:varepsilon}
\varepsilon(\ell)=
\begin{cases}
1 & \text{if $\ell$ is odd;} \\
2 & \text{if $\ell$ is even,}
\end{cases}
\end{equation}
so that the mutations along the right rim of the diagram in Figure~\ref{fig:GeneralSemiCycles-notation}
take the form
\begin{equation*}
Q^{(\ell)}\mutation{\varepsilon(\ell)}Q^{(\ell+1)}. 
\end{equation*}

\begin{lemma}
\label{lem:mut-epsilon}
Let $1\le \ell\le 2k$ and $i\!\in\! [3,n\!-\!1]$. 
Then $\varepsilon(\ell)$ is a descent in $Q^{(\ell)}|_{12i}$.
Furthermore, $Q^{(\ell)}|_{12n}$ is cyclic, with large weights and ascent at~$\varepsilon(\ell)$. 
In particular, $R|_{12n}$ has large weights and the orientation $1\points 2\points n\points 1$. 
%$\varepsilon(\ell)$ is an ascent in $Q^{(\ell)}|_{12n}$. 
\end{lemma}

\begin{proof}
By construction, $R|_{12i}=\tilde R|_{12i}$ is acyclic with elbow at~$2$ and large weights.
It~follows from Lemma~\ref{lem:3VertexAcyclicAscents} that every mutation %(at~1 or at~2)
directed away from~$R$ along the right rim of Figure~\ref{fig:GeneralSemiCycles-notation}
ascends~the subquiver on the vertices $1,2,i$. 
Hence going in the opposite direction gives a descent. 

By Lemma~\ref{lem:Subquiver12nStart}, $Q^{(1)}|_{12n}$ is acyclic with elbow at~$1$ and large weights. 
By Lemma~\ref{lem:3VertexAcyclicAscents}, every mutation %(at~1 or at~2)
directed towards~$R$ along the right rim of Figure~\ref{fig:GeneralSemiCycles-notation}
ascends~the subquiver on the vertices $1,2,n$ and makes this subquiver cyclic. 
Since $1 \points 2$ in $R$, the last claim follows. 
\end{proof}

%\textbf{Step 2:}
\begin{lemma}
\label{lem:binRQ}
We have 
\begin{equation}
\label{eq:binRQ}
Q^{(1)}|_{[3,n]} =Q^{(2)}|_{[3,n]} = \cdots =Q^{(2k+1)}|_{[3,n]} = R|_{[3,n]}.
\end{equation}
\end{lemma}

\begin{proof}%[Proof of Lemma~\ref{lem:binRQ}]
Let $1\le \ell\le 2k$ and $3\le i<j\le n$. 
By Lemma~\ref{lem:mut-epsilon},
$\varepsilon(\ell)$ is an ascent or descent in both $Q^{(\ell)}|_{12i}$ and $Q^{(\ell)}|_{12j}$.
Then Lemma~\ref{lem:4VertexWeightChanges} implies that $b_{ij}$ is unchanged by the 
mutation $Q^{(\ell)}\mutation{\varepsilon(\ell)}Q^{(\ell+1)}$. 
The claim follows. 
\end{proof}

%\textbf{Step 3:}
\label{Sec:Sinks}

\begin{lemma}
\label{lem:Sinks}
Every mutation $\mu[i]$ with $i\ge 3$ in Figure~\ref{fig:GeneralSemiCycles-notation}
is a sink mutation (assuming we are moving clockwise  along the top or bottom rim). 
\end{lemma}

We note that in Figure~\ref{fig:GeneralSemiCycles-notation},
there is exactly one mutation $\mu[i]$ for every $i\ge 3$. 

\begin{proof}%[Proof of Lemma~\ref{lem:Sinks}]
The mutation $Q\mutation{n}Q^{(1)}$ is a sink mutation by construction. 

Let us examine the mutations $\mu[\ell]$ ($3\le \ell \le n-1$) along the bottom rim:
\begin{equation}
\label{eq:bottom(n-2)}
Q^{(2k+n-2)} \mutation{3} Q^{(2k+n-3)} \mutation{4} \cdots \mutation{n-2}
         Q^{(2k+2)} \mutation{n-1} Q^{(2k+1)} =R. 
\end{equation}
Recall that $\tilde R=R|_{[1,n-1]}$ is acyclic with $b_{ij}(\tilde R)>0$ for $i<j$,
so all such mutations are sink mutations 
within the subquiver on $[1,n-1]$, cf.\ Proposition~\ref{pr:acyclic-mutation-cycle}. 

By Lemma~\ref{lem:binRQ}, for $i \in [3,n-1]$, we have~$b_{in}(R) = b_{in}(Q^{(1)}) = - b_{in}(Q) < 0$,
so~$R|_{[3,n]}$ is acyclic with linear order~$n\to 3\to 4\to\cdots\to n-1$. 
Consequently, all mutations~$\mu[\ell]$ ($3\le \ell \le n-1$) are sink mutations 
within the subquiver on~$[3,n]$, cf.~Proposition~\ref{pr:acyclic-mutation-cycle}. 
Putting everything together, we see that all mutations~$\mu[\ell]$~(${3\le \ell \le n-1}$) are sink mutations. 
\end{proof}

In order to prove that $Q$ and $Q'$ agree on their subquivers with vertices $1,2,n$
(see Lemma~\ref{lem:Subquiver12n}), we are going to completely describe all subquivers $Q^{(\ell)}_{12n}$, cf.\ Figure~\ref{fig:Q12nShapes}. 

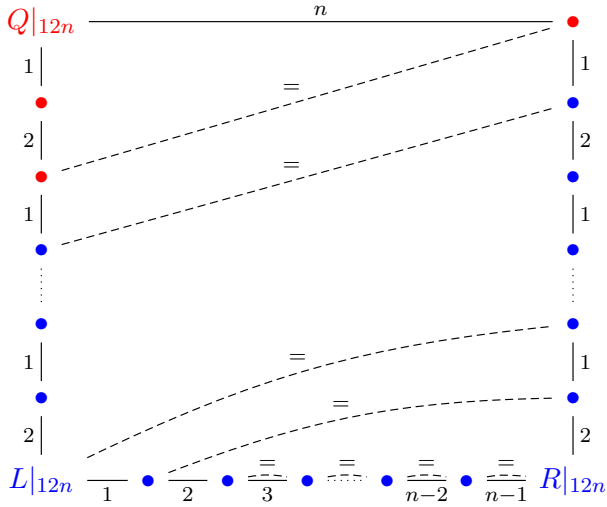
\begin{figure}[ht]
\begin{equation*}
\begin{tikzcd}[arrows={-stealth, cramped}, sep=15]
\red{Q|_{12n}}  \arrow[rrrrrr, no head, "n"] \arrow[d, swap, no head, "1"]
  &&&&&& \red{\bullet} \arrow[d, no head, "1"] \\
\red{\bullet} \arrow[d, swap, no head,"2"] &&&&&&  \blue{\bullet} \arrow[d, no head,"2"]  \\
\red{\bullet} \arrow[d, swap, no head,"1"] \arrow[uurrrrrr, dashed, no head,"="] &&&&&&  \blue{\bullet} \arrow[d, no head,"1"]  \\
%\blue{\bullet} \arrow[d, swap, no head,"2"] \arrow[uurrrrrr, dashed, no head,"="] &&&&&&  \blue{\bullet} \arrow[d, no head,"2"]  \\
\blue{\bullet} \arrow[d, swap, no head,dotted] \arrow[uurrrrrr, dashed, no head,"="] &&&&&&  \blue{\bullet} \arrow[d, no head,dotted]  \\
\blue{\bullet} \arrow[d, swap, no head,"1"] &&&&&&  \blue{\bullet} \arrow[d, no head,"1"]  \\
\blue{\bullet} \arrow[d, swap, no head,"2"] &&&&&&  \blue{\bullet} \arrow[d, no head,"2"]  \\
\blue{L|_{12n}} \arrow[r, swap, no head, "1"] \arrow[uurrrrrr, dashed, bend left=10pt, no head,"="] 
    & \blue{\bullet} \arrow[r, swap, no head, "2" ] \arrow[urrrrr, dashed, bend left=10pt, no head,"="]
%    & \blue{\bullet} \arrow[r, swap, no head,"3"] \arrow[rrrr, dashed, bend left=10pt, no head,"="]
%    & \blue{\bullet} \arrow[r, dotted, no head] & \blue{\bullet} \arrow[r, swap, no head,"n-2"] & \blue{\bullet} \arrow[r, swap, no head,"n-1"] & \blue{R|_{12n}}
   & \blue{\bullet} \arrow[r, swap, no head,"3"] \arrow[r, dashed, bend left=10pt, no head,"="]
   & \blue{\bullet} \arrow[r, dotted, no head] \arrow[r, dashed, bend left=10pt, no head,"="] 
   & \blue{\bullet} \arrow[r, swap, no head,"n-2"] \arrow[r, dashed, bend left=10pt, no head,"="] 
   & \blue{\bullet} \arrow[r, swap, no head,"n-1"] \arrow[r, dashed, bend left=10pt, in=175, no head,"="] 
   & \blue{R|_{12n}}
\end{tikzcd}
\end{equation*}
\caption{Subquivers on the vertices $1,2,n$ along the mutation cycle. 
Dashed lines indicate equality of these subquivers. 
A red dot $\red{\bullet}$ indicates that the subquiver on $1,2,n$ is acyclic; a blue dot $\blue{\bullet}$ indicates that it is cyclic.
}
\vspace{-.2in}
\label{fig:Q12nShapes}
\end{figure}

%\newpage
\begin{lemma}
\label{lem:Subquiver12nBottom}
We have 
$Q^{(2k+n-2)}|_{12n}=\cdots =
Q^{(2k+2)}|_{12n} = Q^{(2k+1)}|_{12n} = R|_{12n}$. 
\end{lemma}

Note that these are precisely the subquivers of the quivers appearing in~\eqref{eq:bottom(n-2)}. 

\begin{proof}
By Lemma~\ref{lem:Sinks}, each mutation at $3,\dots,n-2,n-1$ is a sink mutation.
As such, it does not change the subquiver supported on the remaining vertices $1,2,n$. 
%These are exactly the mutations between $Q^{(2k+1)}$ and $Q^{(2k+n-2)}$.
\end{proof}

\begin{lemma}
\label{lem:Subquiver12nLeft}
We have $Q^{(4k+n-2-j)}|_{12n} = Q^{(j+1)}|_{12n}$ for $ 0 \leq j \leq 2k$.
\end{lemma}

\begin{proof}
%We argue by descending induction on~$j$. 
In the case $j=2k$, the claim
$Q^{(2k+n-2)}|_{12n}=Q^{(2k+1)}|_{12n}$ holds by Lemma~\ref{lem:Subquiver12nBottom}. %and then do regular induction on \ell. 
Applying mutations at $2, 1, 2,\dots$, we obtain the desired identities for all~$j$. 
\end{proof}

\begin{lemma}
\label{lem:Subquiver12n}
We have $Q|_{12n} = Q'|_{12n}$.
\end{lemma}

\begin{proof}
By Lemma~\ref{lem:Subquiver12nLeft} (with $j=0$), we have $Q^{(4k+n-2)}|_{12n} = Q^{(1)}|_{12n}$. By Lemma~\ref{lem:Subquiver12nStart}, the subquiver~$Q^{(1)}|_{12n}$ lies on the mutation 3-cycle 
\begin{equation*}
Q|_{12n} \mutation{n} Q^{(1)}|_{12n} = Q^{(4k+n-2)}|_{12n}
\mutation{2} Q^{(4k+n-1)}|_{12n} \mutation{1} Q'|_{12n}. \qedhere
\end{equation*}
\end{proof}

We next examine the entries $b_{in}(Q^{(j)})$ for $i\in [3,n-1]$. 
Lemma~\ref{lem:fourChanges} describes the evolution of these entries under the mutations
\begin{align*}
Q^{(2k+n-2)} \mutation{2}Q^{(2k+n-1)}\mutation{1}Q^{(2k+n)} \\
Q^{(4k+n-2)} \mutation{2}Q^{(4k+n-1)}\mutation{1}Q^{(4k+n)}
\end{align*}
located near the bottom left and top left corners of %left end of the bottom rim and at the top end of the left rim in 
Figure~\ref{fig:GeneralSemiCycles-notation}.

\newcommand{\xv}{a}
\newcommand{\yv}{x}
\newcommand{\zv}{y}

\begin{lemma}
\label{lem:fourChanges}
Set $\xv\!=\!b_{12}(Q)\!=\!b_{12}(R), \yv\!=\!b_{1n}(Q)$, and $\zv\!=\!b_{2n}(Q)$. 
For $i\!\in\! [3,n-1]$,  
\begin{align}
    \label{eq:change1}
    b_{in}(Q^{(2k+n-1)}) - b_{in}(Q^{(2k+n-2)})  &= b_{2i}(R) (\h_{2k-1}(\xv) \yv + \h_{2k-2}(\xv) \zv), \\
%\end{equation}
% $(\T{1}{L}|_{12n}$ $i \rightarrow 1 \rightarrow n$,
%\begin{equation}
    \label{eq:change2}
    b_{in}(Q^{(2k+n)}) - b_{in}(Q^{(2k+n-1)}) &= b_{1i}(R) (\h_{2k-2}(\xv) \yv + \h_{2k-3}(\xv) \zv), \\
%\end{equation}
%\begin{equation}
    \label{eq:change3}
    b_{in}(Q^{(4k+n-1)}) - b_{in}(Q^{(4k+n-2)}) &= -\zv (\h_{2k-2}(\xv) b_{2i}(R) + \h_{2k-3}(\xv) b_{1i}(R)), \\
%\end{equation}
%\begin{equation}
    \label{eq:change4}
    b_{in}(Q') - b_{in}(Q^{(4k+n-1)}) &= -\yv (\h_{2k-2}(\xv) b_{1i}(R) + \h_{2k-1}(\xv) b_{2i}(R)).
\end{align}
\end{lemma}

We note that the numbers $b_{12}(R), b_{1i}(R), b_{2i}(R)$ and $b_{1n}(Q), b_{2n}(Q), b_{in}(Q)$ 
are all at least $2$ (hence positive) by construction.

\begin{proof}
Let us prove the formula~\eqref{eq:change1}; 
the proofs of  \eqref{eq:change2}--
\eqref{eq:change4} are analogous.

As $Q^{(2k+n-2)} \mutation{2} Q^{(2k+n-1)}),$ Lemma~\ref{lem:Subquiver12nBottom} implies $b_{2n}(Q^{(2k+n-2)}) = b_{2n}(R|_{12n}).$ 
Thus by applying Proposition~\ref{pr:3VertexAlternatingMutations} to $Q^{(1)}|_{12n}$ (with $n$ relabeled to $3$):
\[
b_{2n}(Q^{(2k+n-2)}) =b_{2n}(R|_{12n}) = b_{2n}(\T{(21)^k}{Q^{(1)}} = \h_{2k-1}(\xv) \yv + \h_{2k-2}(\xv) \zv > 0
\]
Lemma~\ref{lem:SubquiverTildeQ} implies that $b_{2i}(Q^{(2k+n-2)}|_{[1,n-1]}) = b_{2i}(\T{21}{\tilde R}).$ 
Since both $1$ and $2$ are source mutations in $\tilde R$, we have 
$b_{2i}(Q^{(2k+n-2)}) = -b_{2i}(\tilde R) <0$.
Therefore $i \points 2 \points n$ in $Q^{(2k+n-2)}$ and we compute:
\begin{align*}
b_{in}(Q^{(2k+n-1)}) - b_{in}(Q^{(2k+n-2)}) &=b_{i2}(Q^{(2k+n-2)}) b_{2n}(Q^{(2k+n-2)}) \\
%&= b_{2i}(R) b_{2n}(R)\\
&=b_{2i}(R) (\h_{2k-1}(\xv) \yv + \h_{2k-2}(\xv) \zv).
\end{align*}

To show \eqref{eq:change3}, \eqref{eq:change4}, one should apply Proposition~\ref{pr:3VertexAlternatingMutations} to relabelings of the quivers $R|_{12i}$ instead of $Q^{(1)}|_{12n}$.
\end{proof}

\begin{lemma}
\label{lem:binUnchanged}
We have $Q^{(2k+n)}|_{[3,n]} = Q^{(2k+n+1)}|_{[3,n]} = \cdots = Q^{(4k+n-2)}|_{[3,n]}$. 
\end{lemma}
\begin{proof}
Let $2 \leq \ell \leq 2k$ and $3 \leq i < j \leq n$.  
By Equation~\eqref{eq:SubquiverTildeQ-2} (proved in Lemma~\ref{lem:SubquiverTildeQ}) and Lemma~\ref{lem:Subquiver12nLeft}, 
$Q^{(4k+n-\ell)}|_{12i}$ and $Q^{(4k+n-\ell)}|_{12j}$ correspond to subquivers with the same support in either $Q^{(\ell+1)}$ or $Q^{(\ell-1)}$. %Which Q^(ell) depends on if j is n or not.
Lemma~\ref{lem:mut-epsilon} implies that the mutation~$\varepsilon(\ell)$ is an ascent or descent in both of these subquivers. %the mutation $\varepsilon(\ell)$ is an ascent or descent in both $Q^{(4k+n-\ell)}|_{12i}$ and $Q^{(4k+n-\ell)}|_{12j}$. 
Thus Lemma~\ref{lem:4VertexWeightChanges} implies that~$b_{ij}$ is unchanged by the mutation 
$Q^{(4k+n-\ell)}~\mutation{\varepsilon(\ell)}~Q^{(4k+n+1-\ell)}$. 
\end{proof}

\begin{lemma}
\label{lem:bin}
We have $b_{in}(Q) = b_{in}(Q')$ for $3 \leq i \leq n-1$.
\end{lemma}
\begin{proof}
We have
\[
b_{in}(Q) =-b_{in}(Q^{(1)}) = -b_{in}(R) = b_{in}(Q^{(2k+n-2)})
\]
by Lemmas~\ref{lem:binRQ},~\ref{lem:Sinks}.
% (the only sink mutations which change $b_{in}$ are the mutations at $i$ and $n$). 
Further we can rewrite:
\begin{align*}
b_{in}(Q') - b_{in}(Q^{(2k+n-2)}) &= (b_{in}(Q') - b_{in}(Q^{(4k+n-2)})) \\
&\quad + (b_{in}(Q^{(4k+n-2)}) - b_{in}(Q^{(2k+n)})) \\
&\quad + (b_{in}(Q^{(2k+n)}) - b_{in}(Q^{(2k+n-2)})).
\end{align*}
Lemma~\ref{lem:binUnchanged} implies $b_{in}(Q^{(4k+n-2)})\! =\! b_{in}(Q^{(2k+n)})$. 
Summing equations~\eqref{eq:change1}--\eqref{eq:change4} from Lemma~\ref{lem:fourChanges} implies that 
\begin{align*}
    b_{in}(Q') - b_{in}(Q^{(2k+n-2)}) &=  b_{2i}(R) (\h_{2k-1}(b_{12}(R)) b_{1n}(Q) + \h_{2k-2}(b_{12}(R)) b_{2n}(Q)\\
    &\quad + b_{1i}(R) (\h_{2k-2}(b_{12}(R)) b_{1n}(Q) + \h_{2k-3}(b_{12}(R)) b_{2n}(Q))\\
    &\quad - b_{2n}(Q) (\h_{2k-2}(b_{12}(R)) b_{2i}(R) + \h_{2k-3}(b_{12}(R)) b_{1i}(R))\\
    &\quad - b_{1n}(Q) (\h_{2k-2}(b_{12}(R)) b_{1i}(R) + \h_{2k-1}(b_{12}(R)) b_{2i}(R))\\
    &=  b_{1i}(R)b_{1n}(Q) (\h_{2k-2}(b_{12}(R)) -  \h_{2k-2}(b_{12}(R)))\\
    &\quad + b_{1i}(R)b_{2n}(Q) (\h_{2k-3}(b_{12}(R)) - \h_{2k-3}(b_{12}(R))) \\
    &\quad + b_{2i}(R)b_{1n}(Q) (\h_{2k-1}(b_{12}(R)) - \h_{2k-1}(b_{12}(R))) \\
    &\quad + b_{2i}(R)b_{2n}(Q) (\h_{2k-2}(b_{12}(R)) - \h_{2k-2}(b_{12}(R))) \\
    &=  0, 
\end{align*}
as claimed.
\end{proof}

\begin{proof}[Proof of Theorem~\ref{thm:GeneralSemiCycles}]
We have $b_{ij}(Q') = b_{ij}(Q)$ for all $i<j;$ first for $i,j < n$ in Lemma~\ref{lem:SubquiverTildeQ}, % (essentially by construction), 
then for $i=1,2$ and $j=n$ in Lemma~\ref{lem:Subquiver12n}, % (where we used that several mutations were sink mutations), 
and finally by checking for $i>2$ and $j=n$ in Lemma~\ref{lem:bin}. %(which involved both the previous results, and an explicit computation of some of the weights).
\end{proof}

\newpage

\section{Distinctness}
\label{sec:distinctness}

In this section, we show that each mutation cycle described in Theorem~\ref{thm:GeneralSemiCycles} 
does not visit the same quiver more than once. 

\begin{definition}
\label{def:minimal-cycle}
A~mutation cycle is called \emph{simple} if any two quivers lying on the cycle are distinct. 
\end{definition}

\begin{theorem}
\label{thm:GeneralSemiCycleDistinct}
%If $k>0$, 
The mutation cycle in Theorem~\ref{thm:GeneralSemiCycles} is simple. 
\end{theorem}

\begin{theorem}
\label{thm:GeneralSemiCycleDistinct-noniso}
Suppose that either $n\ge 5$ or else $n=4$ and $b_{14}(Q) \neq b_{23}(R)$. 
Then no two quivers on the mutation cycle described in Theorem~\ref{thm:GeneralSemiCycles} 
are related by an isomorphism and/or global reversal of arrows.
\end{theorem}

Our proofs of Theorems \ref{thm:GeneralSemiCycleDistinct}--\ref{thm:GeneralSemiCycleDistinct-noniso} 
will rely on the descriptions, 
given in Lemmas~\ref{lem:orientations}--\ref{lem:descents} below, of 
the various 3-vertex subquivers~$Q^{(j)}|_{pqr}$ of the quivers $Q^{(j)}$ (cf.\ \eqref{eq:Q123}). 
These lemmas, in turn, summarize the results obtained in Lemmas~\ref{lem:R12nSubquiv}--\ref{lem:Q[3,n]Subquiv}. 

\begin{definition}
For an $n$-vertex quiver $Q$ and a permutation~$\sigma$ in the symmetric group~$S_n$, 
we will denote by $\sigma(Q)$ the result of relabeling the vertices of~$Q$ 
by~$\sigma$.  
\end{definition}

For example, if $\sigma\in S_3$ is defined by $\sigma(1)=2$, $\sigma(2)=3$, $\sigma(3)=1$,
then $\sigma(R)=R'$, where $R$ and~$R'$ are the quivers from Example~\ref{eg:acyclicClasses}.

\begin{remark}
%\label{rem:}
See Example~\ref{eg:4vertexSymmetric} for why the condition $b_{14}(Q) \neq b_{23}(R)$ 
is included in the $n=4$ case of 
Theorem~\ref{thm:GeneralSemiCycleDistinct-noniso}.
This condition is guaranteed by 
choosing a value $b_{14}(Q)\ge 2$ different from $b_{23}(\tilde R)$. 
In fact, this restriction can be further relaxed. 
Let $\sigma \in S_4$ be the permutation defined by 
${\sigma(1) = 2}$, ${\sigma(2)=1}$, ${\sigma(3)=4}$, ${\sigma(4)=3}$.  
Then the conclusion of 
Theorem~\ref{thm:GeneralSemiCycleDistinct-noniso} holds whenever reversing the arrows of $\sigma(Q)$ gives a quiver different from~$Q^{(2k+2)}$. 
\end{remark}

\begin{lemma}
\label{lem:R12nSubquiv}
Each 3-vertex quiver $Q^{(j)}|_{12n}$, for $0\le j\le 4k+n$, is mutation equivalent to an acyclic quiver with large weights. 
More specifically, the descent sequences of the quivers $Q^{(j)}|_{12n}$ are determined from: 
%(move left-to-right in either \eqref{eq:R12nDesc} or \eqref{eq:Lp12nDesc}): 
\begin{align}
\label{eq:R12nDesc}
& R|_{12n} = Q^{(2k+1)}|_{12n} \mutation{2} Q^{(2k)}|_{12n} \mutation{1} \cdots \mutation{1} Q^{(1)}|_{12n} = 
{\begin{tikzcd}[arrows={-stealth}, sep=small, ampersand replacement=\&]
  \scriptstyle 1 \arrow[r, dashed] 
  \& \scriptstyle 2 
  \\
   \& \scriptstyle n \arrow[u, dashed] \arrow[lu, dashed]
\end{tikzcd}},
\\
\label{eq:C12n}
& Q^{(2k+1)}|_{12n}  = Q^{(2k+2)}|_{12n} = \cdots = Q^{(2k+n-2)}|_{12n} , 
\\
\label{eq:Lp12nDesc}
& Q^{(2k+n-2)}|_{12n} \mutation{2} Q^{(2k+n-1)}|_{12n} \mutation{1} \cdots \mutation{1} Q^{(4k+n-2)}|_{12n} = 
\begin{tikzcd}[arrows={-stealth}, sep=small, ampersand replacement=\&]
  \scriptstyle 1 \arrow[r, dashed] 
  \& \scriptstyle 2 
  \\
   \& \scriptstyle n \arrow[u, dashed] \arrow[lu, dashed]
\end{tikzcd}, \\
\label{eq:Top12n}
& \text{$Q^{(4k+n-1)}|_{12n}$ and~$Q^{(0)}|_{12n}=Q^{(4k+n)}|_{12n}$ are acyclic.}
\end{align}
If $Q^{(j)}|_{12n}$ is cyclic, then its descent sequence appears as a subsequence of consecutive mutations 
%read from left-to-right, 
in either \eqref{eq:R12nDesc} or~\eqref{eq:Lp12nDesc}. 
(For $2k+1\le j\le 2k+n-2$, first use \eqref{eq:C12n}.) 
\end{lemma}

To illustrate, $Q^{(2k)}|_{12n}$ has descent sequence $1(21)^{k-1}$.

\pagebreak[3]

\begin{proof}

Lemma~\ref{lem:Subquiver12nStart} states $Q^{(1)}|_{12n}$ is acyclic with large weights and with elbow at~$1$. 
Thus Lemma~\ref{lem:3VertexAcyclicAscents-1} implies~\eqref{eq:R12nDesc}.
Now \eqref{eq:C12n} follows from Lemma~\ref{lem:Subquiver12nBottom}. 
(Thus sequences of quivers~\eqref{eq:R12nDesc} and~\eqref{eq:Lp12nDesc} are identical to each other.)

By Lemma~\ref{lem:Subquiver12nStart}, vertex $1$ is a source in $Q|_{12n}$, and \eqref{eq:Top12n} follows. 
%so the quivers $Q|_{12n}$ and $Q^{(4k+n-1)}|_{12n}$ are acyclic.
\end{proof}

\begin{lemma}
\label{lem:Q12iSubquiv}
Each 3-vertex quiver $Q^{(j)}|_{12i}$, for $0\le j\le 4k+n$ and $3 \le i \le n-1$, is mutation equivalent to an acyclic quiver with large weights. 
More specifically, the descent sequences of the quivers $Q^{(j)}|_{12i}$ are determined from:
\begin{align}
\label{eq:Qone12iDesc}
& Q^{(1)}|_{12i} \mutation{1} Q^{(2)}|_{12i} \mutation{2} \cdots \mutation{2} Q^{(2k+1)}|_{12i} = R|_{12i}= 
\begin{tikzcd}[arrows={-stealth}, sep=small, ampersand replacement=\&]
  1 \arrow[r, dashed] \arrow[rd, dashed]
  \& 2 \arrow[d, dashed]
  \\
   \& i
\end{tikzcd},
\\
\label{eq:C12i}
& \text{$Q^{(2k+1)}|_{12i}, \ldots, Q^{(2k+n)}|_{12i}$ are acyclic,} 
\\
\label{eq:Q12iDesc}
& Q^{(0)}|_{12i} \mutation{1} Q^{(4k+n-1)}|_{12i} \mutation{2} \cdots \mutation{2} Q^{(2k+n)}|_{12i} = L|_{12i}= 
\begin{tikzcd}[arrows={-stealth}, sep=small, ampersand replacement=\&]
  1 \arrow[r, dashed] \arrow[rd, dashed]
  \& 2 \arrow[d, dashed]
  \\
   \& i
\end{tikzcd}.
\end{align}
If $Q^{(j)}|_{12i}$ is cyclic, then its descent sequence appears as a subsequence of consecutive mutations in either~\eqref{eq:Qone12iDesc} or~\eqref{eq:Q12iDesc}.
In particular, $Q^{(0)}|_{12i}$ has descent sequence $(12)^{k}$.
\end{lemma}

\begin{proof}
By construction, $R|_{12i}$ is acyclic with large weights and with elbow at $2$. % for~${3 \le i \le n-1}$. 
Thus Lemma~\ref{lem:3VertexAcyclicAscents-1} implies~\eqref{eq:Qone12iDesc}.
Now \eqref{eq:C12i} follows from Lemma~\ref{lem:Subquiver12nBottom}. 
Since $n$ is a sink in $Q^{(0)},$ we have $Q^{(0)}|_{12i} = Q^{(1)}|_{12i}$ and so the sequences of quivers~\eqref{eq:Qone12iDesc} and~\eqref{eq:Q12iDesc} are identical to each other.
\end{proof}

\begin{lemma}
\label{lem:Q1inOr2inSubquiv}
Let $3\!\le i \!\le\! n-1$ and $0\!\le\! j\!\le\! 4k+n$. 
Then the quivers $Q^{(j)}|_{1in}$ and~$Q^{(j)}|_{2in}$ are mutation equivalent to acyclic quivers with large weights. 
More specifically, 
\begin{itemize}[leftmargin=.2in]
\item
if $j\notin \{2k+n,4k+n-1\}$, then $Q^{(j)}|_{1in}$ is acyclic; 
\item
if $j\!\in\! \{2k+n,4k+n-1\}$, then $Q^{(j)}|_{1in}$ is mutation acyclic with one-term descent sequence~$(1)$;
\item
if $j\notin \{2k+n-1,4k+n-2\}$, then $Q^{(j)}|_{2in}$ is acyclic; 
\item
if $j\in \{2k+n-1,4k+n-2\}$, then $Q^{(j)}|_{2in}$  is mutation acyclic with one-term descent sequence~$(2)$. 
\end{itemize}
\end{lemma}

\begin{proof}
By Lemma~\ref{lem:fourChanges}, all four quivers $Q^{(2k+n)}|_{1in}$, $Q^{(4k+n-1)}|_{1in}$, $Q^{(2k+n-1)}|_{2in}$, $Q^{(4k+n-2)}|_{2in}$ are cyclic with descents at $1$, $1$, $2$, and~$2$, respectively. 
We will next check that $Q^{(j)}|_{1in}$ and $Q^{(j)}|_{2in}$ are acyclic otherwise, 
which will imply that these descents are actually the entire descent sequences.

%1in and 2in shared
Vertex $n$ is a sink or source in each of $Q|_{1in}$, $Q|_{2in}$, $Q^{(1)}|_{1in}$, and $Q^{(1)}|_{2in}$, so all four quivers are acyclic. %Since $n$ is a sink/source in $Q,$ both $Q^{(1)}|_{1in}$ and $Q^{(1)}|_{2in}$ are acyclic too. 
In $R,$ we have $1 \points i$ and $2 \points i$, while $n \points i$ by Lemma~\ref{lem:binRQ}. 
Thus $R|_{1in}$ and $R|_{2in}$ are acyclic. 
It then follows by Lemma~\ref{lem:Sinks} that $Q^{(j)}|_{1in}$ and $Q^{(j)}|_{2in}$ are acyclic for~${2k+1 \le j \le 2k+n-2.}$
By inspection, at most two subquivers of a $4$-vertex quiver can be cyclic. 
By Lemmas~\ref{lem:R12nSubquiv}--\ref{lem:Q12iSubquiv}, both $Q^{(j)}|_{12i}$ and $Q^{(j)}|_{12n}$ are cyclic if ${2\le j \le2k}$ or~${2k+n+1 \le j \le 4k+n-3}$. 
So $Q^{(j)}|_{12in}$ has two cyclic $3$-vertex subquivers and therefore both~$Q^{(j)}|_{1in}$ and $Q^{(j)}|_{2in}$ are acyclic if
${0 \le j \le 2k}$ or ${2k+n+1 \le j \le 4k+n-3}$.

%1in finish
By Lemma~\ref{lem:binRQ}, we have $n \points i$ in~$R$. After applying a sink mutation at each vertex in $[2,n-\!1],$ we find that $i \points n$ in~$Q^{(2k+n-1)}$. 
Also, since $1$ is a source in $L|_{[1,n-1]}$ by Lemma~\ref{lem:SubquiverTildeQ}, $1$ is a sink in~$Q^{(2k+n-1)}|_{[1,n-1]}$. 
Thus $i \points 1$ and $i \points n,$ so $Q^{(2k+n-1)}|_{1in}$ is acyclic. 
Likewise, in $Q^{(4k+n-2)},$ Lemma~\ref{lem:3VertexAcyclicAscents} applied to $L|_{12i}$ implies that $i \points 1 \points 2$ %we mutate \T{(12)^k}
while Lemma~\ref{lem:Subquiver12nStart} implies that $1$ is an elbow in $Q^{(4k+n-2)}|_{12n}$, %we mutate at 1,2.
so~$n \points 1$. Thus $Q^{(4k+n-2)}|_{1in}$ is acyclic too.

%2in finish
The case of the quivers $Q^{(j)}|_{2in}$ is handled similarly. 
By Lemma~\ref{lem:binRQ}, $n \points i$ in~$R$. After applying a sink mutation at each vertex in ${[3,n-1]},$ we find that $i \points n$ in $Q^{(2k+n-2)}$ and $2$ is a sink in~$Q^{(2k+n-2)}|_{[1,n-1]}$. 
Thus ${i \points 2},$ so $Q^{(2k+n-2)}|_{2in}$ is acyclic. 
Likewise, in $Q^{(4k+n-1)},$ Lemma~\ref{lem:3VertexAcyclicAscents} applied to $L|_{12i}$ implies that $i \points 2 \points 1$ while Lemma~\ref{lem:Subquiver12nStart} implies that $2$ is a source in $Q^{(4k+n-2)}|_{12n}$, so~$n \points 2$. 
Thus $Q^{(4k+n-1)}|_{2in}$ is acyclic too.
\end{proof}

\begin{lemma}
\label{lem:Q[3,n]Subquiv}
For any~$j$, the subquivers $Q^{(j)}|_{[3,n]}$, $Q^{(j)}|_{\{1\}\cup[3,n-1]}$, and $Q^{(j)}|_{\{2\}\cup[3,n-1]}$ are acyclic,
with large weights. 
\end{lemma}

\begin{proof}
%[3,n]
{\ }

\noindent
\textbf{Case~1:} subquivers $Q^{(j)}|_{[3,n]}$. 
For ${1\le j\le 4k+n-2}$, 
Lemmas~\ref{lem:binRQ},\ref{lem:Sinks} and \ref{lem:binUnchanged} imply that the quiver $Q^{(j)}|_{[3,n]}$ is acyclic. 
As $Q^{(1)}|_{[3,n]}$ is acyclic and $n$ is a sink, $Q|_{[3,n]}$ is acyclic. 
The remaining case of $Q^{(4k+n-1)}|_{[3,n]}$ follows from Lemma~\ref{lem:fourChanges}. 

%forced by cyclics. 1ij and 2ij
\noindent
\textbf{Case~2:} subquivers $Q^{(j)}|_{\{1\}\cup[3,n-1]}$ and $Q^{(j)}|_{\{2\}\cup[3,n-1]}$. 
Lemmas~\ref{lem:R12nSubquiv}--\ref{lem:Q12iSubquiv} imply that 
for $\ell \in[i+1,n-1]$ and ${j \in[0,2k] \cup [2k+n+1, 4k+n]}$, 
both $Q^{(j)}|_{12i}$ and $Q^{(j)}|_{12\ell}$ are cyclic. 
Since a $4$-vertex quiver can have at most two cyclic $3$-vertex subquivers, 
both $Q^{(j)}|_{1i\ell}$ and $Q^{(j)}|_{2i\ell}$ are acyclic. 
Since $R|_{[1,n-1]}$ is acyclic by construction, Lemma~\ref{lem:Sinks} implies that $Q^{(j)}|_{1ij}$ and $Q^{(j)}|_{2ij}$ are acyclic for all~$j$. 
Given that we have already shown that $Q^{(j)}|_{[3,n]}$ is acyclic, it follows that $Q^{(j)}|_{[3,n-1]}$ is acyclic. 
Thus both $Q^{(j)}|_{\{1\}\cup[3,n-1]}$ and~$Q^{(j)}|_{\{2\}\cup[3,n-1]}$ are acyclic.
\end{proof}

We summarize some useful consequences of Lemmas~\ref{lem:R12nSubquiv}--\ref{lem:Q[3,n]Subquiv}
in Lemmas \ref{lem:orientations}--\ref{lem:cycleIsSinksAndDescents} below.

\begin{lemma}
\label{lem:orientations}
Let $0\le j\le 4k+n$. 
Any $3$-vertex subquiver of $Q^{(j)}$ is mutation equivalent to an acyclic quiver with large weights. 
Such a 3-vertex subquiver is cyclic if and only if it appears on the list below (here $i\in [3,n-1]$):
\begin{itemize}[leftmargin=.2in]
\item $Q^{(j)}|_{12n}$, for $j \in [2, 4k+n-3]$;
\item $Q^{(j)}|_{12i}$, for $j \not \in [2k+1, 2k+n]$;
\item $Q^{(j)}|_{1in}$, for $j \in \{2k+n, 4k+n-1\}$; 
\item $Q^{(j)}|_{2in}$, for $j \in \{2k+n-1, 4k+n-2\}$. 
\end{itemize}
\end{lemma}

\begin{lemma}
\label{lem:Qj-large}
Each quiver $Q^{(j)}$ has large weights.
\end{lemma}

\begin{lemma}
\label{lem:descents}
Let $0\le j \le 4k+n$. Then: 

\noindent
1. The descent sequence of each $3$-vertex subquiver of $Q^{(j)}$ consists of $1$'s and~$2$'s. 

\noindent
2. Both $1$ and $2$ appear within descent sequences of $3$-vertex subquivers of~$Q^{(j)}$. 

\noindent
3.  Vertex $1$ is a descent of some $3$-vertex subquiver of $Q^{(j)}$ if and only if $j \!\notin\! [{2k+1}, \linebreak[3]
2k+n-2]$. 

\noindent
4. Vertex $2$ is a descent of some $3$-vertex subquiver of $Q^{(j)}$ if and only if $j \not \in \{0, 1\}$.
\end{lemma}

\pagebreak[3]

\begin{lemma}
\label{lem:no2cycles}
Let $ 0 \leq j < 4k+n$. 
If $\T{v}{Q^{(j)}}=\T{u}{Q^{(j)}}$, then $u=v$.
\end{lemma}
\begin{proof}
Suppose $u\ne v$. Pick a vertex $i \notin \{u,v\}$.  
By Lemma~\ref{lem:orientations}, $Q^{(j)}|_{iuv}$ is mutation equivalent to an acyclic quiver with large weights.
But $\T{v}{Q^{(j)}|_{iuv}} = \T{u}{Q^{(j)}|_{iuv}}$,  contradicting Lemma~\ref{lem:3VertexAcyclicAscents-1}.
\end{proof}

\begin{lemma}
\label{lem:cycleIsSinksAndDescents}
For any $0 \le j < 4k+n$ and $1 \le v \le n$, the following are equivalent:

\begin{itemize}[leftmargin=.55in]
\item[\rm(\ref{lem:cycleIsSinksAndDescents}a)]
$\T{v}{Q^{(j)}} = Q^{(j \pm 1)}$ (with the superscript $j\pm 1$ taken modulo $4k+n$);
\item[\rm(\ref{lem:cycleIsSinksAndDescents}b)]
one of the following conditions holds: 
\begin{itemize}[leftmargin=.2in]
\item vertex $v$ is a sink/source in~$Q^{(j)}$; 
\item vertex $v$ is a descent of some $3$-vertex subquiver of $Q^{(j)}$ and there is at most one sink/source in~$Q^{(j)}$.
\end{itemize}
\end{itemize}
\end{lemma}

\begin{proof}
The claim can be checked using Lemmas~\ref{lem:R12nSubquiv}--\ref{lem:Q[3,n]Subquiv}.
We note that by Lemma~\ref{lem:no2cycles}, $\T{v}{Q^{(j)}} = Q^{(j \pm 1)}$ implies that the mutation $Q^{(j)}\mutation{v}Q^{(j \pm 1)}$
lies on the mutation cycle  in Theorem~\ref{thm:GeneralSemiCycles} 
(so for example, if $j=1$, then $v=1$ or $v=n$). 
\end{proof}

\begin{lemma}
\label{lem:oneSinkSource}
The only quivers $Q^{(j)}$ with exactly one sink/source vertex are 
$Q$, $Q^{(1)}$, $R$, and~$Q^{(2k+n-2)}$. % \commentS{$= \T{21}{L}$.}
\end{lemma}

\begin{proof}
Lemma~\ref{lem:Sinks} implies that each of $Q$, $Q^{(1)}$, $R$, and $Q^{(2k+n-2)}$ has either a sink or a source. 
Lemma~\ref{lem:orientations} and the construction of $\tilde R$ imply that this sink/source is unique. 
More explicitly: 
\begin{itemize}[leftmargin=.25in]
\item[(a)]
in both $Q$ and~$Q^{(1)}$, vertex $n$ is a sink/source; 
every other vertex lies in some cyclic $3$-vertex subquiver; 
\item[(b)]
in $R$, vertex $n-1$ is the sink; 
vertex $1$ is the source of $\tilde R=R|_{[1,n-1]}$, but is not a source in~$R$ since 
$1$ is contained in the cyclic subquiver $R|_{12n}$; 
\item[(c)]
the same argument applies to $Q^{(2k+n-2)}$, with $1$ and $n-1$ replaced by $2$ and~$3$, respectively. 
\end{itemize}

Again by Lemmas~\ref{lem:Sinks} and~\ref{lem:orientations}, 
every quiver $Q^{(j)}$ other than $Q$, $Q^{(1)}$, $R$ or $Q^{(2k+n-2)}$ has either
both a sink and a source, or no sink/source. 
\end{proof}

\begin{lemma}
\label{lem:Q-neq-Qj}
For any $j\in [1,4k+n-1]$, we have $Q=Q^{(0)}\neq Q^{(j)}$. 
Also, $Q$ is distinct from $Q^{(j)}$ with all arrows reversed. 
\end{lemma}

\begin{proof}
Quiver $Q$ has a sink at $n$ and no sources. 
Therefore by Lemma~\ref{lem:oneSinkSource}, 
$Q$ is distinct from every quiver $Q^{(j)}$ (even allowing the reversal of all arrows)
except possibly $Q^{(1)}, R,$ and~$Q^{(2k+n-2)}$.  
However, $R$ has a sink at $n-1$, whereas $Q^{(2k+n-2)}$ has a source at~$3$. 
Finally, $Q^{(1)}_{12n}$ has an elbow at~1, whereas $Q_{12n}$ has an elbow~at~2.   
\end{proof}

\begin{lemma}
\label{lem:QDistinct} 
None of the quivers $Q^{(j)}$, for $1\le j\le 4k+n-1$, is isomorphic to~$Q$,
even if we allow the global reversal of arrows---unless $n=4$ and $b_{14}(Q)=b_{23}(R)$.
\end{lemma}

\begin{proof}
Any isomorphism, possibly involving a global reversal of arrows, must map sinks/sources to sinks/sources. 
Likewise, it must map the descent sequence (resp., the elbow) of a $3$-vertex subquiver 
to another descent sequence (resp., elbow). 

Lemma~\ref{lem:oneSinkSource} implies that if the quivers $Q$ and $Q^{(j)}$ are isomorphic, possibly with a global reversal of arrows, 
then the quiver $Q^{(j)}$ is equal to either $Q^{(1)}$, $R$, or~$Q^{(2k+n-2)}$. 
We consider each possibility in turn.

Suppose that $\sigma\in S_n$ is a permutation such that 
$\sigma(Q)$ coincides with $Q^{(1)}$ up to global reversal of arrows.
Then ${\sigma(n) = n}$, ${\sigma(1) = 1}$, and~${\sigma(2) = 2}$. 
By Lemma~\ref{lem:Subquiver12nStart}, the elbow of $Q|_{12n}$ is $2$, 
yet the elbow of $Q^{(1)}|_{12n}$ is $1,$ a contradiction.
%This holds regardless of a global reversal of arrows.

Lemma~\ref{lem:orientations} implies that $Q$ has $n-3$ cyclic subquivers, whereas both %, $Q|_{12i}$ for~$ 2 < i <n$. 
$R$ and $Q^{(2k+n-2)}$ have exactly one. 
Thus if $n\geq 5$, then $Q^{(j)}$ cannot be $R$ nor~$Q^{(2k+n-2)}$. 

If $n=4,$ then each of $Q$, $R$, and~$Q^{(2k+4-2)}=Q^{(2k+2)}$ have one sink/source and one cyclic subquiver.
The cyclic subquiver $Q|_{123}$ has descent $1$ while the cyclic subquiver $Q^{(2k+2)}|_{124} = R|_{124}$ has descent $2,$ 
and both $Q^{(2k+2)}$ and $R$ have a sink/source at~$3$.
So~the only permutation $\sigma\in S_4$ that could potentially work is given by 
$\sigma(1) = 2$, $\sigma(2)=1$, $\sigma(3)=4$, and $\sigma(4)=3$.  
But $2$ is an elbow in both $R|_{123}$ and $Q|_{124},$ so $Q^{(j)}$ is not $R$.
As $Q^{(2k+2)}$ has a source, $\sigma(Q)$ could only be $Q^{(2k+2)}$ after a global reversal of arrows. 
But then
\[
b_{23}(R) = -b_{23}(\T{3}{R}) = b_{23}(\sigma(Q)) = b_{14}(Q),
\]
as claimed. 
\end{proof}

\pagebreak[3]

\begin{lemma}
\label{lem:QAutomorphisms}
The quiver $Q=Q^{(0)}$ has no nontrivial automorphisms, even if we allow a global reversal of arrows.
\end{lemma}

\begin{proof}
The quiver $Q$ has a (unique) sink~$n$ and no sources. 
Therefore an isomorphism~$\sigma$ of~$Q$
%(potentially involving a global reversal of arrows) 
must leave the sink~$n$ in~$Q$ in place: $\sigma(n)=n$.
Also, there are no isomorphisms between $Q$ and the quiver obtained by reversing all arrows in~$Q$. 

By Lemma~\ref{lem:orientations}, the only cyclic 3-vertex subquivers of~$Q=Q^{(0)}$ are 
the quivers $Q|_{12i}$, for $i=3,\dots,n-1$;
their descent sequences, by Lemma~\ref{lem:Q12iSubquiv}, are of the form~$(12)^k$.  
It follows that ${\sigma(1)=1}$ and~${\sigma(2)=2}$.  

By Lemmas~\ref{lem:binRQ}--\ref{lem:Sinks}, the subquiver~$Q|_{[3,n-1]}$ is acyclic with orientations 
\begin{equation*}
3 \points 4 \points \cdots \points n-1.  
\end{equation*}
We conclude that $\sigma(i)=i$ for all~$i$.
\end{proof}

\begin{lemma}
\label{lem:isosMove}
Suppose that an isomorphism $\sigma$ sends $Q^{(\ell)}$ to $Q^{(j)}$ (or to $Q^{(j)}$ with the arrows reversed). 
Then one of the following statements holds: 
\begin{itemize}[leftmargin=.2in]
\item 
$\sigma$ sends $Q^{(\ell-m)}$ to $Q^{(j+m)}$, for all $m \geq 0$; 
\item 
$\sigma$ sends $Q^{(\ell-m)}$ to $Q^{(j+m)}$, with all the arrows reversed, for all $m \geq 0$; 
\item 
$\sigma$ sends $Q^{(\ell-m)}$ to $Q^{(j-m)}$, for all $m \geq 0$;
\item 
$\sigma$ sends $Q^{(\ell-m)}$ to $Q^{(j-m)}$, with all the arrows reversed, for all $m \geq 0$.
\end{itemize}
(Here the superscripts are taken modulo~$4k+n$.) 
In particular, taking $m=\ell$, we see that $\sigma$ sends $Q$ to $Q^{(j \pm \ell)}$, possibly with all arrows reversed.
\end{lemma}

\begin{proof}
We argue by induction on $m$.
For $m=0$, there is nothing to show.
Suppose the claim is true for some~$m$, i.e., $\sigma$ sends $Q^{(\ell-m)}$ to $Q^{(j\pm m)}$, 
potentially with the arrows reversed. 
We will argue that the analogous statement holds for $m+1$. 
Let $v_m$ be such that $Q^{(\ell-(m+1))} \mutation{v_m} Q^{(\ell-m)}$.
By Lemma~\ref{lem:cycleIsSinksAndDescents}, 
the vertex $v_m$ satisfies condition~(\ref{lem:cycleIsSinksAndDescents}b) with respect to the quiver~$Q^{(\ell-m)}$. 
It follows that $\sigma(v_m)$ satisfies~(\ref{lem:cycleIsSinksAndDescents}b) with respect to $\sigma(Q^{(\ell-m)})$,
i.e., with respect to $Q^{(j\pm m)}$ (possibly with the arrows reversed). 
Again by Lemma~\ref{lem:cycleIsSinksAndDescents}, this implies that 
$Q^{(j \pm m \pm 1)} \mutation{\sigma(v_m)} Q^{(j \pm m)}$.
Therefore
\[
\sigma(Q^{(\ell-(m+1))}) = \sigma(\T{v_m}{Q^{(\ell-m)}}) = \T{\sigma(v_m)}{Q^{(j\pm m)}} = Q^{(j \pm m \pm 1)}
\]
(or the same with all arrows in $Q^{(j \pm m)}$ and $Q^{(j \pm m\pm 1)}$ reversed). 

It remains to show that the signs in the superscript $j \pm m \pm 1$ must agree. 
Suppose that $\sigma(Q^{(\ell-m)}) = Q^{(j-m)}$; the other cases are completely analogous.
Let $v_{m-1}$ be such that $Q^{(\ell-m)} \mutation{v_{m-1}} Q^{(\ell-(m-1))}$. 
Then $v_{m-1} \neq v_m$ and therefore $\sigma(v_m) \neq \sigma(v_{m-1})$. 
The claim then follows by Lemma~\ref{lem:no2cycles}. 
\end{proof}

\begin{lemma}
\label{lem:noIsoFixesQ}
Let $\ell, j \in[0,4k+n-1]$. 
If $\sigma(Q^{(\ell)})=Q^{(j)}$ and $\sigma(Q)=Q$, then $\ell = j$.
Also, if $\sigma$ sends $Q^{(\ell)}$ to $Q^{(j)}$ with all arrows reversed 
and sends $Q$ to itself with all arrows reversed, then $\ell = j$.
\end{lemma}

\begin{proof}
By Lemma~\ref{lem:isosMove} (with $m=\ell$), $\sigma$ sends $Q = Q^{(0)}$ to $Q^{(j \pm \ell)}$,
possibly with all arrows reversed. 
Assume that there is no reversal of arrows, as the other cases are similar. 
Then $Q=Q^{(j \pm \ell)}$.
Lemma~\ref{lem:Q-neq-Qj} implies that $j \pm \ell \equiv 0 \bmod\!(4k+n)$.
Since $\ell, j \in [0, 4k+n-1],$ we conclude that either $\ell=j$ or $j + \ell = 4k+n$,
in which case (cf.\ Lemma~\ref{lem:isosMove}, first bullet), $\sigma(Q^{(\ell-m)})=Q^{(j+m)}$.
If $\ell=j$, then we are~done. Otherwise, Lemma~\ref{lem:isosMove} (first bullet, with $m = \ell-1 = 4k+n-j-1$) 
implies ${\sigma(Q^{(1)})=Q^{(4k+n-1)}}$. 
Now $Q \shortmutation{n} Q^{(1)}$ implies $Q \mutation{\sigma(n)} {Q^{(4k+n-1)}}$.
But $Q \shortmutation{1} Q^{(4k+n-1)} $, so $\sigma(n)=1$ by Lemma \ref{lem:no2cycles}.
Thus $\sigma$ is a nontrivial automorphism of~$Q$, contradicting Lemma~\ref{lem:QAutomorphisms}.
\end{proof}

\begin{proof}[Proof of Theorem~\ref{thm:GeneralSemiCycleDistinct}]
We need to show that all the quivers $Q^{(j)}$ are distinct.
Suppose ${Q^{(\ell)} = Q^{(j)}}$. 
Applying Lemma~\ref{lem:noIsoFixesQ} with $\sigma=\textup{id}$, we get $\ell = j$.
%So all $Q^{(j)}$ are distinct.
\end{proof}

\begin{proof}[Proof of Theorem~\ref{thm:GeneralSemiCycleDistinct-noniso}]
Suppose $\sigma$ sends $Q^{(\ell)}$ to $Q^{(j)}$, possibly with all arrows reversed. 
By Lemma~\ref{lem:isosMove}, the isomorphism $\sigma$ sends $Q$ to $Q^{(j \pm \ell)}$,
possibly with all arrows reversed.
By Lemma~\ref{lem:QDistinct}, $j \pm \ell = 0\bmod\!(4k+n)$.
Then $\ell=j$ by Lemma~\ref{lem:noIsoFixesQ}.
\end{proof}

\begin{example}
\label{eg:4vertexSymmetric}
Consider the quiver $Q$ defined by the matrix
\[
B(Q) = \begin{pmatrix} 
0 & a & -(a^2-1)b -ac & c \\
-a & 0 & c + ab & b \\
(a^2-1)b +ac & -c-ab & 0 & x \\
-c & -b & -x & 0
\end{pmatrix}
\]
for some integers~$a,b,c,x\ge 2$. 
This quiver $Q$ lies on the (simple) mutation cycle from Theorem~\ref{thm:GeneralSemiCycles}, with $n=4$, $k=1$, 
and $\tilde R$ being the first quiver in~\eqref{eq:RR'-3vertex}. 
That~is, $Q = \T{1 2 123 21 4}{Q}$. 
Let $\sigma(1)=2, \sigma(2)=1, \sigma(3)=4$ and $\sigma(4)=3$. 
Then $\sigma(Q)$ coincides with $\T{3 2 1 4}{Q}$ with all arrows reversed. 
In fact, each quiver $Q^{(j)}$, for $0\le j\le 4$, is isomorphic to $Q^{(4-j)}$ with all arrows reversed, 
cf.\ Lemma~\ref{lem:isosMove}. %Thus while the mutation cycle looks like a loop, up to isomorphism it is a strip, with two different mutations giving isomorphic quivers from the ends. 
See Figure~\ref{fig:cycleCollapse}.
\end{example}

\begin{figure}[ht]
\vspace{-.1in}
\begin{equation*}
\begin{tikzcd}[arrows={-stealth, cramped}
]
\red{\boxed{Q}}  \arrow[rrr, no head, "{[4]}"] \arrow[d, swap, no head, "1"] \ar[ddrr, dashed, no head]
  &&& \red{\bullet} \arrow[d, no head, "1"] \ar[dd, dashed, no head, bend right=10pt]\\
\blue{\bullet} \arrow[d, swap, no head,"2"]  \ar[dr, dashed, no head] &&&  \blue{\bullet} \arrow[d, no head,"2"]  \\
\blue{\boxed{L}} \arrow[r, no head, "1"] & \blue{\bullet} \arrow[r, no head, "2" ] 
    & \red{\bullet} \arrow[r, no head,"{[3]}"] & \red{\boxed{R}}
\end{tikzcd} 
\end{equation*}

\vspace{-.15in}
\caption{
Example~\ref{eg:4vertexSymmetric}. 
The dashed lines indicate pairs of isomorphic quivers, after a reversal of arrows. 
%On the left is a picture of the same mutation cycle, but with isomorphic quivers identified. 
}%Note that all three of $\T{1}{Q}, \T{1}{L}$ and $\T{2}{L}$ result in isomorphic quivers.}
\label{fig:cycleCollapse}
\vspace{-.2in}
\end{figure}
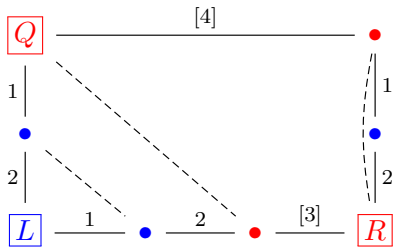

\newpage

\section{Vortices, global descents, and exits}
\label{sec:exits}

In this section, we introduce some technical tools that will be used in Sections~\ref{Sec:primitiveCycles}--\ref{sec:genericity}. 

Several ideas appearing in this section were discovered much earlier, albeit in somewhat different form, 
in the Ph.D.\ thesis of M.~Warkentin, 
see \cite[pages 13--18]{Warkentin}. \linebreak[3]
The key distinguishing features of the approach developed below are the use of vortex-free quivers 
and the notion of global descent, see Definitions~\ref{def:vortex} and~\ref{def:globalDescent}, respectively. 

\begin{remark}
%\label{rem:}
Instead of the aforementioned concepts, \cite{Warkentin}~utilizes the notions of a ``fork'' and a ``point of return.''
For example, per \cite[Definition~2.1]{Warkentin}, 
a cyclic quiver~$Q$ with large weights and vertex set $\{i,j,k\}$ 
such that $|b_{ij}(Q)|>\max(|b_{ik}|,|b_{jk}|)$ 
would be called a``fork with point of return~$k$.''
Our Corollary~\ref{cor:weak-propagates} and Proposition~\ref{pr:weak-is-exit}
can be viewed as loose counterparts of \cite[Lemmas 2.5, 2.8, 3.5]{Warkentin},
although the exact statements and some of the proof arguments are different. 
\end{remark}

The following terminology is an adaptation of one introduced by D.~Knuth~\cite[Section~4]{MR1226891}. 
The papers \cite{brouwer, cameron, moon} use different terms for the same objects.

\begin{definition}
\label{def:vortex}
A \emph{vortex} is a 4-vertex quiver $Q$ such that
\begin{itemize}[leftmargin=.2in]
\item 
all weights in $Q$ are nonzero; 
\item
one of the vertices of $Q$ is a source or a sink; 
\item
the remaining three vertices of $Q$ support a cyclic 3-vertex subquiver. 
\end{itemize}
The (unique) sink/source of a vortex is called its \emph{apex}. 
See Figure~\ref{fig:vortices}. 

A quiver is \emph{vortex-free} if none of its (full) 4-vertex subquivers is a vortex. 
\end{definition}

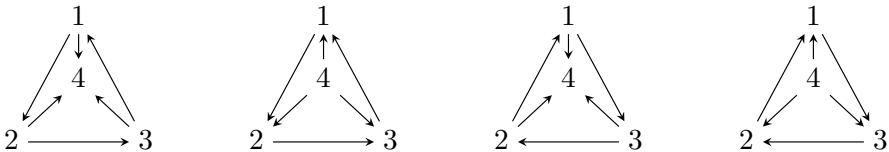
\begin{figure}[ht]
\vspace{-.15in}
\begin{equation*}
\begin{tikzcd}[arrows={-stealth}, sep=small, cramped]
  & 1 \ar[d] \ar[ddl]  & \\ % \arrow[dr, swap, "b"]
  & 4 & \\ 
  2  \ar[ur] \ar[rr]
  & & 3 \ar[ul] \ar[uul]
\end{tikzcd} 
\hspace{.5in} 
\begin{tikzcd}[arrows={-stealth}, sep=small, cramped]
  & 1  \ar[ddl]  & \\ % \arrow[dr, swap, "b"]
  & 4 \ar[u] \ar[dl] \ar[dr]& \\ 
  2   \ar[rr]
  & & 3  \ar[uul]
\end{tikzcd}
\hspace{.5in} 
\begin{tikzcd}[arrows={-stealth}, sep=small, cramped]
  & 1 \ar[d] \ar[ddr]  & \\ % \arrow[dr, swap, "b"]
  & 4 & \\ 
  2  \ar[ur] \ar[uur]
  & & 3 \ar[ul] \ar[ll]
\end{tikzcd} 
\hspace{.5in} 
\begin{tikzcd}[arrows={-stealth}, sep=small, cramped]
  & 1  \ar[ddr]  & \\ % \arrow[dr, swap, "b"]
  & 4 \ar[u] \ar[dl] \ar[dr]& \\ 
  2   \ar[uur]
  & & 3 \ar[ll]
\end{tikzcd}
\end{equation*}
\vspace{-.15in}
\caption{Four vortices with apex at vertex $4$. Weights are not shown.}
\label{fig:vortices}
\end{figure}

\vspace{-.25in}

\begin{definition}
\label{def:globalDescent}
We say that an $n$-vertex quiver $Q$ has a \emph{global descent} 
at vertex~$i$ if
\begin{itemize}[leftmargin=.2in]
%\item $Q$ has large weights,
\item $Q$ contains at least one cyclic $3$-vertex subquiver, 
and 
\item
all such subquivers have descent at~$i$. 
\end{itemize}
(The first condition simply means that $Q$ is not acyclic.) 
\end{definition}

In a quiver with large weights, the global descent vertex~$i$ is unique by Lemma~\ref{lem:3VertexUniqueDescents}. 
(In all applications appearing in this paper, the weights are large.) 
If the particular vertex~$i$ is not important, we will just say that $Q$ has a global descent.

\begin{remark}
%\label{rem:}
In general, a global descent vertex~$i$ does not have to be a ``descent:'' 
mutating at~$i$ might increase some weights.
However, if $Q$ has large weights and is vortex-free, 
then a global descent  does not increase any weights. 
\end{remark}

\begin{lemma}
\label{lem:Q-i-is-acyclic}
Let $Q$ be a quiver on the vertex set $[1,n]$ that has large weights. 
If $Q$ has global descent at $i$, then the $(n-1)$-vertex subquiver $Q|_{[1,n]-\{i\}}$ is acyclic. 
\end{lemma}

\begin{proof}
If $Q|_{[1,n]-\{i\}}$ contains a cyclic 3-vertex subquiver, then $i$ cannot be its descent, a contradiction. 
\end{proof}

\pagebreak[3]

\begin{lemma}
\label{lem:4VertexGlobalDescentTest}
Let $Q$ be a quiver with at least 4 vertices. 
Assume that $Q$ is not acyclic and has large weights. Fix a vertex~$i$. 
The following are equivalent: 
\begin{itemize}[leftmargin=.2in]
\item $Q$ has global descent at~$i$.
%\item every $4$-vertex subquiver of $Q$ either has global descent at $i$ or is acyclic. 
\item every $4$-vertex subquiver of $Q$ containing $i$ either has global descent at $i$ or is acyclic. 
\end{itemize}
\end{lemma}

\begin{proof}
Suppose that $Q$ has global descent at~$i$. Let $Q'$ be a subquiver of~$Q$. 
Every cyclic 3-vertex subquiver of~$Q'$ is a cyclic 3-vertex subquiver of~$Q$, so it must have descent~$i$. 
It follows that either $Q'$ has global descent~$i$ or else is acyclic, cf.\ Definition~\ref{def:globalDescent}. 

Now suppose that every $4$-vertex subquiver of $Q$ containing $i$ either has global descent at $i$ or is acyclic.
Every cyclic $3$-vertex subquiver $Q'$ of~$Q$ appears in a $4$-vertex subquiver~$Q''$ which contains~$i$. 
Since $Q''$ is not acyclic, it has global descent at~$i$; thus, in particular, $Q'$ has descent~$i$. 
\end{proof}

\begin{lemma}
\label{lem:strong=>weak}
Let $i$ be a vertex in a quiver $Q$ with large weights.
Suppose that the pair $(Q,i)$ satisfies the conditions 
\begin{align}
\label{eq:Q-global-descent}
&\text{$Q$ has global descent at a vertex different from~$i$; \hspace{1.8in}} \\
&\label{eq:Q-vortex-free}
\text{$Q$ is vortex-free.}
\end{align}
Then $(Q,i)$ satisfies the conditions 
\begin{align}
\label{eq:Q-ascent-3}
&\text{$i$ is an ascent in every cyclic 3-vertex subquiver of~$Q$ that contains it; \hspace{.35in}} \\
&\label{eq:Q-not-sink/source}
\text{$i$ is not a sink/source in~$Q$;}\\
&\label{eq:Q-not-apex}
\text{$i$ is not the apex of a vortex in~$Q$.}
\end{align}
\end{lemma}

\begin{proof}%\let\qed\relax %command to hide qed symbol
Suppose $Q$ has global descent~$v$.
%We check conditions \eqref{eq:Q-ascent-3}-\eqref{eq:Q-not-apex}:
\begin{itemize}[leftmargin=.42in]
\item[\eqref{eq:Q-ascent-3}:]
Since $v$ is a global descent of~$Q$, any cyclic 3-vertex subquiver of~$Q$ has descent~$v$. 
Thus, by Lemma~\ref{lem:3VertexUniqueDescents}, $i$ is an ascent of any such subquiver. 
\item[\eqref{eq:Q-not-sink/source}:]
Since $Q$ has global descent~$v$, it must contain a cyclic subquiver $Q|_{uvw}$. 
If $i$ were a sink/source, then $Q|_{iuvw}$ would be a vortex with apex~$i$. 
\item[\eqref{eq:Q-not-apex}:]
By assumption $Q$ is vortex-free, so $i$ is not the apex of a vortex. \qedhere
\end{itemize}
\end{proof}

\begin{lemma}
\label{lem:StrictAscents}
Let $i$ and $i'$ be distinct vertices in a quiver $Q$ with large weights.
Let~$Q'\!=\!\T{i}{Q}$. 
If the pair $(Q,i)$ satisfies  \eqref{eq:Q-ascent-3}-\eqref{eq:Q-not-apex}, then 
\begin{itemize}[leftmargin=.2in]
\item 
$(Q',i')$ satisfies~\eqref{eq:Q-global-descent}, with global descent~$i$; 
\item
$(Q',i')$ satisfies~\eqref{eq:Q-vortex-free}; 
\item
$|b_{uv}(Q)| \leq |b_{uv}(Q')|$ for all $u$ and~$v$, with at least one strict inequality.
\end{itemize}
\end{lemma}

\begin{proof}
Let $Q$ be a quiver on the vertex set $[1,n]$. 
We will first prove the result for $n=1, 2, 3, 4$, then use the $n=4$ case to prove the general case.

If $n=1$ or $n=2$, the result is vacuous: every vertex is a sink/source, so \eqref{eq:Q-not-sink/source} is never satisfied.  
For $n=3$, the claim is trivial, cf.\ Lemmas~\ref{lem:acyclic-descent}--\ref{lem:3VertexUniqueDescents}. 

Let $n=4$ and say $i=2$. 
Since $2$ is not a sink/source, we may assume, without loss of generality, that $1\points 2$, $2\points 3$, and $2\points 4$. 
We may further assume that $3\points 4$. 
Thus $Q|_{123}$ and $Q|_{124}$ are either cyclic or have an elbow at~$2$. 
Regardless, $2$~is an ascent in both $Q|_{123}$ and $Q|_{124}$. 
By Lemma~\ref{lem:acyclic-descent}, both $Q'|_{123}$ and $Q'|_{124}$ are cyclic with descent at~$2$.
We conclude that $Q'$ is oriented as follows: 
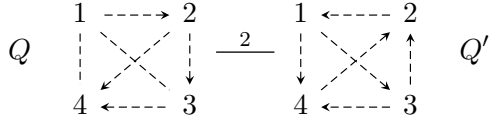
\begin{equation*}
Q \quad \begin{tikzcd}[arrows={stealth-, dashed}, sep=normal, ampersand replacement=\&]
  1 \arrow[d, -] \arrow[dr, -]
  \& 2 \arrow[l]
  \\
   4 \arrow[ur] \arrow[r] \& 3 \arrow[u] 
\end{tikzcd}
\mutation{2}
\begin{tikzcd}[arrows={-stealth, dashed}, sep=normal, ampersand replacement=\&]
  1 \arrow[d] \arrow[dr]
  \& 2 \arrow[l]
  \\
   4 \arrow[ur] \& 3 \arrow[u] \arrow[l]
\end{tikzcd}
\quad
Q'
\end{equation*}
In particular, $Q'$ is vortex free and all cyclic subquivers have descent~$2$. 
Finally, since $2$ is an ascent in $Q|_{123}$ and $Q|_{124}$, 
we have $|b_{13}(Q)| < |b_{13}(Q')|$ and $|b_{14}(Q)| < |b_{14}(Q')|$. 
Since all other weights in~$Q'$ are unchanged from~$Q$, we are done with the $n=4$ case.

\pagebreak[3]

Now suppose that $n > 4$. 
Since $i$ is not a sink/source in~$Q$, let $u\points i\points v$ in~$Q$.
As $i$ is an ascent of every cyclic $3$-vertex subquiver of~$Q$ containing~$i$,
it follows that 
\begin{itemize}[leftmargin=.3in]
\item[{\rm (a)}]
 the weights don't decrease when we mutate from~$Q$ to~$Q'$;
moreover, at least one weight does increase: $|b_{uv}(Q)| < |b_{uv}(Q')|$;
\item[{\rm (b)}]
the subquiver $Q'|_{iuv}$ has descent at~$i$, hence $Q'$ is not acyclic.
\end{itemize}
Statement (a) above establishes the last claim in Lemma~\ref{lem:StrictAscents}. 
To prove \eqref{eq:Q-global-descent}-\eqref{eq:Q-vortex-free}, we will need the following statements for all distinct vertices $u,v,w$: 
\begin{itemize}[leftmargin=.3in]
\item[{\rm (c)}]
any 4-vertex subquiver $Q'|_{iuvw}$ is vortex-free; 
\item[{\rm (d)}]
any 4-vertex subquiver $Q'|_{iuvw}$ is either acyclic or has global descent~$i$.
\end{itemize}
Indeed, if $i$ is not a sink/source in $Q|_{iuvw}$, then $(Q|_{iuvw}, i)$ satisfies \eqref{eq:Q-ascent-3}-\eqref{eq:Q-not-apex}; %the assumptions of Lemma~\ref{lem:StrictAscents}; 
hence $Q'|_{iuvw}= \T{i}{Q|_{iuvw}}$ is vortex-free and has global descent~$i$ by the $n=4$ case. 
If instead $i$ is a sink/source in~$Q|_{iuvw}$ (but not the apex of a vortex), then~$Q|_{iuvw}$ is acyclic,
so~$Q'|_{iuvw}$ is acyclic as well (in particular, not a vortex). 
In either case, statements (c) and (d) follow.

Combining Lemma~\ref{lem:4VertexGlobalDescentTest} (for the quiver~$Q'$)
with statements (b) and (d) above,
we conclude that $Q'$ has global descent~$i$. 
In particular, $(Q', i')$ satisfies~\eqref{eq:Q-global-descent}.

It remains to show that $Q'$ satisfies~\eqref{eq:Q-vortex-free}. % is vortex-free. 
By Lemma~\ref{lem:Q-i-is-acyclic}, the subquiver~$Q'|_{[1,n]-\{i\}}$ is acyclic. 
So a vortex in~$Q'$ must contain~$i$; but this is impossible by statement~(c). 
\end{proof}

Lemmas~\ref{lem:strong=>weak}--\ref{lem:StrictAscents} imply that conditions \eqref{eq:Q-ascent-3}-\eqref{eq:Q-not-apex}
(resp., \eqref{eq:Q-global-descent}-\eqref{eq:Q-vortex-free})
propagate: 

\begin{corollary}
\label{cor:weak-propagates}
Let $i$ and $i'$ be distinct vertices in a quiver $Q$ with large weights.
Let~$Q'=\T{i}{Q}$. 
Suppose that the pair $(Q,i)$ satisfies  \eqref{eq:Q-ascent-3}-\eqref{eq:Q-not-apex} (resp., \eqref{eq:Q-global-descent}-\eqref{eq:Q-vortex-free}). 
Then the pair $(Q',i')$ satisfies the same conditions; 
in fact, $Q'$~has global descent~$i$. 
Moreover, $|b_{uv}(Q)| \leq |b_{uv}(Q')|$ for all $u$ and~$v$, with at least one strict inequality.
\end{corollary}

The following notion and its properties discussed below will be used in Section~\ref{Sec:primitiveCycles}. 

\begin{definition}
\label{def:transient} 
In a quiver $Q$, a vertex~$i$ is an \emph{exit} if for every sequence of vertices 
$i\neq i_1\neq i_2\neq\cdots\neq i_{\ell-1}\neq i_\ell$, we have~$Q \neq \T{i_\ell \cdots i_1\,i}{Q}$. 
Informally, once we mutate at~$i$, we cannot return to~$Q$. 
Consequently, the edge $Q\shortmutation{i} \mu[i](Q)$ of the mutation graph does not lie on any mutation cycle. 
\end{definition}

\begin{lemma}
\label{lem:exit-leads-to-tree}
Suppose that $i$ is an exit in a quiver~$Q$. Let $Q'=\mu[i](Q)$. 
Consider the subgraph of the mutation graph that ``lies beyond'' the edge $Q \shortmutation{i}Q'$. 
That is, take the induced subgraph of the mutation graph whose vertices correspond to all possible quivers of the form $\T{i_\ell \cdots i_1\,i}{Q}$ as above (including~$Q'$). 
This subgraph is a complete rooted $(n-1)$-ary tree with root~$Q'$.
\end{lemma}

\begin{proof}
Suppose for contradiction that there are two mutation sequences, both based at~$Q$ and starting with~$i$, that lead to the same quiver. 
Reversing one of the sequences, concatenating them, and removing consecutive entries equal to each other
produces a sequence  $i\neq i_1\neq i_2\neq\cdots\neq i_{\ell-1}\neq i_\ell$ satisfying $Q = \T{i_\ell \cdots i_1\,i}{Q}$, a contradiction.
\end{proof}

\begin{remark}
%\label{rem:}
The description given in Lemma~\ref{lem:exit-leads-to-tree} still holds 
if we identify quivers up to isomorphism and global reversal of arrows, 
except that the degree of each non-root vertex in a tree will be \emph{at most}~$n$. 
Cf.\ \cite[Lemmas 2.7--2.8]{Warkentin}.
\end{remark}

\pagebreak[3]

\begin{proposition}
\label{pr:weak-is-exit}
If $(Q,i)$ satisfies \eqref{eq:Q-ascent-3}-\eqref{eq:Q-not-apex}, then $i$ is an exit (cf.\ Definition~\ref{def:transient}).
\end{proposition}

\begin{proof}
Let $i=i_0\neq i_1\neq i_2\neq\cdots$. 
Repeatedly applying Corollary~\ref{cor:weak-propagates}, we conclude that for any~$j$,
the pair $(\T{i_j \cdots i_1\,i}{Q}, i_{j+1})$ satisfies \eqref{eq:Q-ascent-3}-\eqref{eq:Q-not-apex} 
and moreover the total number of arrows in $\T{i_j \cdots i_1\,i}{Q}$ increases with~$j$.
Thus $\T{i_j \cdots i_1\,i}{Q}\neq Q$ for all~$j$, as desired. 
\end{proof}

We conclude this section by some observations that will be used in Section~\ref{sec:genericity}. 

\begin{lemma}
\label{lem:orientations-fixed}
Let $Q$ be a quiver with large weights and let $i$ be a vertex in~$Q$
such that the pair $(Q, i)$ satisfies conditions \eqref{eq:Q-ascent-3}--\eqref{eq:Q-not-apex}. 
Let $i\,i_1\cdots i_m$ be a sequence that begins with~$i$
and satisfies $i\neq i_1\neq \cdots\neq i_m$. 
Then the orientations of arrows in the quiver $\T{i_m\cdots i_1\,i}{Q}$ are uniquely determined by 
the orientations of arrows in~$Q$. 
\end{lemma}

That is, suppose that $Q'$ is another quiver on the same set of vertices such that \linebreak[3]
(a) $Q'$ has large weights, (b) $Q$ and $Q'$ have the same orientations of arrows, 
and \linebreak[3]
(c) the pair $(Q',i)$ satisfies conditions \eqref{eq:Q-ascent-3}--\eqref{eq:Q-not-apex}. 
Then the quivers $\T{i_m\cdots i_1\,i}{Q}$ and $\T{i_m\cdots i_1\,i}{Q'}$
have the same orientations of arrows. 

\begin{proof}
By Corollary~\ref{cor:weak-propagates}, the quiver $\mu[i](Q)$ and the sequence $i_1\cdots i_m$ satisfy 
the conditions of Lemma~\ref{lem:orientations-fixed}. 
It therefore suffices to check that the orientations of arrows in $\mu[i](Q)$ are uniquely determined by 
the orientations of arrows in~$Q$. 

By definition, $\mu[i]$ reverses all arrows in~$Q$ that are incident to~$i$. 
For any $u \points v$  in $Q$ with $u\neq i$ and $v\neq i$, we have the following cases. 

\noindent
\emph{Case~1:} $Q|_{iuv}$ is acyclic.
Then $u \points v$ in $\T{i}{Q}$, unchanged from~$Q$. 

\noindent
\emph{Case~2:} $Q|_{iuv}$ is cyclic.
Then, by condition \eqref{eq:Q-ascent-3}, 
\begin{equation*}
b_{uv}(\mu[i](Q)) = b_{uv}(Q) - b_{ui}(Q) b_{iv}(Q) < b_{uv}(Q) < |b_{uv}(\mu[i](Q))|, 
\end{equation*}
implying that $v\points u$ in $\T{i}{Q}$. 
\end{proof}

\begin{corollary}
\label{cor:orientations-fixed}
Let $Q$ be an acyclic quiver with large weights. 
Then for any sequence $i_1\cdots i_m$, 
the orientations of arrows in the quiver $\T{i_m\cdots i_1}{Q}$ are uniquely determined by 
the orientations of arrows in~$Q$. 
\end{corollary}

\begin{proof}
If $i_1$ is a sink/source, then $\mu[i_1](Q)$ is again acyclic. 
Therefore, without loss of generality, we may assume that $i_1$ is not a sink/source. 
We may moreover assume that  $i_1\neq \cdots\neq i_m$. 
Now Lemma~\ref{lem:orientations-fixed} applies, proving the claim. 
\end{proof}

\begin{remark}
The results of this section have been generalized by T.~Ervin \cite{ervinPrefork} to include certain quivers with a single `small' weight. 

While not all of Theorem~\ref{thm:Summary} extends to arbitrary weights $|b_{ij}| \le 1$, 
many such choices do produce a mutation cycle, and the size and location of other mutation cycles can be controlled.
\end{remark}

\newpage

\section{Mutation cycles that cannot be paved by short cycles}
\label{Sec:primitiveCycles}

In this section, we show that the long mutation cycles constructed in Theorem~\ref{thm:GeneralSemiCycles}
cannot be paved by mutation cycles of bounded length: 

\begin{theorem}
\label{thm:primitive}
Fix~$n \geq 4$ and $k\ge 1$. 
Let $Q$ be an $n$-vertex quiver constructed as in Theorem~\ref{thm:GeneralSemiCycles}.
Then $Q$ does not lie on any mutation cycle of length $\le 4k+1$. 
Consequently, the mutation cycle~\eqref{eq:GeneralSemiCycle} 
(which has length $4k+n$) cannot be paved by mutation cycles of length $\le 4k+1$. 
\end{theorem}

\begin{theorem}
\label{thm:primitive-nonisom}
If we also assume that either $n\ge 5$ or  else $n=4$ and ${b_{14}(Q) \neq b_{23}(R)}$,  
then the above statements remain true if we identify quivers that differ by an isomorphism and/or a reversal of all arrows.
\end{theorem}

The proofs of these theorems will rely on Lemmas~\ref{lem:QTildeGlobalDescentVF}--\ref{lem:cycle-exits} below.

We will need some notation that will be used throughout. 
We will use~$L$ (instead of~$Q$) as a base point for our mutation cycle~\eqref{eq:GeneralSemiCycle}. 
We will denote by~$L^{(\ell)}$ the result of applying $\ell$ mutations to~$L$, going clockwise along the cycle. 
Thus $L^{(\ell)}=Q^{(n+2k+\ell)}$. 

\begin{lemma}
\label{lem:QTildeGlobalDescentVF}
The subquivers $L^{(\ell)}|_{[1,n-1]}$ are vortex-free. 
\end{lemma}

\begin{proof}
By Lemmas~\ref{lem:orientations}--\ref{lem:Qj-large}, the quiver $L_{[1,n-1]}$ is acyclic, has large weights, 
and its vertex~$2$ is not a sink/source.
So Lemma~\ref{lem:StrictAscents} implies that 
the quiver $L^{(1)}|_{[1,n-1]} = \T{2}{L|_{[1,n-1]}}$ is vortex-free and has global descent at~$2$. 
Then, by induction on~$\ell$, Corollary~\ref{cor:weak-propagates} implies that $L^{(\ell)}|_{[1,n-1]}$ is vortex-free 
%and has a global descent at $\varepsilon(\ell+1)$ 
for~$0 < \ell \leq 2k$. 
By Lemma~\ref{lem:SubquiverTildeQ}, ${L^{(\ell)}|_{[1,n-1]}=L^{(4k+1 - \ell)}|_{[1,n-1]}}$,
implying the claim for $2k<\ell\le 4k$.
For the remaining values $4k\le \ell\le 4k+n$, the claim follows from the fact 
(cf.~\eqref{eq:L[n-1]=tildeL}) 
that the mutations at the bottom of Figure~\ref{fig:GeneralSemiCycles-notation},
when restricted to the vertex set $[1,n-1]$, are sink/source mutations,
hence they preserve the property of being vortex-free. 
\end{proof}

\begin{lemma}
\label{lem:VortexFreeEll}
If $0\le \ell \le 2k-1$ or $2k+2\le \ell \leq 4k$, then $L^{(\ell)}$ is vortex-free. 
If $\ell = 2k$ or $\ell=2k+1$, then $L^{(\ell)}$ has vortices, but all of them have apex~$n$. 
%$n$ is the only apex of a vortex in~$L^{(\ell)}$.
\end{lemma}
\begin{proof}
By Lemma~\ref{lem:QTildeGlobalDescentVF}, the quiver $L^{(\ell)}|_{[1,n-1]}$ is vortex-free. 
Let us now assume that $\ell \notin \{2k, 2k+1\}$.
We can check using Lemma~\ref{lem:orientations} 
that in every 4-vertex subquiver of $L^{(\ell)}$ containing~$n$, 
the number of cyclic 3-vertex subquivers is either 0 or~2---hence this subquiver is not a vortex. 
Finally, $n$ is a sink/source in both $L^{(2k)}$ and $L^{(2k+1)}$, so $n$ is the only possible apex
for the vortices contained in these quivers. 
\end{proof}

The following result shows that for a quiver $L^{(\ell)}$ 
not lying on the bottom rim of Figure~\ref{fig:GeneralSemiCycles-notation},
every mutation that takes us away from the mutation cycle~\eqref{eq:GeneralSemiCycle}
gives rise to a tree in the mutation graph. 

\begin{lemma}
\label{lem:cycle-exits}
For all $0 < \ell \le 4k$ and every vertex $i$, either 
$\T{i}{L^{(\ell)}} = L^{(\ell \pm 1)}$ or 
$i$ is an exit in $L^{(\ell)}$.
\end{lemma}

\begin{proof}
Suppose that $\T{i}{L^{(\ell)}} \neq L^{(\ell \pm 1)}$.
It follows from Lemma~\ref{lem:VortexFreeEll} that $i$ is not an apex of a vortex in~$L^{(\ell)}$. 
One can check, via repeated use of Lemma~\ref{lem:orientations}, 
that $L^{(\ell)}$ has at most one sink/source. 
(Specifically, $L^{(\ell)}$ has no sinks/sources unless $\ell = 2k$ or $\ell=2k+1$, 
in which case $n$ is the only sink/source.)

Combining  Lemma~\ref{lem:cycleIsSinksAndDescents},
the assumption $\T{i}{L^{(\ell)}} \neq L^{(\ell \pm 1)}$, and the fact that $L^{(\ell)}$ has at most one sink/source, 
we conclude that $i$ is an ascent of every cyclic $3$-vertex subquiver of~$L^{(\ell)}$ that contains~$i$.

Now Proposition~\ref{pr:weak-is-exit} implies that $i$ is an exit in $L^{(\ell)}$. 
\end{proof}

\begin{remark}
It may well be that the restriction on $\ell$ in Lemma~\ref{lem:cycle-exits} is unnecessary. 
The current proof of the lemma does not extend to all~$\ell$ since Proposition~\ref{pr:weak-is-exit} may not apply. 
This is illustrated in Figure~\ref{fig:mutations-which-may-not-be-exits}.
\end{remark}

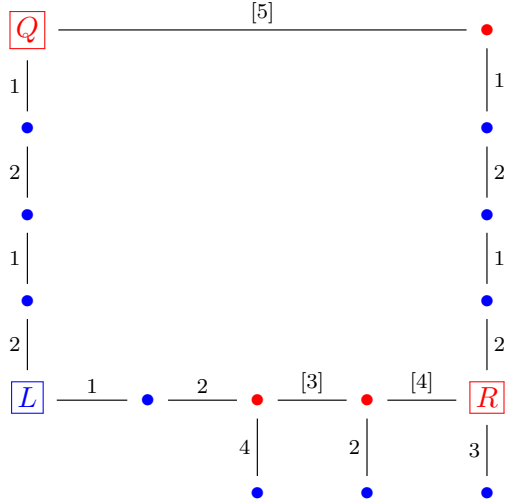
\begin{figure}[ht]
\begin{center}
\vspace{-10pt}
\begin{equation*}
\begin{tikzcd}[arrows={-stealth, cramped, sep=small}]
\red{\boxed{Q}}  \arrow[rrrr, no head, "{[5]}"] \arrow[d, swap, no head, "1"]
  &&&& \red{\bullet} \arrow[d, no head, "1"] \\
\blue{\bullet} \arrow[d, swap, no head,"2"] &&&&  \blue{\bullet} \arrow[d, no head,"2"]  \\
\blue{\bullet} \arrow[d, swap, no head,"1"] &&&&  \blue{\bullet} \arrow[d, no head,"1"]  \\
\blue{\bullet} \arrow[d, swap, no head,"2"] &&&&  \blue{\bullet} \arrow[d, no head,"2"]  \\
\blue{\boxed{L}} \arrow[r, no head, "1"] & \blue{\bullet} \arrow[r, no head, "2" ] 
    & \red{\bullet} \arrow[r, no head,"{[3]}"] & \red{\bullet} \arrow[r, no head,"{[4]}"] & \red{\boxed{R}}\\
& & \blue{\bullet} \ar[u, no head, "4"] & \blue{\bullet} \ar[u, no head, "2"] & \blue{\bullet} \ar[u, no head, "3"]
\end{tikzcd}
\end{equation*}
\vspace{-10pt}
\end{center}
\caption{
The cycle in Theorem~\ref{thm:GeneralSemiCycles} (${n=5, k=2}$). 
Each of the three mutations shown at the bottom as vertical edges does not satisfy \eqref{eq:Q-ascent-3} 
or \eqref{eq:Q-not-apex}. 
}
\label{fig:mutations-which-may-not-be-exits}
\end{figure}

\vspace{-20pt}

\begin{proof}[Proof of Theorem~\ref{thm:primitive}]
By Lemmas~\ref{lem:cycle-exits} and~\ref{lem:exit-leads-to-tree}, 
no mutation cycle can contain a mutation of the form 
$L^{(\ell)}\shortmutation{i} L'\neq L^{(\ell\pm1)}$ for $0<\ell\le 4k$. 
That is, we cannot deviate from the original mutation cycle anywhere outside the bottom rim of Figure~\ref{fig:GeneralSemiCycles-notation}. 

By Theorem~\ref{thm:GeneralSemiCycleDistinct}, all the quivers $L^{(\ell)}$ are distinct, 
so any mutation cycle containing $Q=L^{(2k)}$ has to contain all the quivers $L^{(\ell)}$, for $0\le \ell\le 4k+1$. 
The claim follows. 
\end{proof}

\begin{proof}[Proof of Theorem~\ref{thm:primitive-nonisom}]
We claim that no cycle in the mutation graph (where we identify quivers up to isomorphism and/or reversal of all arrows) 
can go through~$Q$ and contain a mutation of the form 
$L^{(\ell)}\shortmutation{i} L'\neq L^{(\ell\pm1)}$ for $0<\ell\le 4k$. 
Otherwise, continuing the cycle beyond~$L'$, we would eventually reach a quiver~$Q'$ isomorphic to~$Q$
(up to global reversal of arrows).
The mutation graph contains a cycle passing through~$Q'$, 
namely an appropriate version of the cycle~\eqref{eq:GeneralSemiCycle}. 
On the other hand, $i$~is an exit by Lemma~\ref{lem:cycle-exits}, so the portion of the mutation graph
that lies beyond~$L'$ is a tree (by Lemma~\ref{lem:exit-leads-to-tree}), a contradiction.
We conclude that we cannot deviate from the original mutation cycle anywhere outside the bottom rim of Figure~\ref{fig:GeneralSemiCycles-notation}. 

By Theorem~\ref{thm:GeneralSemiCycleDistinct-noniso}, all the quivers $L^{(\ell)}$ are pairwise non-isomorphic, 
even allowing a reversal of all arrows. 
Hence any cycle in the mutation graph (where we identify quivers up to isomorphism and/or reversal of all arrows) 
that contains $Q=L^{(2k)}$ has to contain all the quivers $L^{(\ell)}$, for $0\le \ell\le 4k+1$. 
The claim follows.
\end{proof}

\newpage

\section{Genericity}
\label{sec:genericity}

We next demonstrate that the quivers
that appear in Theorem~\ref{thm:GeneralSemiCycles} are in some sense ``generic.''
Making this statement precise will require some preliminaries.

\begin{definition}
\label{def:quivers=lattice-points}
We denote by $\Quiv_n$ the set of all quivers on the vertex set $[1,n]$.
A subset $\mathbf{Q}\subset\Quiv_n$ is called a \emph{family of quivers} (on this vertex set). 

Each quiver $Q\in\Quiv_n$ is encoded by a skew-symmetric matrix ${B(Q)=(b_{ij})}$, as explained in Definition~\ref{def:B(Q)}. 
Restricting to the upper-triangular part $(b_{ij})_{1\le i<j\le n}$, 
we can~view~$\Quiv_n$ as an ${\binom{n}{2}}$-dimensional integer lattice. 
Consequently, any family $\mathbf{Q}\subset\Quiv_n$ can be viewed as a collection of lattice points. 
\end{definition}

\begin{definition}
\label{def:z-biregular}
Let $A$ and $B$ be two subsets of $\ZZ^d$. 
A map 
\vspace{-5pt}
\begin{align*}
A&\stackrel{f}{\to} B \\[-3pt]
(a_1,\dots,a_d)&\mapsto (b_1,\dots,b_d) 
\end{align*}
is called \emph{$\ZZ$-biregular} if 
\begin{itemize}[leftmargin=.2in]
\item 
$f$ is a bijection; 
\item
both $f$ and its inverse  $g=f^{-1}$ are given by polynomials with integer coefficients: 
\begin{align*}
b_i&=f_i(a_1,\dots,a_d)\in\ZZ[a_1,\dots,a_d], \quad i=1,\dots, d; \\
a_i&=g_i(b_1,\dots,b_d)\in\ZZ[b_1,\dots,b_d], \quad i=1,\dots, d. 
\end{align*}
\end{itemize}
In light of Definition~\ref{def:quivers=lattice-points}, we can extend this notion to the cases where
either $A$ or $B$ (or both) are subsets of~$\Quiv_n\cong \ZZ^{\binom{n}{2}}$.  %In that case, $d=\binom{n}{2}$.
\end{definition}

\begin{remark}
%\label{rem:}
We will typically use the above notion for infinite subsets of~$\ZZ^d$ defined by algebraic inequalities. 
\end{remark}

\begin{remark}
%\label{rem:}
The inverse of a $\ZZ$-biregular map (resp., a composition of $\ZZ$-biregular maps) is $\ZZ$-biregular. 
Also, if $f:A\to B$ is a $\ZZ$-biregular map and $A'\subset A$, then the restriction $f|_{A'}:A'\to f(A')$ 
is again a $\ZZ$-biregular map. 
\end{remark}

\begin{remark}
\label{rem:relabeling-regular}
The notion of $\ZZ$-biregularity, when applied to maps involving families of quivers, does not depend on the choice of vertex labeling. 
Put differently, any permutation of the vertex labels of a quiver induces a $\ZZ$-biregular map. 
\end{remark}

\begin{example}
\label{eg:QQ'}
Let $\mathbf{Q}$ be the following family of acyclic 3-vertex quivers:
\begin{equation*}
\mathbf{Q}=\Bigl\{
Q(a,b,c)=
\begin{tikzcd}[arrows={-stealth}, sep=small, cramped]
  \scriptstyle 1 \arrow[r, "a"]  
  & \scriptstyle 2 
  \\
  \scriptstyle 3 \arrow[u,"b"] \arrow[ur, swap, "c" , outer sep=-1pt] &
\end{tikzcd}
\,\Bigl| \, a,b,c\ge 0
\Bigr\}\subset\Quiv_3. 
\end{equation*}
The parametrization map $\ZZ_{\ge0}^3\to\mathbf{Q}$ given by $(a,b,c)\mapsto Q(a,b,c)$ is clearly $\ZZ$-biregular. 
A more interesting example is the restriction of the mutation map $\mu[1]$ to~$\mathbf{Q}$: 
\begin{align*}
\mu[1]: \mathbf{Q}&\to\mathbf{Q'}=\{Q'(a,b,c')=\begin{tikzcd}[arrows={-stealth}, sep=small, cramped, ampersand replacement=\&]
  \scriptstyle 1    \arrow[d, swap, "b"] 
  \& \scriptstyle 2 \arrow[l, swap, "a" ]  
  \\
   \scriptstyle 3 \arrow[ur, swap, "{c'}", outer sep=-2pt] \&
\end{tikzcd} \mid a\ge 0, b\ge 0, c'\ge ab\}\subset\Quiv_3. 
\end{align*}
Composing the two maps, we get the $\ZZ$-biregular parametrization $\ZZ^3_{\ge0}\to\mathbf{Q'}$ defined by
$(a,b,c)\mapsto Q'(a,b,c+ab)$. 
The inverse map sends $Q'(a,b,c')\in \mathbf{Q'}$ to $(a,b,c'-ab)\in\ZZ^3_{\ge0}$. 
\end{example}

\begin{example}
\label{eg:Q1Q3}
Let $a,b,c\ge0$ and let the quiver $Q$ be as shown in~\eqref{eq:Q1Q2Q3-12} below. 
Apply the mutation sequence $\mu[21]$ to~$Q$ and assume that the resulting quivers are oriented as follows: 
\begin{equation}
\label{eq:Q1Q2Q3-12}
\!\!\!\!\begin{array}{ccccccc}
\begin{tikzcd}[arrows={-stealth}]
  1 \arrow[r, "a"]  
  & 2 
  \\
  3 \arrow[u,"b"] \arrow[ur, swap, "c" ] &
\end{tikzcd}
& &
\begin{tikzcd}[arrows={-stealth}]
  1    \arrow[d, swap, "b"] 
  & 2 \arrow[l, swap, "a" ]  
  \\
   3 \arrow[ur, swap, "ab+c"] &
\end{tikzcd}
& &
\!\begin{tikzcd}[arrows={-stealth}]
  1    \arrow[r, "a"]
  & 2  \arrow[dl, "ab+c" ] 
  \\
   3 \arrow[u, "\!\!\!(a^2\!-1)b+ac"] &
\end{tikzcd}
\\
Q
& \!\!\!\!\!\!\mutation{1}\!\!\!\!\!\!
& Q'
& \!\!\!\!\!\!\mutation{2}\!\!\!\!\!\!
& \qquad\qquad Q''.
\end{array}
\end{equation}
In other words, we assume that 
\begin{equation*}
Q\in \mathbf{Q} = \{ Q(a,b,c) \mid a,b,c, (a^2-1)b+ac\ge 0\} \subset\Quiv_3. 
\end{equation*}

%(That is, require $(a^2-1)b+ac\ge 0$, or equivalently, exclude the case $a=0$, $b>0$.) 

Alternatively, one could start with a quiver $Q''$ as in~\eqref{eq:Q1Q2Q3-12-tilde} below and apply $\mu[12]$. 
%Setting $\tilde b = (a^2\!-1)b+ac$ and $\tilde c = ab+c$, the same sequence can be written:
Assume that the orientations are the same as in~\eqref{eq:Q1Q2Q3-12}: 
\begin{equation}
\label{eq:Q1Q2Q3-12-tilde}
\!\!\!\!\begin{array}{ccccccc}
\begin{tikzcd}[arrows={-stealth}]
  1 \arrow[r, "a"]  
  & 2 
  \\
  3 \arrow[u,"a {c''} - {b''}"] \arrow[ur, swap, "a{b''} - ( a^2 - 1){c''}" ] &
\end{tikzcd}
& &
\begin{tikzcd}[arrows={-stealth}]
  1    \arrow[d, swap, "a {c''} - {b''}"] 
  & 2 \arrow[l, swap, "a" ]  
  \\
   3 \arrow[ur, swap, "{c''}"] &
\end{tikzcd}
& &
\!\begin{tikzcd}[arrows={-stealth}]
  1    \arrow[r, "a"]
  & 2  \arrow[dl, "{c''}" ] 
  \\
   3 \arrow[u, "{b''}"] &
\end{tikzcd}
\\
Q
& \!\!\!\!\!\!\!\mutation{1}\!\!\!\!\!\!
& \qquad Q'
& \!\!\!\!\!\!\mutation{2}\!\!\!\!\!\!
& Q''
\end{array}
\end{equation}
In other words, we assume that 
\begin{equation*}
Q''\in \mathbf{Q''} = \Bigl\{ \begin{tikzcd}[arrows={-stealth}, sep=small, cramped]
  \scriptstyle 1    \arrow[r, "a"]
  & \scriptstyle 2  \arrow[dl, "{c''}" , outer sep=-2] 
  \\
   \scriptstyle 3 \arrow[u, "{b''}"] &
\end{tikzcd} \mid a,{b''},{c''}, a{c''}-{b''}, a{b''} -(a^2-1){c''} \ge 0
\Bigr\} . 
\end{equation*}
The natural map $\mu[21]:\mathbf{Q}\to \mathbf{Q''}$ given by
%\begin{equation*}
${b''} = (a^2\!-1)b+ac$ and ${c''} = ab+c$
%\end{equation*}
is $\ZZ$-biregular. 
The inverse map %$Q''\mapsto\T{12}{Q''}$ 
$\mu[12]:\mathbf{Q''}\to \mathbf{Q}$
is given by 
%\begin{equation*}
$b=a{c''}-{b''}$,  $c=a{b''} - ( a^2 - 1){c''}$. 
%\end{equation*}

One could also define the family $\mathbf{Q'}$ of quivers~$Q'$ that appear in mutation sequences
of the form \eqref{eq:Q1Q2Q3-12}--\eqref{eq:Q1Q2Q3-12-tilde}, with specified orientations of all quivers. 
We would then get $\ZZ$-biregular maps $\mathbf{Q}\to \mathbf{Q'}\to \mathbf{Q''}$. 
\end{example}

\begin{remark}
%\label{rem:}
Generalizing Example~\ref{eg:Q1Q3}, we can see that fixing orientations of the quivers related to each other by 
a sequence of successive mutations yields a $\ZZ$-biregular map
between the families of quivers defined by appropriate inequalities. 
\end{remark}

Let $\mathbf{Acyc}_n$ denote the family of all $n$-vertex acyclic quivers 
with standard orientation and large weights:
\begin{equation}
\label{eq:Acyc_n}
\mathbf{Acyc}_n=\{Q\in\Quiv_n\mid b_{ij}(Q)\ge 2\ \text{for all $i<j$}\}. 
\end{equation}

\begin{lemma}
\label{lem:mut-acyclic-orientations}
Fix a mutation sequence $\mathbf{i}=i_1\cdots i_m$. 
%and set $\mathbf{Q}_{\mathbf{i}}=\{\mu[\mathbf{i}](Q)\mid Q\in\mathbf{Q^A}$. 
Then the restriction of the map $\mu[\mathbf{i}]:\Quiv_n\to\Quiv_n$ to $\mathbf{Acyc}_n$
is $\ZZ$-biregular onto its image.
\end{lemma}

\begin{proof}
By Corollary~\ref{cor:orientations-fixed}, the orientations of arrows in each quiver $\T{i_j\cdots i_1}{Q}$
do not depend on the choice of $Q\in\mathbf{Acyc}_n$. 
This implies that every map 
\begin{align*}
\mu[i_j]:\T{i_{j-1} \cdots i_1}{\mathbf{Acyc}_n} &\longrightarrow \T{i_j\cdots i_1}{\mathbf{Acyc}_n} \\
 \T{i_{j-1}\cdots i_1}{Q} &\longmapsto\T{i_j\cdots i_1}{Q} 
\end{align*}
is given by polynomials, as is its inverse. 
Thus, all these maps are $\ZZ$-biregular. 
Composing them all, we obtain the claim. 
\end{proof}

\begin{definition}
\label{def:generic-quivers}
Fix an integer $n\ge 2$. 
A family $\mathbf{Q}$ of $n$-vertex quivers is called \emph{fully generic} if
it contains a subset that is related to $\ZZ_{\ge0}^{\binom{n} 2}$ via a $\ZZ$-biregular parametrization. 
\end{definition}

\begin{example}[Quivers with a given orientation]
\label{eg:fixed-orientation-generic}
For each pair of indices $i,j \in [1,n]$, fix a sign $\varepsilon_{ij}\in\{-1,1\}$. 
Then the following quiver family is fully generic: 
\begin{equation*}
\mathbf{Q}=\{Q\in\Quiv_n \mid b_{ij}(Q)\cdot\varepsilon_{ij}\ge 0\}. 
\end{equation*}
%Clearly, the family $\mathbf{Q}$ is fully generic.
This remains true if we replace (some of) the weak inequalities $b_{ij}(Q)\cdot\varepsilon_{ij}\ge 0$ 
by the strict inequalities $b_{ij}(Q)\cdot\varepsilon_{ij}> 0$, 
or require $Q$ to have large weights. 
In particular, the family $\mathbf{Acyc}_n$ (cf.\ \eqref{eq:Acyc_n}) is fully generic. 
\end{example}

\begin{definition}
\label{def:generic-mut-cycle}
We say that a quiver family 
$\mathbf{Q}\subset\Quiv_n$ and a sequence $\mathbf{i}=i_1\cdots i_k$ define a \emph{fully generic mutation cycle} if
\begin{itemize}[leftmargin=.2in]
\item 
$\mathbf{Q}$ is a fully generic family of quivers; 
\item
for any $Q\in\mathbf{Q}$, the sequence $\mathbf{i}$ is a mutation cycle based at~$Q$, cf.\ Definition~\ref{def:mutation cycle}. 
\end{itemize}
\end{definition}

\begin{example}
\label{eg:acyclic-generic-cycle}
%Recall from \eqref{eq:Acyc_n} that $\mathbf{Acyc}_n$ is the family of all acyclic quivers on $n$ vertices, with standard orientation and large weights. 
Each acyclic quiver $Q\in\mathbf{Acyc}_n$ lies on the mutation cycle $\mathbf{i}=12\cdots n$, 
see Proposition~\ref{pr:acyclic-mutation-cycle}.
Since $\mathbf{Acyc}_n$ is a fully generic family (see Example~\ref{eg:fixed-orientation-generic}),
this mutation cycle is generic. 
\end{example}

\begin{example}[{cf.\ \cite[Example~12.30]{Warkentin}}]
\label{eg:6-cycle-vortices}\
Figure~\ref{fig:generic-6-cycle} shows a fully generic mutation cycle of length~6 consisting
of 4-vertex quivers, all of them vortices. 
\end{example}

\newcommand{\sns}{-stealth}
\newcommand{\snsl}{stealth-} 

\begin{figure}[ht]
\vspace{-.1in}
\begin{equation*}
\begin{array}{ccccc}
\hspace{-17pt}Q=\!\!
\begin{tikzcd}[arrows={-stealth}, column sep=60, row sep=60]
  1  \arrow[r,  "a", \snsl] \arrow[dr, "d" near start, outer sep=-.8, \snsl]  \arrow[d, "f", \snsl]  
  & 2  \arrow[d, swap, "b+ce", \snsl] \arrow[dl, "e", near end, outer sep=-.8, \sns]
  \\
   4 
   & 3  \arrow[l, "c", swap, \snsl] 
\end{tikzcd}

 &\mutation 4 &
 \begin{tikzcd}[arrows={-stealth}, column sep=60, row sep=60]
  1  \arrow[r,  "a+ef", \snsl] \arrow[dr, "d" near start, outer sep=-.8, \snsl]  \arrow[d, "f", \sns]  
  & 2  \arrow[d, swap, "b", \snsl] \arrow[dl, "e", near end, outer sep=-.8, \snsl]
  \\
   4 
   & 3  \arrow[l, "c", swap, \sns] 
\end{tikzcd}
 
 & \mutation 3 
 &  \begin{tikzcd}[arrows={-stealth}, column sep=60, row sep=60]
  1  \arrow[r,  "a+ef", \snsl] \arrow[dr, "d" near start, outer sep=-.8, \sns]  \arrow[d, "f", \sns]  
  & 2  \arrow[d, swap, "b", \sns] \arrow[dl, "e", near end, outer sep=-.8, \snsl]
  \\
   4 
   & 3  \arrow[l, "c", swap, \snsl] 
\end{tikzcd}
\\
\\[-10pt]
\vmutation{1}   & & & & \vmutation{4} \\[7pt]
\begin{tikzcd}[arrows={-stealth}, column sep=60, row sep=60]
  1  \arrow[r,  "a", \sns] \arrow[dr, "d" near start, outer sep=-.8, \sns]  \arrow[d, "f", \sns]  
  & 2  \arrow[d, swap, "b+ce", \snsl] \arrow[dl, "e", near end, outer sep=-.8, \sns]
  \\
   4 
   & 3  \arrow[l, "c", swap, \snsl] 
\end{tikzcd}
& \mutation 4 
&  \begin{tikzcd}[arrows={-stealth}, column sep=60, row sep=60]
  1  \arrow[r,  "a", \sns] \arrow[dr, "d+cf" near start, outer sep=-2, \sns]  \arrow[d, "f", \snsl]  
  & 2  \arrow[d, swap, "b", \snsl] \arrow[dl, "e", near end, outer sep=-.8, \snsl]
  \\
   4 
   & 3  \arrow[l, "c", swap, \sns] 
\end{tikzcd}
& \mutation 2 
& \begin{tikzcd}[arrows={-stealth}, column sep=60, row sep=60]
  1  \arrow[r,  "a", \snsl] \arrow[dr, "d+cf" near start, outer sep=-2, \sns]  \arrow[d, "f", \snsl]  
  & 2  \arrow[d, swap, "b", \sns] \arrow[dl, "e", near end, outer sep=-.8, \sns]
  \\
   4 
   & 3  \arrow[l, "c", swap, \sns] 
\end{tikzcd}
\end{array}
\end{equation*}
\vspace{-.1in}
\caption{A fully generic mutation cycle involving 6 vortices.
Here ${a,b,c,d,e,f \geq 0}$. 
The parametrization $(a,b,c,d,e,f)\mapsto Q$ 
is a $\ZZ$-biregular map $\ZZ_{\ge0}^6\to\mathbf{Q}$, 
where $\mathbf{Q}$~consists of all 4-vertex quivers oriented as shown at the upper-left
and satisfying $b_{43}(Q)\ge b_{23}(Q)b_{42}(Q)$. 
}
\label{fig:generic-6-cycle}
\vspace{-.2in}
\end{figure}
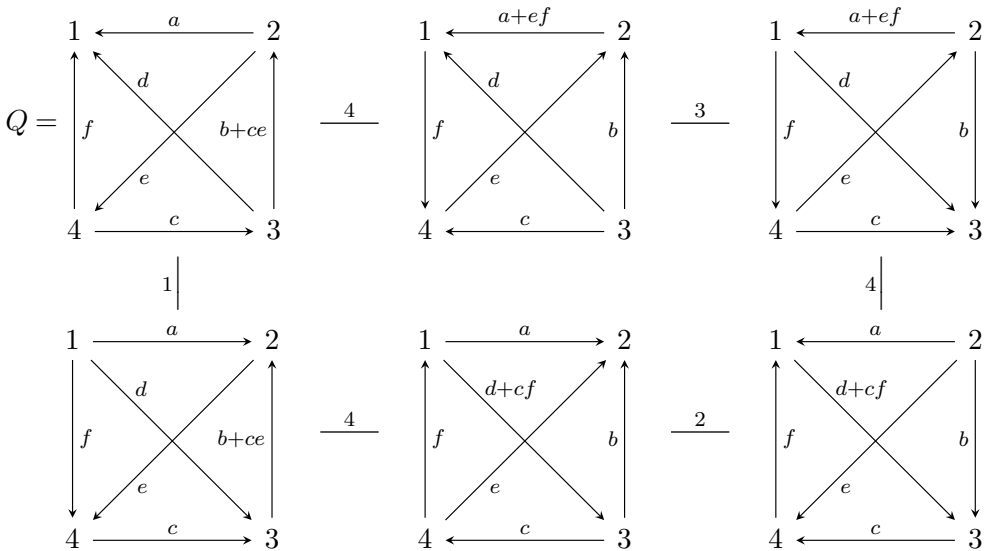

\pagebreak[3]

\begin{example}[\emph{Non-example \#1:} Mutation cycles of type~$A_2$]
\label{eg:A2-nongeneric-cycle}
Let $\mathbf{Q}=\{Q\in\Quiv_n\mid b_{12}(Q)=1\}$.
Then $\mathbf{i}=(12)^5$ is a mutation cycle, see \cite{FWZ, ca2}. 
However, the family of quivers $\mathbf{Q}$ is not fully generic [why?]. 
\end{example}

\begin{example}[\emph{Non-example \#2:} Fordy-Marsh mutation cycles]
\label{eg:fordy-marsh-nongeneric-cycle}
Mutation cycles constructed in \cite{Fordy-Marsh} are not fully generic:
the weights in the quivers that they consider need to satisfy some algebraic constraints. 
\end{example}

\begin{remark}
%\label{rem:}
While we will not rely on this statement, it can be shown that the property of being fully generic does not depend 
on the choice of a base point on a mutation cycle. 
Thus, replacing $\mathbf{Q}$ (resp., $\mathbf{i}=i_1\cdots i_k$) 
by $\mu[i_1](\mathbf{Q})$ (resp., $\mathbf{i}'=i_2\cdots i_k i_1$)
produces another fully generic mutation cycle, etc. 
\end{remark}

For $n\ge 4$ and~$k\ge 1$, 
we denote by $\mathbf{Q}_{n,k}\subset\mathbf{Q}_n$ the family of quivers~$Q$ described in Theorem~\ref{thm:GeneralSemiCycles}. 

\begin{theorem}
\label{th:long-cycle-generic}
For any $n\ge 4$ and~$k\ge 1$, 
the quiver family $\mathbf{Q}_{n,k}$ is fully generic. % cf.\ Definition~\ref{def:generic-quivers}. 
Consequently, the mutation cycle in Theorem~\ref{thm:GeneralSemiCycles} is fully generic. 
\end{theorem}

Before proving Theorem~\ref{th:long-cycle-generic}, we illustrate it by examining its simplest instance.  

\begin{example}
For $n$=4 and $k=1$, we get the mutation cycle of length~8 shown in 
Figure~\ref{fig:generic-8-cycle}, with $a,b,c,d,e,f\ge2$.
The family $\mathbf{Q}_{4,1}$ consists of all 4-vertex quivers~$Q$ of the form 
shown in the upper-left corner of Figure~\ref{fig:generic-8-cycle}. 
To see that this family is fully generic, we will verify that the map $(a,b,c,d,e,f)\mapsto Q$ 
gives a $\ZZ$-biregular parametrization of $\mathbf{Q}_{4,1}$ by the lattice points in~$\ZZ_{\ge2}^6$. 
This amounts to checking that the inverse map $(b_{ij})\mapsto (a,b,c,d,e,f)$,
where $(b_{ij})=B(Q)$ (see Definition~\ref{def:B(Q)}), is given by polynomials over~$\ZZ$. 
It is straightforward to check that, indeed,
\begin{align*}
a&=b_{12}, \quad \\
b&=-b_{23}+b_{12}b_{31}, \quad \\
c&=b_{34},\quad \\
d&=b_{31}+b_{12}b_{23}-b_{12}^2b_{31}, \quad \\
e&=b_{24}, \quad \\
f&=b_{14}. 
\end{align*}
\end{example}

\begin{proof}[Proof of Theorem~\ref{th:long-cycle-generic}]
We will prove this theorem by constructing a $\ZZ$-biregular parametrization 
$\ZZ_{\ge2}^{\binom{n}{2}} \to \mathbf{Q}_{n,k}$. 
The parameters of this parametrization are: 
\begin{itemize}[leftmargin=.2in]
\item 
the weights $b_{ij}(R)$ ($1\le i< j\le n-1$) of the subquiver~$\tilde R$, and 
\item
the numbers $b_{in}(Q)$ for $i\in [1,n-1]$. 
\end{itemize}
(These are precisely the parameters $q_{ij}$ in Theorem~\ref{thm:Summary}.)
By Lemma~\ref{lem:mut-acyclic-orientations}, the map 
\begin{align*}
\mu[(12)^k]: \mathbf{Acyc}_{n-1} &\longrightarrow \mu[(12)^k](\mathbf{Acyc}_{n-1}) \\
\tilde R &\longrightarrow Q|_{[1,n-1]}
\end{align*}
is $\ZZ$-biregular. 
Adding the $n-1$ parameters $b_{in}(Q)$, we obtain the claim. 
\end{proof}

\pagebreak[3]

It is not hard to show that for $n\ge5$ and $k\ge1$, 
different choices of $\binom{n}{2}$ parameters in Theorem~\ref{thm:Summary}/\ref{thm:GeneralSemiCycles} 
produce mutation cycles that are different 
from each other even if we allow rotating the cycle, reversing the direction along it, 
treating quivers up to isomorphism, and/or global reversal of arrows. 
The $n=4$ case is a bit different: 

\begin{proposition}
\label{pr:double-cycle}
Let $n=4$ and $k\ge1$. 
Let $Q=Q(a,b,c,d,e,f)$ denote the quiver constructed as in Theorem~\ref{thm:GeneralSemiCycles}, 
with parameters
\begin{equation*}
\begin{array}{lll}
b_{12}(\tilde R)\!=\!a, \quad & b_{13}(\tilde R)\!=\!d,   \quad & b_{14}(Q)\!=\!f, \\[2pt]
& b_{23}(\tilde R)\!=\!b, \quad & b_{24}(Q)\!=\!e, \\[2pt]
& & b_{34}(Q)\!=\!c. 
\end{array}
\end{equation*}
%$b_{12}(\tilde R)\!=\!a$, 
%$b_{13}(\tilde R)\!=\!d$, 
%$b_{23}(\tilde R)\!=\!b$, 
%$b_{14}(Q)\!=\!f$, 
%${b_{24}(Q)\!=\!e}$, 
%$b_{34}(Q)\!=\!c$, 
Let $\mathcal{C}(a,b,c,d,e,f)$ be the mutation cycle~\eqref{eq:GeneralSemiCycle} based at $Q(a,b,c,d,e,f)$, cf.\ Figure~\ref{fig:generic-8-cycle}. 
Transform this mutation cycle as follows: 
\begin{itemize}[leftmargin=.2in]
\item[-]
Move the base point to the opposite quiver~$Q^{(2k+2)}$ (cf.\ Figure~\ref{fig:GeneralSemiCycles-notation}).
\item[-]
Reverse the arrows of all quivers.
\item[-]
Relabel their vertices using the permutation $1\leftrightarrow2$, $3\leftrightarrow4$.
\item[-]
Relabel the mutations in the same way.
\item[-]
Reverse the direction along the cycle from clockwise to counterclockwise.
\end{itemize}
Then the resulting mutation cycle is $\mathcal{C}(a,f,c,e,d,b)$. %$was: \mathcal{C}(a,e,b,f,d,c)$. 
\end{proposition}

We omit the proof, which is a straightforward application of Proposition~\ref{pr:3VertexAlternatingMutations}. 

\begin{example}
Consider the case $k=1$ of Proposition~\ref{pr:double-cycle}. 
In this case, the quiver $Q^{(2k+2)}=Q^{(4)}$ appears in the bottom-right corner of Figure~\ref{fig:generic-8-cycle}. 
%, between $Q$ in the top-left and the quiver in the bottom right.
%The operations described in Proposition~\ref{pr:double-cycle} change the sink at $4$ in $Q$ into a source at $3$.
Changing the labels and reversing all arrows produces the following quiver:
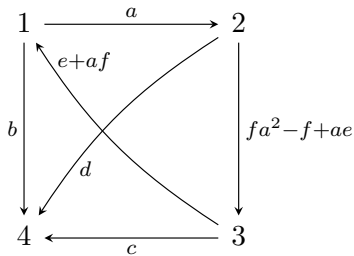
\begin{equation*}
\qquad 
\begin{tikzcd}[arrows={-stealth}, sep=6em]
  1  \arrow[r,  "a"]
  & 2  \arrow[d,  "fa^2-f+ae"]
  \\
   4 \arrow[ur, "d", swap, bend left=12, near start, outer sep=-1.8, stealth-] \arrow[u, "b", stealth-] \arrow[r, "c", swap, stealth-]  
   & 3  \arrow[ul, "e+af", swap, bend left=12, very near end, outer sep=-1.8]
\end{tikzcd}
\end{equation*}
Swapping $b\leftrightarrow f$ and $d\leftrightarrow e$ everywhere recovers the original quiver~$Q$
shown in the upper-left corner of Figure~\ref{fig:generic-8-cycle}, in agreement with Proposition~\ref{pr:double-cycle}. 
In other words, all our theorems about~$Q$ apply to~$Q^{(4)}$ as well, after a suitable relabeling. 
\end{example}

\begin{theorem}
For $n=4$ and $k \geq 1,$ let $Q$ be a quiver constructed as in Theorem~\ref{thm:GeneralSemiCycles}. The mutation cycle described in Theorem~\ref{thm:GeneralSemiCycles} is the unique simple mutation cycle in its mutation class.
\end{theorem}

\begin{proof}
This mutation cycle is simple by Theorems~\ref{thm:GeneralSemiCycles} and~\ref{thm:GeneralSemiCycleDistinct}.

Let us show that every mutation away from the cycle is an exit. 
That is, we claim that for every $0 \leq j < 4k+4$ and every vertex $v \in [1,4]$, 
either $\T{v}{Q^{(j)}} = Q^{(j\pm1)}$ or $Q^{(j)}$ has exit $v$.
By Lemma~\ref{lem:cycle-exits} applied to $Q$ (respectively $Q^{(2k+2)}$), the claim holds for $2k+4 < j \leq 4k+3$ and $0 \leq j \leq 2k$ (respectively $2 < j \leq 4k+2$).
Since $k \ge 1$, the claim holds for all $j$.
\end{proof}

\newpage

\section{No cycles in exchange graphs}
\label{sec:no-seed-cycles}

It is natural to ask whether any of the mutation cycles discussed above give rise to cycles of seed mutations
(equivalently, cycles in the exchange graph of an associated cluster algebra). 
It turns out that they do not: 

\begin{theorem}
\label{th:no-seed-cycle}
Let $C$ be a mutation cycle of $n$-vertex quivers ($n\geq 2$) such that all quivers along the cycle have large weights. 
Then $C$ does not underlie a mutation cycle of seeds in the associated cluster algebra (with arbitrary coefficients).

In particular, the mutation cycle in Theorem~\ref{thm:GeneralSemiCycles} is never a mutation cycle of seeds.
\end{theorem}

\begin{proof}
We replicate the ``tropical degeneration'' approach used in the proof of \cite[Theorem~5.1.1]{FWZ} and in many other places. 
Roughly, the idea is to track the degrees of cluster variables with respect to a particular initial variable, 
verifying that these degrees increase strictly as we go around the cycle. 

It is enough to consider the cluster algebra~$\mathcal{A}$ with trivial coefficients, as the general case will follow. 

Assume that our cycle $C$ is based at a quiver $Q$, with mutation sequence $i_1 i_2\cdots$: 
\begin{equation*}
Q=Q^{(0)} \mutation{i_1}
Q^{(1)} \mutation{i_2}
Q^{(2)} \mutation{i_3}
\cdots
.
\end{equation*}
Let $(Q,\xx)$ be the initial seed in~$\mathcal{A}$, with $\xx\!=\!(x_1,\dots,x_n)$ the initial cluster.
Let $(Q^{(\ell)},\xx^{(\ell)})$ be the seed obtained after $\ell$ mutations as above, with $\xx^{(\ell)}=(x^{(\ell)}_1,\dots,x^{(\ell)}_n)$. 
Then every $x^{(\ell)}_i$ is a Laurent polynomial in $x_1,\dots,x_n$. 
Assume, without loss of generality, that $i_1 \neq 1$. 
Let $\deg_1$ denote the degree with respect to the variable~$x_1$. 

We will argue by induction on $\ell$ that $\deg_1(x_{i_{\ell+1}}^{(\ell +1)}) > \deg_1(x_j^{(\ell)}) \geq 0$ for all~$j$.
This will imply that, as the degrees of cluster variables keep increasing, there may be no cycle in the exchange graph.

\emph{Base:} $\ell = 0$.
%We have $\deg_1(x_j^{(0)}) \geq 0$ by construction.
Since $i_1 \neq 1$, we have 
$x_{i_1}^{(1)} = \frac{ x_1^{b}m  + n}{ x_{i_1}}$, 
where $b=|b_{1i_1}(Q)|\ge 2$ 
and $m$ and $n$ are monomials in $x_2, \ldots, x_n$. 
Therefore 
\begin{equation*}
\deg_1(x_{i_1}^{(1)}) = b > 1 = \deg_1(x_1^{(0)}) \geq \deg_1(x_j^{(0)})
\end{equation*}
for all $j$.

\emph{Induction step.} 
Suppose that $\deg_1(x_{i_{\ell}}^{(\ell)}) > \deg_1(x_j^{(\ell-1)}) \geq 0$. 
Since $x_{j}^{(\ell)} = x_j^{(\ell-1)}$ for $j \neq i_{\ell}$, we conclude that $\deg_1(x_j^{(\ell)}) \geq 0$ for all~$j$.
Since $i_{\ell+1} \neq i_{\ell}$, it follows that 
\begin{equation*}
x_{i_{\ell+1}}^{(\ell+1)} = \frac{ (x_{i_{\ell}}^{(\ell)})^b\,m  + n}{ x_{i_{\ell+1}}^{(\ell)}},
\end{equation*}
where $b=|b_{i_{\ell}i_{\ell+1}}(Q^{(\ell)})|\ge2$ and $m$ and $n$ are monomials in $x_1^{(\ell)}, \ldots, x_n^{(\ell)}$. 
Since $\deg_1(x_j^{(\ell)}) \geq 0$ for all $j$, we have $\deg_1(m)\ge0$ and $\deg_1(n) \geq 0$.
Consequently  
\begin{align*}
\deg_1(x_{i_{\ell+1}}^{(\ell+1)}) 
&\geq b \deg_1(x_{i_{\ell}}^{(\ell)}) - \deg_1(x_{i_{\ell+1}}^{(\ell)}) \\
&=b \deg_1(x_{i_{\ell}}^{(\ell)}) - \deg_1(x_{i_{\ell+1}}^{(\ell-1)})
> \deg_1(x_{i_{\ell}}^{(\ell)}) 
\geq \deg_1(x_j^{(\ell)})
\end{align*}
for all $j$, as desired.
\end{proof}

\newpage

\section{A gallery of mutation cycles}
\label{sec:moreCycles}

In this section, we discuss several additional families of mutation cycles. 

\begin{example}
\label{eg:rosette}
Figure~\ref{fig:generic-12-cycle-cyclic} shows a fully generic mutation cycle of length~12. 
%with a dihedral symmetry of order~4. 

Let $\mathbf{C}=\mathbf{C}(a,b,c,d,e,f)$ denote the mutation class containing this mutation cycle. 
In can be shown that (a) this mutation cycle is unique within~$\mathbf{C}$, 
(b) none of the quivers in $\mathbf{C}$ has a sink/source (in particular, none is acyclic), but 
(c) every proper subquiver of every quiver in~$\mathbf{C}$ is mutation-acyclic. 
\end{example}

\newcommand{\sepps}{3.5em}
\newcommand{\outsepd}{-1}
\newcommand{\outsepe}{-1.5}

\begin{figure}[ht]
\begin{equation*}
\hspace{-10pt}
\begin{array}{ccccccccc}
\begin{tikzcd}[arrows={-stealth}, sep=\sepps]
  1  \arrow[r,  "a", -stealth] \arrow[dr, "d" near start, outer sep=\outsepd, -stealth]  \arrow[d, "f", stealth-]  
  & 2  \arrow[d, swap, "b", -stealth] \arrow[dl, "\overline{e}", near end, outer sep=\outsepe, stealth-]
  \\
   4 
   & 3  \arrow[l, "c_{\ssquare}", swap, -stealth] 
\end{tikzcd}
 & \shortmutation{1}
 &
\begin{tikzcd}[arrows={-stealth}, sep=\sepps]
  1  \arrow[r,  "a", stealth-] \arrow[dr, "d" near start, outer sep=\outsepd, stealth-]  \arrow[d, "f", -stealth]  
  & 2  \arrow[d, swap, "b", -stealth] \arrow[dl, "\widehat{e}", near end, outer sep=\outsepe, stealth-]
  \\
   4 
   & 3  \arrow[l, "c_\circ", swap, -stealth] 
\end{tikzcd}
 & \shortmutation{2}
 & 
\begin{tikzcd}[arrows={-stealth}, sep=\sepps]
  1  \arrow[r,  "a"] \arrow[dr, "d" near start, outer sep=\outsepd, stealth-]  \arrow[d, "f_\bullet", stealth-]  
  & 2  \arrow[d, swap, "b", stealth-] \arrow[dl, "\widehat{e}", near end, outer sep=\outsepe]
  \\
   4 
   & 3  \arrow[l, "c", swap, stealth-] 
\end{tikzcd}
 & \shortmutation{3}
 &
\begin{tikzcd}[arrows={-stealth}, sep=\sepps]
  1  \arrow[r,  "a"] \arrow[dr, "d" near start, outer sep=\outsepd]  \arrow[d, "f_{\ssquare}", stealth-]  
  & 2  \arrow[d, swap, "b"] \arrow[dl, "e", near end, outer sep=\outsepe]
  \\
   4 
   & 3  \arrow[l, "c", swap] 
\end{tikzcd}
\\
\\[-8pt]
\vmutation{2}   & & & & & & \vmutation{2} \\[-6pt]
\\
\begin{tikzcd}[arrows={-stealth}, sep=\sepps]
  1  \arrow[r,  "a", stealth-] \arrow[dr, "\widehat{d}" near start, outer sep=\outsepd, -stealth]  \arrow[d, "f", stealth-]  
  & 2  \arrow[d, swap, "b", stealth-] \arrow[dl, "\overline{e}", near end, outer sep=\outsepe, -stealth]
  \\
   4 
   & 3  \arrow[l, "c_\bullet", swap, -stealth] 
\end{tikzcd}
 &\multicolumn{2}{c}{
 \begin{aligned} 
 c_\bullet &\!=\! abf\!+\!df\!-\!c\\
 c_{\nsquare}  &\!=\! abf\!+\!b\overline{e}\!+\!df\!-\!c\\
c_\circ &\!=\! abf\!+\!b\overline{e}\!-\!c
\end{aligned}
\hspace{-20pt}
 } 
& &
\multicolumn{2}{c}{
\hspace{-15pt}
 \begin{aligned}
  f_\bullet &\!=\! abc\!+\!ae\!-\!f \\
 f_{\nsquare} &\!=\! abc\!+\!ae\!+\!cd \!-\!f  \\
 f_\circ &\!=\! abc\!+\!cd \!-\!f 
 %&= f_\circ+ae,\\
 \end{aligned}
 }
 &  
\begin{tikzcd}[arrows={-stealth}, sep=\sepps]
  1  \arrow[r,  "a", stealth-] \arrow[dr, "\widehat{d}", near start, outer sep=\outsepd]  \arrow[d, "f_\circ", stealth-]  
  & 2  \arrow[d, swap, "b",stealth-] \arrow[dl, "e", near end, outer sep=\outsepe,stealth-]
  \\
   4 
   & 3  \arrow[l, "c", swap] 
\end{tikzcd}
 & \\
 \\[-8pt]
\vmutation{1}   & 
&\multicolumn{1}{c}{\begin{aligned} \overline{e} &\!=\! e \!+\! bc \!-\!af \\
\widehat{e}&\!=\!e\!+\!bc
\end{aligned}} & & 
\multicolumn{1}{c}{\begin{aligned} \overline{d} &\!=\! d \!+\! ab \!-\!cf \\
\widehat{d} &\!=\!d\!+\!ab \end{aligned} }& &\vmutation{3} \\[-6pt]
\\
\begin{tikzcd}[arrows={-stealth}, sep=\sepps]
  1  \arrow[r,  "a", -stealth] \arrow[dr, "\widehat{d}" near start, outer sep=\outsepd, stealth-]  \arrow[d, "f", -stealth]  
  & 2  \arrow[d, swap, "b_\circ", -stealth] \arrow[dl, "\overline{e}", near end, outer sep=\outsepe, -stealth]
  \\
   4 
   & 3  \arrow[l, "c", swap, stealth-] 
\end{tikzcd}
 &\multicolumn{2}{c}{
 \begin{aligned}
  b_\bullet &\!=\! acf\!+\!c\overline{e}\!-\!b \\
 b_{\nsquare} &\!=\! acf\!+\!a\overline{d}\! +\!c\overline{e}\!-\!b\\
b_\circ &\!=\! acf\!+\!a\overline{d} \!-\!b 
 \end{aligned}
 \hspace{-20pt}
}
& &\multicolumn{2}{c}{
\hspace{-15pt}
 \begin{aligned}
  a_\bullet &\!=\! bcf\!+\!b\overline{d}\!-\!a \\
 a_{\nsquare} &\!=\! bcf\!+\!b\overline{d}\!+\!ef\!-\!a \\
 a_\circ &\!=\! bcf\!+\!ef\!-\!a
\end{aligned}
 }
 
&  
\begin{tikzcd}[arrows={-stealth}, sep=\sepps]
  1  \arrow[r,  "a_\bullet", -stealth] \arrow[dr, "\widehat{d}", near start, outer sep=\outsepd, stealth-]  \arrow[d, "f", -stealth]  
  & 2  \arrow[d, swap, "b",-stealth] \arrow[dl, "e", near end, outer sep=\outsepe,stealth-]
  \\
   4 
   & 3  \arrow[l, "c", swap, stealth-] 
\end{tikzcd}
 & \\ 
\\[-8pt]
\vmutation{4}   & & & & & & \vmutation{4} \\[-6pt]
\\
\begin{tikzcd}[arrows={-stealth}, sep=\sepps]
  1  \arrow[r,  "a", -stealth] \arrow[dr, "\overline{d}", near start, outer sep=\outsepd, stealth-]  \arrow[d, "f", stealth-]  
  & 2  \arrow[d, swap, "b_{\ssquare}",-stealth] \arrow[dl, "\overline{e}", near end, outer sep=\outsepe,stealth-]
  \\
   4 
   & 3  \arrow[l, "c", swap, -stealth] 
\end{tikzcd}
& \shortmutation 1
& 
\begin{tikzcd}[arrows={-stealth}, sep=\sepps]
  1  \arrow[r,  "a", stealth-] \arrow[dr, "\overline{d}", near start, outer sep=\outsepd, -stealth]  \arrow[d, "f", -stealth]  
  & 2  \arrow[d, swap, "b_\bullet",-stealth] \arrow[dl, "\widehat{e}", near end, outer sep=\outsepe,stealth-]
  \\
   4 
   & 3  \arrow[l, "c", swap, -stealth] 
\end{tikzcd}
& \shortmutation 4 
&  
\begin{tikzcd}[arrows={-stealth}, sep=\sepps]
  1  \arrow[r,  "a_\circ", -stealth] \arrow[dr, "\overline{d}", near start, outer sep=\outsepd, -stealth]  \arrow[d, "f", stealth-]  
  & 2  \arrow[d, swap, "b",stealth-] \arrow[dl, "\widehat{e}", near end, outer sep=\outsepe,-stealth]
  \\
   4 
   & 3  \arrow[l, "c", swap, stealth-] 
\end{tikzcd}
& \shortmutation 3 
&  
\begin{tikzcd}[arrows={-stealth}, sep=\sepps]
  1  \arrow[r,  "a_{\ssquare}", -stealth] \arrow[dr, "\overline{d}", near start, outer sep=\outsepd, stealth-]  \arrow[d, "f", stealth-]  
  & 2  \arrow[d, swap, "b",-stealth] \arrow[dl, "e", near end, outer sep=\outsepe,-stealth]
  \\
   4 
   & 3  \arrow[l, "c", swap, -stealth] 
\end{tikzcd}
\end{array}
\hspace{-22pt}{\ }
\end{equation*}
\vspace{-.1in}
\caption{
Fully generic mutation cycle of length~12 involving 4-vertex quivers.  
The integer parameters $a,b,c,d,e,f$ can be chosen as follows. 
First pick any $a,b,c,f\ge 2$. 
Then choose $d\ge\max(2, 2-ab+cf)$ and $e\ge\max(2,2+af-bc)$
to make sure that $\overline{d} = d +ab-cf \geq 2$ and $\overline{e} = e+bc-af \geq 2$.
%The last constraints can be assured by simple linear inequalities, such as $b\geq c \geq a \geq f$.
The formulas expressing the weights in terms of $a,b,c,d,\overline{d}, e, \overline{e}, f$ 
appear in the center of the diagram.
%Each of these parameters appears somewhere along the cycle as an individual arrow multiplicity.
}
\label{fig:generic-12-cycle-cyclic}
\vspace{-.2in}
\end{figure}
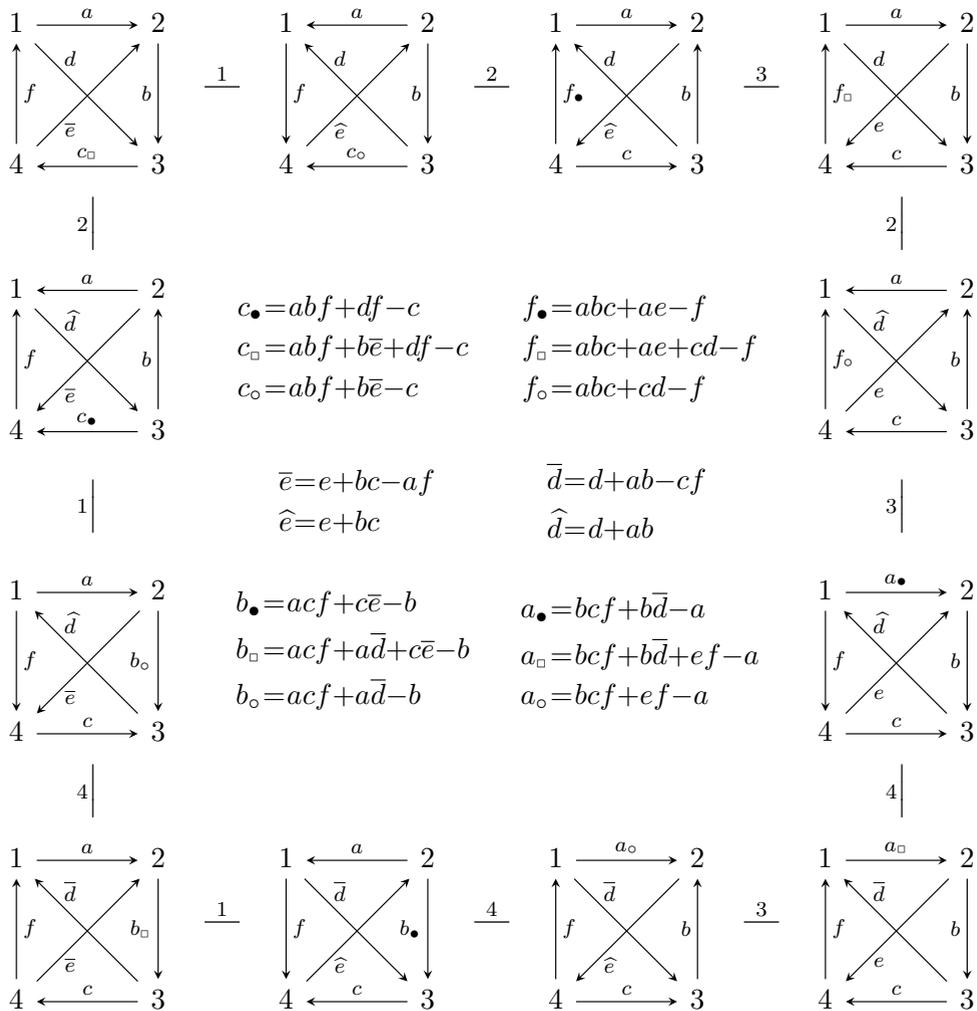

\begin{example}
\label{eg:rosette-2}
Figure~\ref{fig:generic-12-cycle-fractured-cyclic} shows another fully generic mutation cycle of length~12. 
Properties (a)--(c) from Example~\ref{eg:rosette} hold in this example as well. 

\end{example}

\newcommand{\snls}{-stealth}

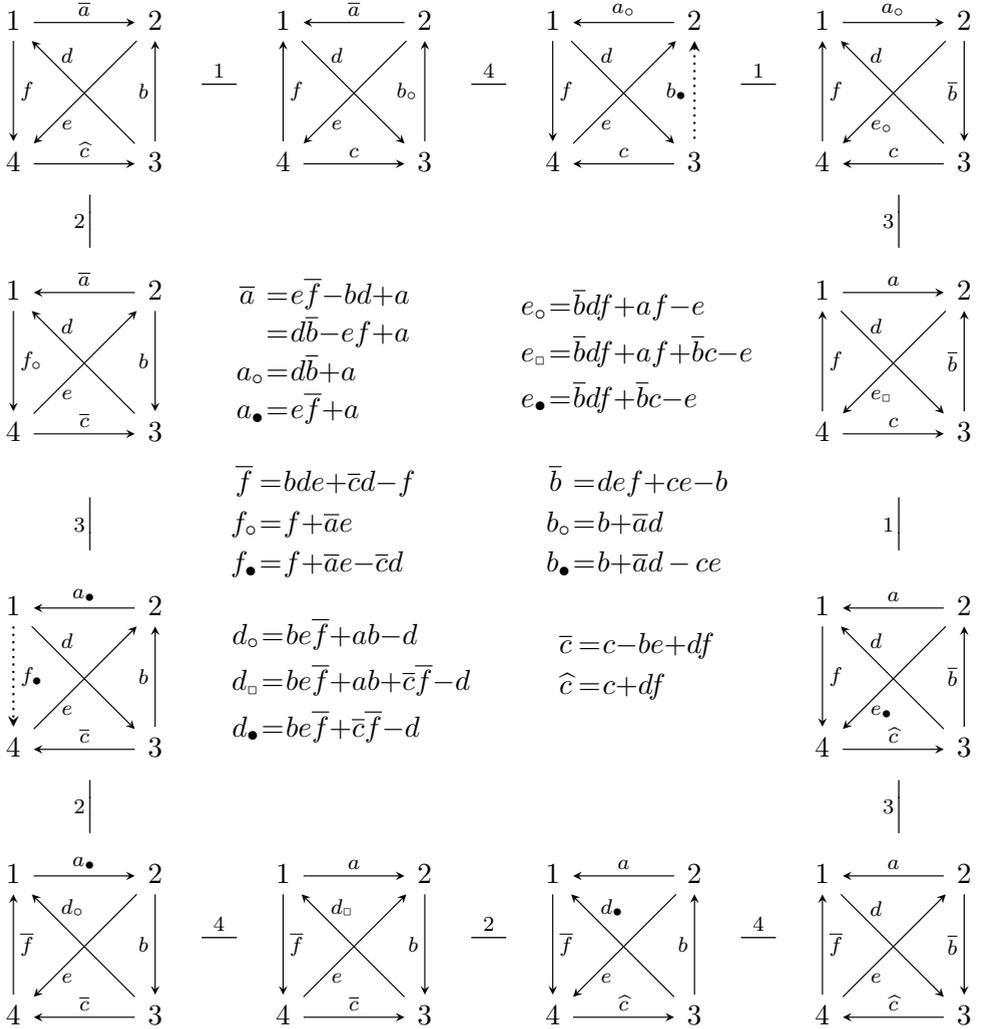
\begin{figure}[ht]
\begin{equation*}
\hspace{-10pt}
\begin{array}{ccccccccc}

\begin{tikzcd}[arrows={-stealth}, sep=\sepps]
  1  \arrow[r,  "\overline{a}", -stealth] \arrow[dr, "d" near start, outer sep=\outsepd, stealth-]  \arrow[d, "f", -stealth]  
  & 2  \arrow[d, swap, "b", stealth-] \arrow[dl, "e", near end, outer sep=\outsepe, -stealth]
  \\
   4 
   & 3  \arrow[l, "\widehat{c}", swap, stealth-] 
\end{tikzcd}
 & \shortmutation{1}
 &
 \begin{tikzcd}[arrows={-stealth}, sep=\sepps]
  1  \arrow[r,  "\overline{a}", stealth-] \arrow[dr, "d" near start, outer sep=\outsepd, -stealth]  \arrow[d, "f", stealth-]  
  & 2  \arrow[d, swap, "b_\circ", stealth-] \arrow[dl, "e", near end, outer sep=\outsepe, -stealth]
  \\
   4 
   & 3  \arrow[l, "c", swap, \snsl] 
\end{tikzcd}
 & \shortmutation{4}
 & 
 
\begin{tikzcd}[arrows={-stealth}, sep=\sepps]
  1  \arrow[r,  "a_{\circ}", \snsl] \arrow[dr, "d" near start, outer sep=\outsepd, \snls]  \arrow[d, "f", \snls]  
  & 2  \arrow[d, swap, "b_\bullet", thick, dotted, \snsl] \arrow[dl, "e", near end, outer sep=\outsepe, \snsl]
  \\
   4 
   & 3  \arrow[l, "c", swap, \snls] 
\end{tikzcd}
 
 & \shortmutation{1}
 &
\begin{tikzcd}[arrows={-stealth}, sep=\sepps]
  1  \arrow[r,  "a_{\circ}", \snls] \arrow[dr, "d" near start, outer sep=\outsepd, \snsl]  \arrow[d, "f", \snsl]  
  & 2  \arrow[d, swap, "\overline{b}", \snls] \arrow[dl, "e_{\circ}", near end, outer sep=\outsepe, \snls]
  \\
   4 
   & 3  \arrow[l, "c", swap, \snls] 
\end{tikzcd}
\\
\\[-8pt]
\vmutation{2}   &\multicolumn{2}{c}{
\begin{aligned}
 \end{aligned}}&
  & & & \vmutation{3} \\[-6pt]
\\

\begin{tikzcd}[arrows={-stealth}, sep=\sepps]
  1  \arrow[r,  "\overline{a}", \snsl] \arrow[dr, "d" near start, outer sep=\outsepd, \snsl]  \arrow[d, "f_\circ", \snls]  
  & 2  \arrow[d, swap, "b", \snls] \arrow[dl, "e", near end, outer sep=\outsepe, \snsl]
  \\
   4 
   & 3  \arrow[l, "\overline{c}", swap, \snsl] 
\end{tikzcd}

 &\multicolumn{2}{c}{
 \begin{aligned}
 \overline{a}\,\,&\!=\! %\widehat{c} d e\! -\!bd\!-\!ef \!+\! a
 e\overline{f}\!-\!bd\!+\!a\\[-3pt]
 &\!=\! d\overline{b}\!-\!ef\!+\!a \\[-3pt]
 a_\circ&\!=\! d\overline{b}\!+\!a \\[-3pt]
 a_\bullet&\!=\! e\overline{f}\!+\!a
 \end{aligned}
 }
& &\multicolumn{2}{c}{
\hspace{-15pt}
\begin{aligned}
 e_{\circ} &\!=\!  \overline{b}df  \!+\!af\!-\!e\\
 e_{\nsquare}%&= e_{\circ}+ c \widehat{c} e - bc\\
 &\!=\! \overline{b}df\!+\!af\!+\!\overline{b}c\! -\! e\\
 e_\bullet &\!=\! \overline{b}df\!+\!\overline{b}c\! -\! e\\
 \end{aligned}
 }
 
&  
 
\begin{tikzcd}[arrows={-stealth}, sep=\sepps]
  1  \arrow[r,  "a", \snls] \arrow[dr, "d" near start, outer sep=\outsepd, \snls]  \arrow[d, "f", \snsl]  
  & 2  \arrow[d, swap, "\overline{b}", \snsl] \arrow[dl, "e_{\ssquare}", near end, outer sep=\outsepe, \snls]
  \\
   4 
   & 3  \arrow[l, "c", swap, \snsl] 
\end{tikzcd}
 & \\
 \\[-8pt]
\vmutation{3}   & 
\multicolumn{2}{c}{
\begin{aligned} 
 \overline{f} \,&\!=\! bde\!+\!\overline{c}d\!-\!f\\
f_\circ&\!=\!f\! +\! \overline{a}e\\ %&=\!f\!+ \widehat{c}de^2\!-\!bde\!-\!e^2f\!+\!ae\\%
f_\bullet &\!=\! f \!+\! \overline{a} e\! - \! \overline{c}d\\
 \end{aligned} }& &
 \multicolumn{2}{c}{
\hspace{-15pt}
\begin{aligned}
 \overline{b}\,\, &\!=\! def\!+\!ce\!-\!b\\
b_\circ&\!=\! b \!+\! \overline{a}d \\
b_\bullet&\!=\!b\!+\!\overline{a} d-ce\\
 \end{aligned}
}&  
\vmutation{1}\\
\begin{tikzcd}[arrows={-stealth}, sep=\sepps]
  1  \arrow[r,  "a_\bullet", \snsl] \arrow[dr, "d" near start, outer sep=\outsepd, \snls]  \arrow[d, "f_\bullet", thick, dotted, \snls]  
  & 2  \arrow[d, swap, "b", \snsl] \arrow[dl, "e", near end, outer sep=\outsepe, \snsl]
  \\
   4 
   & 3  \arrow[l, "\overline{c}", swap, \snls] 
\end{tikzcd}
 &\multicolumn{2}{c}{
\hspace{5pt}  \begin{aligned}
 \ \\[-5pt]
 %&= d(b)+c' (\widehat{c} d -f)\\
 d_\circ &\!=\! b e\overline{f} \! + \!ab\! -\!d\\%&\!=\! b^2 d\! +\! \overline{a} b\! -\! d\\
 d_{\nsquare} &\!= \! b e\overline{f} \! + \!ab\!+\! \overline{c} \overline{f} \! -\!d \\
 d_\bullet &\!=\! b e\overline{f} \!+\! \overline{c} \overline{f}\! -\!d
 \end{aligned}
 \hspace{-26pt}
}
& &\multicolumn{2}{c}{
\hspace{-15pt}
 \begin{aligned}
 \overline{c} \,&\!=\! c \!-\! be\!+\! df  \\
 \widehat{c} \,&\!=\! c \!+\! df\\ 
 \end{aligned}
 }
 
&  

\begin{tikzcd}[arrows={-stealth}, sep=\sepps]
  1  \arrow[r,  "a", \snsl] \arrow[dr, "d" near start, outer sep=\outsepd, \snsl]  \arrow[d, "f", \snls]  
  & 2  \arrow[d, swap, "\overline{b}", \snsl] \arrow[dl, "e_\bullet", near end, outer sep=\outsepe, \snls]
  \\
   4 
   & 3  \arrow[l, "\widehat{c}", swap, \snsl] 
\end{tikzcd}

 & \\ 
\\[-8pt]
\vmutation{2}   & & & & & & \vmutation{3} \\[-6pt]
\\
\begin{tikzcd}[arrows={-stealth}, sep=\sepps]
  1  \arrow[r,  "a_\bullet", \snls] \arrow[dr, "d_\circ" near start, outer sep=\outsepd, \snsl]  \arrow[d, "\overline{f}", \snsl]  
  & 2  \arrow[d, swap, "b", \snls] \arrow[dl, "e", near end, outer sep=\outsepe, \snls]
  \\
   4 
   & 3  \arrow[l, "\overline{c}", swap, \snls] 
\end{tikzcd}

& \shortmutation {4}
& 
\begin{tikzcd}[arrows={-stealth}, sep=\sepps]
  1  \arrow[r,  "a", \snls] \arrow[dr, "d_{\ssquare}" near start, outer sep=\outsepd, \snsl]  \arrow[d, "\overline{f}", \snls]  
  & 2  \arrow[d, swap, "b", \snls] \arrow[dl, "e", near end, outer sep=\outsepe, \snsl]
  \\
   4 
   & 3  \arrow[l, "\overline{c}", swap, \snsl] 
\end{tikzcd}
& \shortmutation 2 
&  
\begin{tikzcd}[arrows={-stealth}, sep=\sepps]
  1  \arrow[r,  "a", \snsl] \arrow[dr, "d_\bullet" near start, outer sep=\outsepd, \snsl]  \arrow[d, "\overline{f}", \snls]  
  & 2  \arrow[d, swap, "b", \snsl] \arrow[dl, "e", near end, outer sep=\outsepe, \snls]
  \\
   4 
   & 3  \arrow[l, "\widehat{c}", swap, \snsl] 
\end{tikzcd}
& \shortmutation 4
&  

\begin{tikzcd}[arrows={-stealth}, sep=\sepps]
  1  \arrow[r,  "a", \snsl] \arrow[dr, "d" near start, outer sep=\outsepd, \snls]  \arrow[d, "\overline{f}", \snsl]  
  & 2  \arrow[d, swap, "\overline{b}", \snls] \arrow[dl, "e", near end, outer sep=\outsepe, \snsl]
  \\
   4 
   & 3  \arrow[l, "\widehat{c}", swap, \snls] 
\end{tikzcd}

\end{array}
\hspace{-22pt}{\ }
\end{equation*}
\vspace{-.1in}
\caption{
Fully generic mutation cycle of length~12 involving 4-vertex quivers.  \linebreak[3]
The integer parameters $a,b,c,d,e,f$ can be chosen as follows. 
Begin by picking any ${b,d,e,f \geq 2}$. 
Choose $c\ge\max(2, 2+be-df)$
and then $a\ge \max(2, 2-d^2ef -c d e + bd+fe)$, 
ensuring that 
$\overline{a}=  d^2ef +c d e -bd-fe + a \geq 2$
and $\overline{c} = c-be+df \geq 2$.
%The last constraints can be assured by simple linear inequalities, such as $c\geq d \geq f \geq b \geq e$.
\linebreak[3]
Then~${\overline{b}=def+ce-b=be^2+\overline{c}e-b\ge e\ge 2}$
and similarly $\overline{f}=bde+\overline{c}d-f\ge d\ge 2$. 
The~edge weights are expressed in terms these parameters 
%(in terms of $a,\overline{a}, b, \overline{b}, c, \overline{c}, d, e,  f, \overline{f}$) 
as shown in the center of the diagram.
It is straightforward to check that these weights are positive, with the possible exception
of $b_\bullet$ and~$f_\bullet$ whose signs do not matter because we do not mutate 
at the vertices incident to an edge carrying one of those weights. 
%Note that $b_\bullet, d_{\nsquare}, e_{\nsquare}$ and $f_\bullet$ are positive by construction.
}
\label{fig:generic-12-cycle-fractured-cyclic}
\vspace{-.2in}
\end{figure}

\clearpage

\newpage

\begin{example}
\label{eg:big-horseshoe}
Figure~\ref{fig:generic-10-cycle-allmutations} shows a fully generic mutation cycle of length~10. 
Like in our construction from Theorem~\ref{thm:GeneralSemiCycles}, there are two quivers with a sink/source at $n=4$ on top, identical sequences of mutations on the left and right sides,
and mutations at $3,2,1$ along the bottom rim (though their order is reversed).
Unlike in our main construction, the bottom rim has no sink/source mutations.
Surprisingly, the weights still line up and forms another generic cycle.
This mutation cycle again satisfies properties (a) and (c) (but not (b)) from Example~\ref{eg:rosette}.
%with a dihedral symmetry of order~4. 
\end{example}

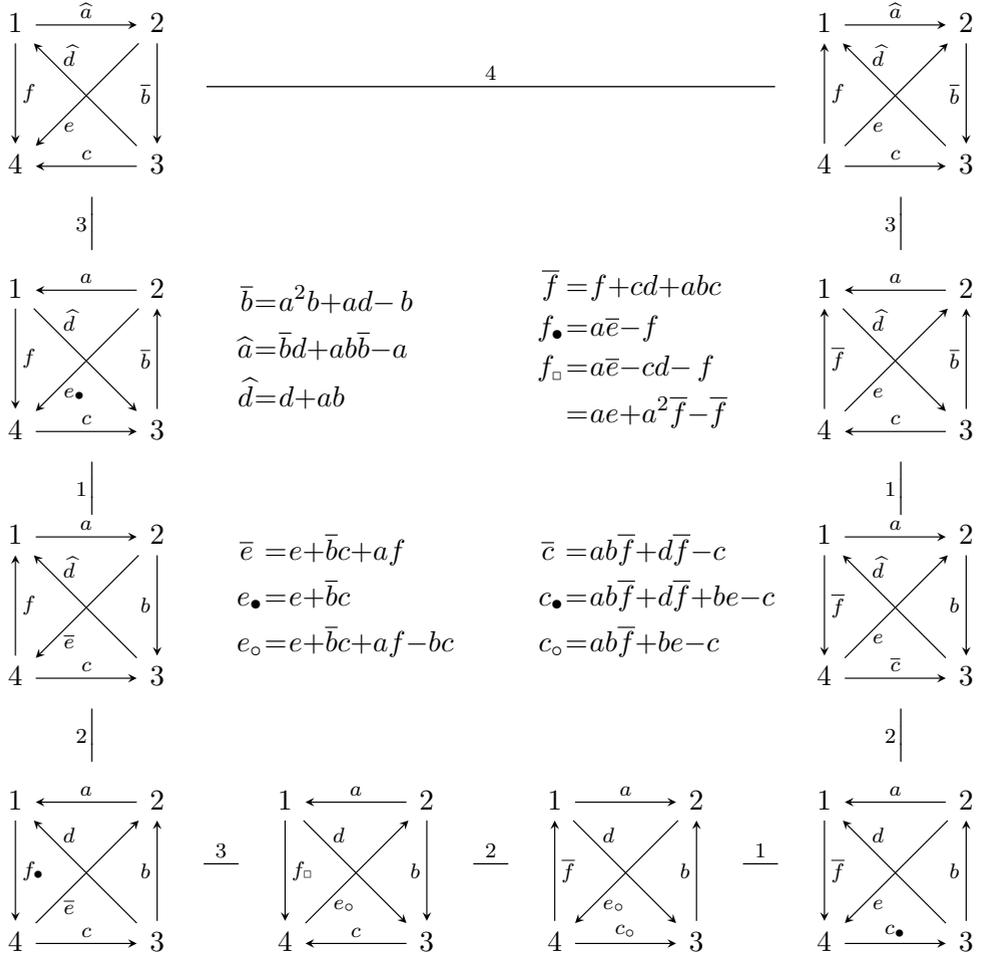
\begin{figure}[ht]
\begin{equation*}
\hspace{-10pt}
\begin{array}{ccccccccc}

\begin{tikzcd}[arrows={-stealth}, sep=\sepps]
  1  \arrow[r,  "\widehat{a}", -stealth] \arrow[dr, "\widehat{d}" near start, outer sep=\outsepd, stealth-]  \arrow[d, "f", -stealth]  
  & 2  \arrow[d, swap, "\overline{b}", \snls] \arrow[dl, "e", near end, outer sep=\outsepe, -stealth]
  \\
   4 
   & 3  \arrow[l, "c", swap, \snls] 
\end{tikzcd}
 & \multicolumn{5}{c}{ \stackrel{4}{\rule[.5ex]{19.5em}{0.5pt}}}
 &
\begin{tikzcd}[arrows={-stealth}, sep=\sepps]
  1  \arrow[r,  "\widehat{a}", -stealth] \arrow[dr, "\widehat{d}" near start, outer sep=\outsepd, stealth-]  \arrow[d, "f", \snsl]  
  & 2  \arrow[d, swap, "\overline{b}",\snls] \arrow[dl, "e", near end, outer sep=\outsepe, \snsl]
  \\
   4 
   & 3  \arrow[l, "c", swap, \snsl] 
\end{tikzcd}
\\
\\[-8pt]
\vmutation{3}   & & & & & & \vmutation{3} \\[-6pt]
\\

\begin{tikzcd}[arrows={-stealth}, sep=\sepps]
  1  \arrow[r,  "a", \snsl] \arrow[dr, "\widehat{d}" near start, outer sep=\outsepd, \snls]  \arrow[d, "f", -stealth]  
  & 2  \arrow[d, swap, "\overline{b}", stealth-] \arrow[dl, "e_\bullet", near end, outer sep=\outsepe, -stealth]
  \\
   4 
   & 3  \arrow[l, "c", swap, \snsl] 
\end{tikzcd}

 &\multicolumn{2}{c}{
 \begin{aligned}
  \overline{b} &\!=\! a^2b \!+\! ad\! -b\\
\widehat{a} &\!=\! \overline{b}d\! +\! ab\overline{b}\!-\!a\\%&= a^3b^2 + 2a^2bd - ab^2 + ad^2 -bd  -a\\
 \widehat{d} &\!=\! d\!+\!ab\\
 \end{aligned}
 }
& &\multicolumn{2}{c}{
\hspace{-20pt}
 \begin{aligned}
\overline{f} \,&\!=\! f\! +\! cd\! +\! abc \\
f_\bullet &\!=\! a \overline e \!-\! f\\
 f_{\nsquare} &\!=\! a \overline e \!-\! cd \!-f \\
 &\!=\!ae \!+\! a^2 \overline{f}\! -\! \overline{f}
 \end{aligned}
 }
 
&  
\begin{tikzcd}[arrows={-stealth}, sep=\sepps]
  1  \arrow[r,  "a", \snsl] \arrow[dr, "\widehat{d}" near start, outer sep=\outsepd, \snls]  \arrow[d, "\overline{f}", \snsl]  
  & 2  \arrow[d, swap, "\overline{b}", stealth-] \arrow[dl, "e", near end, outer sep=\outsepe, \snsl]
  \\
   4 
   & 3  \arrow[l, "c", swap, \snls] 
\end{tikzcd}
 & \\
 
 \\[-8pt]
\vmutation{1}   & 
\multicolumn{5}{c}{\begin{aligned}  \end{aligned} }& 
\vmutation{1}\\

\begin{tikzcd}[arrows={-stealth}, sep=\sepps]
  1  \arrow[r,  "a", \snls] \arrow[dr, "\widehat{d}" near start, outer sep=\outsepd, \snsl]  \arrow[d, "f", \snsl]  
  & 2  \arrow[d, swap, "b", \snls] \arrow[dl, "\overline e", near end, outer sep=\outsepe, \snls]
  \\
   4 
   & 3  \arrow[l, "c", swap, \snsl] 
\end{tikzcd}
 &\multicolumn{2}{c}{
 \begin{aligned}
 \overline e \,\,&\!=\! e \!+\! \overline{b} c \!+ \!af\\
 e_\bullet &\!=\!e\!+\!\overline{b} c\\
 e_\circ &\!=\! e \!+\! \overline{b} c \!+\! af \!-\!bc %&= e+a\overline{f}
 \end{aligned}
 \hspace{-15pt}
 }
& &\multicolumn{2}{c}{
 \begin{aligned}
 \overline c \,\,&\!= \!ab\overline{f}\! +\!d\overline{f}\!-\! c \\ %&=  ab\overline{f}+d\overline{f} - c\\
 c_\bullet &\!= \!ab\overline{f}\! +\!d\overline{f}\!+\!be \!-\! c \\
 c_\circ &\!= \!ab\overline{f}\!+\!be \!-\! c \\%&= ab\overline{f}+be-c \\
 \end{aligned}
 }
 
&  
\begin{tikzcd}[arrows={-stealth}, sep=\sepps]
  1  \arrow[r,  "a", \snls] \arrow[dr, "\widehat{d}" near start, outer sep=\outsepd, \snsl]  \arrow[d, "\overline{f}", \snls]  
  & 2  \arrow[d, swap, "b", \snls] \arrow[dl, "e", near end, outer sep=\outsepe, \snsl]
  \\
   4 
   & 3  \arrow[l, "\overline c",  swap, \snsl] 
\end{tikzcd}

 & \\ 
\\[-8pt]
\vmutation{2}   & & & & & & \vmutation{2} \\[-6pt]
\\
\begin{tikzcd}[arrows={-stealth}, sep=\sepps]
  1  \arrow[r,  "a", \snsl] \arrow[dr, "d" near start, outer sep=\outsepd, \snsl]  \arrow[d, "f_\bullet", \snls]  
  & 2  \arrow[d, swap, "b", \snsl] \arrow[dl, "\overline e", near end, outer sep=\outsepe, \snsl]
  \\
   4 
   & 3  \arrow[l, "c", swap, \snsl] 
\end{tikzcd}

& \shortmutation {3}
& 
\begin{tikzcd}[arrows={-stealth}, sep=\sepps]
  1  \arrow[r,  "a", \snsl] \arrow[dr, "d" near start, outer sep=\outsepd, \snls]  \arrow[d, "f_{\ssquare}", \snls]  
  & 2  \arrow[d, swap, "b", \snls] \arrow[dl, "e_\circ", near end, outer sep=\outsepe, \snsl]
  \\
   4 
   & 3  \arrow[l, "c", swap, \snls] 
\end{tikzcd}
& \shortmutation 2 
&  
\begin{tikzcd}[arrows={-stealth}, sep=\sepps]
  1  \arrow[r,  "a", \snls] \arrow[dr, "d" near start, outer sep=\outsepd, \snls]  \arrow[d, "\overline{f}", \snsl]  
  & 2  \arrow[d, swap, "b", \snsl] \arrow[dl, "e_\circ", near end, outer sep=\outsepe, \snls]
  \\
   4 
   & 3  \arrow[l, "c_\circ", swap, \snsl] 
\end{tikzcd} %\overline c+be-d \overline{f} = -c + o(f) d + o(f)ab +be -d o(f) = c + o(f)ab + be
& \shortmutation 1
&  

\begin{tikzcd}[arrows={-stealth}, sep=\sepps]
  1  \arrow[r,  "a", \snsl] \arrow[dr, "d" near start, outer sep=\outsepd, \snsl]  \arrow[d, "\overline{f}", \snls]  
  & 2  \arrow[d, swap, "b", \snsl] \arrow[dl, "e", near end, outer sep=\outsepe, \snls]
  \\
   4 
   & 3  \arrow[l, "c_\bullet",  swap, \snsl] 
\end{tikzcd}

\end{array}
\hspace{-22pt}{\ }
\end{equation*}
\vspace{-.1in}
\caption{
Fully generic mutation cycle of length~10 involving 4-vertex quivers.  \linebreak[3]
The integer parameters $a,b,c,d,e,f \geq 2$ can be chosen arbitrarily. 
For ease of notation, define the (positive) weights $\overline{b}=a^2b + ad -b$, $\overline{f}=f+cd+abc$ and $\overline{e}=e+\overline{b}c + af$.
The remaining weights can be expressed in terms these parameters, 
as shown in the center of the diagram.
}
\label{fig:generic-10-cycle-allmutations}
\vspace{-.2in}
\end{figure}

\newpage

\begin{example}
\label{eg:7-cycle}
Figure~\ref{fig:7-cycle} shows a 5-vertex quiver that lies on a fully generic mutation cycle of length~7. 
\end{example}

\begin{figure}[ht]
\vspace{-.1in}
\begin{equation*}
\begin{tikzcd}[arrows={-stealth}, sep=3.5em]
  1  \arrow[rr,  "a"] \arrow[drr, "e",  near start, outer sep=-1]  \arrow[ddr, "h", swap, pos=0.6, outer sep=-1.8]  \arrow[d, "j", swap, outer sep=-0.5]
  & & 2  \arrow[d, "b"] \arrow[ddl, "f+di", swap, near end, outer sep=-1.8] \arrow[dll, "i", swap, near end, outer sep=-1, stealth-]
  \\
   5 
  & & 3  \arrow[ll, "g", swap, stealth-] \arrow[dl, "c+gd", outer sep=-1.8]\\
  & 4 \arrow[ul, "d",  near end, outer sep=-1.8]
\end{tikzcd}
\end{equation*}
\vspace{-.1in}
\caption{
A 5-vertex quiver lying on a fully generic mutation cycle of length~7.
Apply mutations at $1, 5, 2, 3, 5, 4, 5$, in this order, to return to the original quiver. 
}
\label{fig:7-cycle}
\vspace{-.2in}
\end{figure}
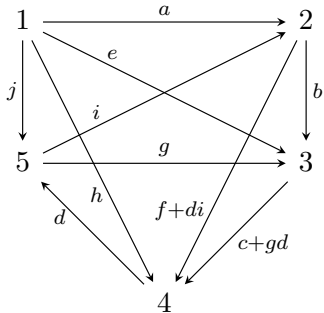

\begin{remark}
The constructions in Examples~\ref{eg:6-cycle-vortices} and~\ref{eg:7-cycle} are special cases (for $n=4$ and $n=5,$ respectively) of a general construction of a mutation cycle of $n$-vertex quivers. We next sketch this construction.

Start with $n-1$ integers $q_{in} \ge 2$, for $i=1,\dots,n-1$.
Pick integers $a$ and~$b$ satisfying $1 \!<\! a \!<\! b \!< \!n\!-\!1$.
Create an acyclic quiver $\tilde Q$ on the vertices $1,\dots,n\!-\!1$, with standard orientation and very large weights, 
say, bigger than any product $q_{in} q_{jn}$. 
Construct the quiver $Q$ as follows.
Set $Q|_{[1,n-1]} = \tilde Q$.
Set $n\stackrel{q_{in}}{\longrightarrow} i$ if $a<i\le b$; set $i\stackrel{q_{in}}{\longrightarrow} n$ otherwise. 
% $n \points u$ if $u \in [a+1, b]$ and $u \points n$ otherwise, and weights $|b_{in}(Q)| = q_{in}$. 
Then quiver~$Q$ lies on the following mutation cycle: 
\[
\T{n (n-1) \cdots (b+1) n b \cdots (a+1) n a \cdots 1}{Q} = Q.
\]
In this construction, every mutation $\mu[i]$ with $i \neq n$ is a sink/source mutation, whereas every mutation at $n$ is not.
\end{remark}

\providecommand{\bysame}{\leavevmode\hbox to3em{\hrulefill}\thinspace}
\providecommand{\MR}{\relax\ifhmode\unskip\space\fi MR }
% \MRhref is called by the amsart/book/proc definition of \MR.
\providecommand{\MRhref}[2]{%
  \href{http://www.ams.org/mathscinet-getitem?mr=#1}{#2}
}
\providecommand{\href}[2]{#2}

%\bibliographystyle{abbrv}
%\bibliography{lmc1.bib}
\end{document}